\documentclass[12pt]{article}

\usepackage{amsmath, amsthm, amssymb}
\usepackage{colortbl}
\usepackage{graphicx}
\usepackage{ifthen, calc}
\usepackage{url}

\numberwithin{equation}{section}
\newtheorem{theo}{Theorem}[section]
\newtheorem{conj}[theo]{Conjecture}
\newtheorem{prob}[theo]{Problem}
\newtheorem{corol}[theo]{Corollary}
\newtheorem{lemm}[theo]{Lemma}
\newtheorem{prop}[theo]{Proposition}

\theoremstyle{definition} 
\newtheorem{remark}[theo]{Remark}

\definecolor{gray1}{gray}{0.8}
\definecolor{gray2}{gray}{0.7}
\definecolor{gray3}{gray}{0.5}

\def\c{\cellcolor{gray1}}
\def\d{\cellcolor{gray3}}

\newcommand{\autotops}[1]{\mathrm{Atp}(#1)} 
\newcommand{\automorphs}[1]{\mathrm{Aut}(#1)} 

\def\id{{\rm id}}
\def\autm{\mathrm{Aut}}
\def\autt{\mathrm{Atp}}
\renewcommand{\geq}{\geqslant}
\renewcommand{\leq}{\leqslant}
\renewcommand{\ge}{\geqslant}
\renewcommand{\le}{\leqslant}
\providecommand{\nequiv}{\not\equiv}
\providecommand{\nin}{\notin}
\def\tdot{{\cdot}} 

\def\C{{\mathcal{C}}}

\def\I{{\mathcal{I}}}

\def\O{{\cal O}}
\def\S{{\cal S}}

\def\tref#1{Theorem~\ref{#1}}
\def\lref#1{Lemma~\ref{#1}}
\def\cref#1{Corollary~\ref{#1}}
\def\fref#1{Figure~\ref{#1}}
\def\eref#1{(\ref{#1})}

\def\half{\tfrac12}

\DeclareMathOperator{\lcm}{lcm}

\DeclareMathOperator{\per}{\text{\scshape{per}}}

\def\narrowcols{ \addtolength{\arraycolsep}{-1.5pt}\begin{footnotesize} }
\def\normalcols{ \end{footnotesize}\addtolength{\arraycolsep}{1.5pt} }
\def\verynarrowcols{ \addtolength{\arraycolsep}{-3pt}\renewcommand{\arraystretch}{0.8}\begin{scriptsize} }
\def\verynormalcols{ \end{scriptsize}\addtolength{\arraycolsep}{3pt}\renewcommand{\arraystretch}{1.25}}

\def\xx{\cdot\!}        
\def\xa{\star\!}            
\def\xb{\circ\!}        
\def\xc{\bullet\!}      
\def\xi{\!\!\!\infty\!\!\!\!}       
\newcounter{xxcounter}  
\def\xxx#1{ 
    \setcounter{xxcounter}{0}
    \whiledo{\value{xxcounter}<#1}{
        \stepcounter{xxcounter}
        \xx
        \ifthenelse{\value{xxcounter}<#1}{&}{}
    }
}

\def\tone{\star}
\def\ttwo{\circ}
\def\tthree{\bullet}
\def\tinfty{\infty}

\newcommand{\offset}[2]{\mathcal{O}_{#1,#2}}

\def\tspacer{{\vrule height 2.75ex width 0ex depth0ex}}
\def\bspacer{{\vrule height 0ex width 0ex depth1.2ex}}

\begin{document}

\title{Cycle structure of autotopisms of\\
quasigroups and Latin squares\thanks{Research supported by ARC grants DP0662946 and DP1093320.  P.~Vojt\v{e}chovsk\'y thanks Monash University for hospitality and financial support during his sabbatical stay.}}

\author{Douglas S.\ Stones$^{1,2}$, Petr Vojt\v{e}chovsk\'{y}$^3$ and Ian M.\ Wanless$^1$ \\
\small $^1$ School of Mathematical Sciences and\\[-0.75ex]
\small $^2$ Clayton School of Information Technology\\[-0.75ex]
\small Monash University\\[-0.75ex]
\small VIC 3800 Australia\\
\small {\tt the\_empty\_element@yahoo.com} (D.\,S.\,Stones) and {\tt ian.wanless@monash.edu} (I.\,M.\,Wanless)\\
\small $^3$ Department of Mathematics\\[-0.75ex]
\small University of Denver\\[-0.75ex]
\small 2360 S Gaylord St, Denver, CO 80208, USA\\
\small {\tt  petr@math.du.edu}
}
\date{}

\maketitle

\begin{abstract}
An autotopism of a Latin square is a triple $(\alpha,\beta,\gamma)$ of
permutations such that the Latin square is mapped to itself by
permuting its rows by $\alpha$, columns by $\beta$, and symbols by
$\gamma$. Let $\autotops{n}$ be the set of all autotopisms of Latin
squares of order $n$.  Whether a triple $(\alpha,\beta,\gamma)$ of
permutations belongs to $\autotops{n}$ depends only on the cycle
structures of $\alpha$, $\beta$ and $\gamma$.  We establish a number
of necessary conditions for $(\alpha,\beta,\gamma)$ to be in
$\autotops{n}$, and use them to determine $\autotops{n}$ for $n\le17$.
For general $n$ we determine if $(\alpha,\alpha,\alpha)\in\autotops{n}$
(that is, if $\alpha$ is an automorphism of some quasigroup of order
$n$), provided that either $\alpha$ has at most three cycles other
than fixed points or that the non-fixed points of $\alpha$ are in
cycles of the same length.
\end{abstract}

\begin{center}
\end{center}


\section{Introduction}\label{SEIntro}

A \emph{Latin square} of order $n$ is an $n \times n$ array $L=L(i,j)$
of $n$ symbols such that the symbols in every row and in every column
are distinct.  We will usually index the rows and columns of $L$ by
elements of $[n]=\{1,2,\ldots,n\}$ and take the symbol set to be
$[n]$. A \emph{quasigroup} $Q$ is a nonempty set with one binary
operation such that for every $a$, $b\in Q$ there is a unique $x\in Q$
and a unique $y\in Q$ satisfying $ax=b$, $ya=b$. Since multiplication
tables of finite quasigroups are precisely Latin squares, all results
obtained in this paper for Latin squares can be interpreted in the
setting of finite quasigroups.

Let $\I_n=S_n \times S_n \times S_n$, where $S_n$ is the symmetric
group acting on $[n]$.  Then $\I_n$ acts on the set of Latin squares
indexed by $[n]$ as follows: For each $\theta=(\alpha,\beta,\gamma)
\in \I_n$ we define $\theta(L)$ to be the Latin square formed from $L$
by permuting the rows according to $\alpha$, permuting the columns
according to $\beta$, and permuting the symbols according to
$\gamma$. More precisely, $\theta(L)=L'$ is the Latin square defined by
\begin{equation}\label{EQFormula}
L'(i,j)=\gamma\big(L(\alpha^{-1}(i),\beta^{-1}(j))\big).
\end{equation}

The elements $\theta$ of $\I_n$ are called \emph{isotopisms}, and the
Latin squares $L$ and $\theta(L)$ are said to be \emph{isotopic}. If
$\theta\in \I_n$ is of the form $\theta=(\alpha,\alpha,\alpha)$, then
$\alpha$ is an \emph{isomorphism}.

If $\theta\in \I_n$ satisfies $\theta(L)=L$, then $\theta$ is an
\emph{autotopism} of $L$.  By \eqref{EQFormula}, $\theta\in\I_n$ is an
autotopism of $L$ if and only if
\begin{equation}\label{EQAutotopism}
\gamma\big(L(i,j)\big)=L\big(\alpha(i),\beta(j)\big)
\end{equation}
for all $i,j \in [n]$.  We use $\id$ to denote the identity
permutation, and we call $(\id,\id,\id) \in \I_n$ the \emph{trivial}
autotopism. The group of all autotopisms of $L$ will be denoted by
$\autt(L)$.

We will be particularly interested in the case where
$(\alpha,\alpha,\alpha)\in \autt(L)$, when we call $\alpha$ an
\emph{automorphism} of $L$.  The group of all automorphisms of $L$
will be denoted by $\autm(L)$.

Autotopisms and automorphisms are natural classes of symmetries of
Latin squares and quasigroups, motivating the question
\begin{equation}
    \textit{``Which isotopisms are autotopisms of Latin squares?''}\label{Eq:Q}\tag{Q}
\end{equation}
and also its specialization ``Which isomorphisms are automorphisms of
Latin squares?''. In this paper we give a partial answer to these
questions.

\subsection{Overview}

Let $n \geq 1$. For $\theta\in \I_n$, let $\Delta(\theta)$ be the
number of Latin squares $L$ of order $n$ for which
$\theta\in\autt(L)$. Let
$\autotops{n}=\{\theta\in\I_n:\Delta(\theta)>0\}$ and
$\automorphs{n}=\{\alpha \in S_n:(\alpha,\alpha,\alpha)\in\autotops{n}\}$.
Hence \eqref{Eq:Q} can be rephrased as ``What is $\autotops{n}$?''

We show in Section~\ref{Sc:Equivalence} that the value of
$\Delta(\theta)$ depends only on the cycle structures of the
components $\alpha$, $\beta$ and $\gamma$ of
$\theta=(\alpha,\beta,\gamma)$.  In Section~\ref{ScAutt} we establish
several necessary conditions for an isotopism to be an
autotopism. These conditions go a long way toward describing
$\autotops{n}$ for all $n\le 17$, with only a few \emph{ad hoc}
computations needed.

To demonstrate that $\theta \in \autotops{n}$, it is usually necessary
to give an explicit construction of a Latin square $L$ with
$\theta\in\autt(L)$.  In Section~\ref{Sc:Contours} we present two
visual tools, called block diagrams and contours, that allow us to
describe the required Latin squares without impenetrable
notation. Additionally, more specialized means of constructing
contours of Latin squares are given in 
Section~\ref{Sc:AdditionalBlockPatterns}.

We call a cycle of a permutation \emph{nontrivial} if it has length
greater than one. In \tref{THEqualCycLen} we characterize all
automorphisms whose nontrivial cycles are of the same length. In
\tref{THTwoCycles} we characterize all automorphisms that contain
precisely two nontrivial cycles. In \tref{THThreeCyc} we characterize
all automorphisms that contain precisely three nontrivial cycles.

Our computational results are summarized in 
Section~\ref{Sc:Computational} and Appendix A, where we determine
$\autotops{n}$ for $12\le n\le 17$.  Combined with the previous
results of Falc\'on \cite{Falcon2009} (which we verify), this
identifies $\autotops{n}$ for all $n\le 17$.

Open problems and conjectures are presented in the final section.

\subsection{Motivation and literature review}

For a notion as pervasive as symmetry it is infeasible to survey all
the relevant results within the vast literature on Latin squares and
quasigroups. The problem is exacerbated by the fact that results may
be proved about symmetries of other objects that implicitly imply
results about symmetries of Latin squares. For instance, Colbourn and
Rosa \cite[Section 7.4]{CoRo} asked which permutations are
automorphisms of Steiner triple systems, hence addressing our question
\eqref{Eq:Q} for Steiner quasigroups (i.e.\ Latin squares that are
idempotent and totally symmetric). To give another example,
autotopisms of Latin squares inherited from one-factorizations of
graphs were studied in \cite{WI}.

There are quite a few recent results on Latin squares where
understanding of autotopisms has been critical. In
\cite{BrowningStonesWanless,McKayWanless2005,StonesWanless2009c,SW,StonesWanless2009b},
autotopisms were used to establish congruences that the number of
Latin squares of given order must satisfy (see also
\cite{Sto10}). Similar ideas were used by Drisko~\cite{drisko2} to
prove a special case of the Alon-Tarsi Conjecture (see also
\cite{StonesWanless2009d}). It was shown in \cite{McKayWanless2005}
that the autotopism group of almost all Latin squares is trivial,
thereby revealing that the asymptotic ratio of the number of Latin
squares to the number of isotopism classes of Latin squares of order
$n$ is $(n!)^3$. Imposing a large autotopism group can make it
feasible to look for Latin squares with desirable properties in search
spaces that would otherwise be too large \cite{cyclatom}, and also to
show that certain properties hold in a Latin square
\cite{ninf,cyclatom}.  Ganfornina
\cite{Ganfornina2006b,Ganfornina2006} suggested using Latin squares
that admit certain autotopisms for secret sharing schemes.  During the
course of resolving the existence question for near-automorphisms,
\cite{CavenaghStones2010a} classified when a Latin rectangle completes
to a Latin square that admits an autotopism with a trivial first
component.

There are also many results concerning autotopisms of quasigroups and
\emph{loops}, that is, quasigroups with a neutral element. In loop
theory, autotopisms have been useful in the study of specific
varieties of loops, particularly those in which the defining
identities can be expressed autotopically. For example, a loop is
\emph{Moufang} \cite{Moufang} if it satisfies the identity $(xy)(zx) =
x((yz)x)$. This is equivalent to the assertion that for each $x$ the
triple $(L_x, R_x, L_x R_x)$ is an autotopism, where $L_x(y) = xy$ and
$R_x(y) = yx$ for all $y$. Thus Moufang loops can be studied by
considering these and other autotopisms \cite[Chapter V]{Bruck}, a
point of view that culminates in the theory of groups with triality
\cite{Doro}. Other varieties of loops in which the defining identities
have autotopic characterizations include conjugacy closed loops
\cite{GR}, extra loops \cite{KK}, and Buchsteiner loops \cite{CDK}. A
new, systematic look at the basic theory of loops defined in this way
can be found in \cite{DJ}. Automorphisms have not played quite the
same role in loop theory as they do in group theory, primarily due to
the fact that inner mappings (stabilizers of the neutral element in
the permutation group generated by all $L_x$ and $R_x$) are generally
not automorphisms. Worth mentioning is the study of loops with
transitive automorphism groups \cite{Dr,Drisko1997b}, and of loops in
which every inner mapping is actually an automorphism
\cite{BP,JKV,JKV2}.

We conclude the literature review with a summary of some results
specifically concerning \eqref{Eq:Q}. The first result was obtained by
Euler~\cite{Euler1782} in 1782.  He answered \eqref{Eq:Q} when
$\alpha$, $\beta$ and $\gamma$ are all $n$-cycles.  This was
generalized by Wanless \cite{Wanless2004} in 2004, who answered
\eqref{Eq:Q} for isomorphisms containing a single nontrivial cycle.
Bryant, Buchanan and Wanless \cite{BryantBuchananWanless2009} later
extended the results in \cite{Wanless2004} to include quasigroups with
additional properties, such as semisymmetry or idempotency.

In 1968, Sade \cite{Sade1968} answered \eqref{Eq:Q} for an isotopism
$\theta$ with a trivial component; a condition that was rediscovered
in \cite{Ganfornina2006,Laywine1981}.  Actually, these papers proved
only the necessity of the condition, but the sufficiency is easy to
show, cf.\ Theorem \ref{THOneFixedPerm}.

In 2007, McKay, Meynert and Myrvold~\cite{McKayMeynertMyrvold2007}
derived an important necessary condition for $\theta\in\I_n$ to belong
to $\autotops{n}$ (see Theorem~\ref{TMMM}) in the course of
enumerating quasigroups and loops up to isomorphism for orders $\le10$.
Recently, Hulpke, Kaski and \"Osterg{\aa}rd
\cite{HulpkeKaskiOstergard} gave a detailed account of the symmetries
of Latin squares of order $11$.

McKay, Meynert and Myrvold \cite{McKayMeynertMyrvold2007} also
identified graphs whose automorphism groups are isomorphic to
$\autt(L)$ and $\autm(L)$. This enabled them to use the
graph isomorphism software \verb!nauty! \cite{nauty} to efficiently
calculate the autotopism groups of Latin squares.  A different
procedure for finding the automorphism group of $L$, based on
equational invariants, was implemented in the \verb!LOOPS!
\cite{LOOPS,NagyVojtechovsky2007} package for \verb!GAP! \cite{GAP}.

Also in 2007, Falc\'{o}n and Mart\'{i}n-Morales
\cite{FalconMartinMorales2007} gave the nonzero values of
$\Delta(\theta)$ for all $\theta \in \I_n$ with $n\le 7$.  Later,
Falc\'on \cite{Falcon2009} determined $\autotops{n}$ for all $n\le
11$, and he gave several results of general nature.


Kerby and Smith~\cite{KerbySmith2009,KerbySmith2010} considered
\eqref{Eq:Q} for isomorphisms from an algebraic point of view.  The
divisors of $\Delta(\theta)$, for isomorphisms $\theta$, were
discussed in \cite{Stones2009d} and were used to determine the parity
of the number of quasigroups for small orders.

\section{Cycle structure}\label{Sc:Equivalence}

We begin by identifying an equivalence relation on isotopisms that
preserves the value of $\Delta$.  Given a Latin square $L=L(i,j)$ of
order $n$ we can construct a set of $n^2$ ordered
triples \[O(L)=\big\{\big(i,j,L(i,j)\big):i,j \in [n]\big\}\] called
the \emph{orthogonal array representation} of $L$. We will call the
elements of $O(L)$ \emph{entries} of $L$.  Conversely, if $O$ is a set
of $n^2$ triples $\big(i,j,L(i,j)\big)\in [n] \times [n] \times [n]$
such that distinct triples differ in at least two coordinates, then
$O$ gives rise to a Latin square of order $n$.

The symmetric group $S_3$ has a natural action on $O(L)$. If
$\lambda\in S_3$, then $O(L)^\lambda$ is obtained from $O(L)$ by
permuting the coordinates of all entries of $O(L)$ by $\lambda$. The
Latin square $L^\lambda$ induced by $O(L)^\lambda$ is called a
\emph{parastrophe} of $L$.

The group $S_3$ also acts on $\I_n = S_n\times S_n\times S_n$ by
permuting the coordinates of $\I_n$. Given $\theta\in\I_n$ and
$\lambda\in S_3$, we denote the resulting isotopism by
$\theta^\lambda$.

\begin{lemm}\label{LMAutConj}
Let $\lambda\in S_3$, let $\theta,\varphi\in \I_n$, and let $L$ be a Latin square of order $n$. Then
\begin{enumerate}
  \item[(i)] $\theta\in\autt(L)$ if and only if $\varphi\theta\varphi^{-1}\in\autt(\varphi(L))$,
  \item[(ii)] $\theta\in\autt(L)$ if and only if $\theta^\lambda\in\autt(L^\lambda)$.
\end{enumerate}
\end{lemm}

\begin{proof}
To prove the first claim, observe that the following conditions are
equivalent: $\varphi\theta\varphi^{-1}\in\autt(\varphi(L))$,
$\varphi\theta\varphi^{-1}\varphi(L)=\varphi(L)$,
$\varphi\theta(L)=\varphi(L)$, $\theta(L)=L$, $\theta \in
\autt(L)$. The second claim is even more straightforward.
\end{proof}

Every $\alpha\in S_n$ decomposes into a product of disjoint cycles,
where we consider fixed points to be cycles of length $1$.  We say
$\alpha$ has the \emph{cycle structure} $c_1^{\lambda_1}\tdot
c_2^{\lambda_2}\cdots c_m^{\lambda_m}$ if $c_1>c_2>\cdots>c_m\ge 1$
and there are $\lambda_i$ cycles of length $c_i$ in the unique cycle
decomposition of $\alpha$.  Hence $c_1\lambda_1+c_2\lambda_2+\cdots
+c_m\lambda_m=n$. If $\lambda_i=1$, we usually write $c_i$ instead of
$c_i^1$ in the cycle structure.

We define the \emph{cycle structure} of
$\theta=(\alpha,\beta,\gamma)\in\I_n$ to be the ordered triple of
cycle structures of $\alpha$, $\beta$ and $\gamma$.

Since two permutations in $S_n$ are conjugate if and only if they have
the same cycle structure \cite[p.~25]{Cameron1999}, we deduce from
Lemma~\ref{LMAutConj} that the value of $\Delta(\theta)$ depends only
on the (unordered) cycle structure of $\theta$. In particular, if
$\alpha$, $\beta$ and $\gamma$ have the same cycle structures then
$\Delta((\alpha,\beta,\gamma))=\Delta((\alpha,\alpha,\alpha))$.

We say that a permutation in $S_n$ is \emph{canonical} if (i) it is
written as a product of disjoint cycles, including $1$-cycles
corresponding to fixed points, (ii) the cycles are ordered according
to their length, starting with the longest cycles, (iii) each
$c$-cycle is of the form $(i,i+1,\ldots,i+c-1)$, with $i$ being
referred to as the \emph{leading symbol} of the cycle, and (iv) if a
cycle with leading symbol $i$ is followed by a cycle with leading
symbol $j$, then $i<j$.  The purpose of this definition is to
establish a unique way of writing a representative permutation with a
given cycle structure.  For instance, if we consider permutations with
the cycle structure $3 \tdot 2 \tdot 1^2$, then $(123)(45)(6)(7)\in
S_7$ is canonical, whereas $(357)(41)(2)(6)$, $(132)(45)(6)(7)$ and
$(123)(45)(7)(6)$ are not.

By Lemma~\ref{LMAutConj}, while studying the value of
$\Delta((\alpha,\beta,\gamma))$, we may assume that the permutations
$\alpha$, $\beta$ and $\gamma$ are canonical.

Finally, we deduce from Lemma~\ref{LMAutConj} that for
$\varphi\in\I_n$ the autotopism groups $\autt(L)$ and
$\autt(\varphi(L))$ are conjugate in $\I_n$, and thus they are
isomorphic.  We will therefore study isotopisms modulo the equivalence
induced by conjugation and parastrophy.

\section{Conditions on isotopisms to be autotopisms}\label{ScAutt}

In this section we review and extend some important conditions for
membership in $\autotops{n}$.

\subsection{Previously known conditions}

The following two lemmas are easy to observe.  A submatrix of a Latin
square $L$ is called a \emph{subsquare} of $L$ if it is a Latin
square.

\begin{lemm}\label{LMMaxSize}
Let $L$ be a Latin square of order $n$ that contains a subsquare of
order $m$. Then either $m=n$ or $m\le \lfloor \half n\rfloor$.
\end{lemm}

The \emph{direct product} of two Latin squares $L$ and $L'$ of orders
$n$ and $n'$, respectively, is a Latin square $K=L\times L'$ of order
$nn'$ defined by $K((i,i'),(j,j'))=(L(i,j),L'(i',j'))$.  The
\emph{direct product} of two permutations $\alpha$ of $[n]$ and
$\alpha'$ of $[n']$ is defined by
$(\alpha\times\alpha')(i,i')=(\alpha(i),\alpha'(i'))$.

\begin{lemm}\label{LMDirectProd}
Let $L$ and $L'$ be Latin squares such that
$\theta=(\alpha,\beta,\gamma) \in \autt(L)$ and
$\theta'=(\alpha',\beta',\gamma') \in \autt(L')$.  Then $\theta \times
\theta' \in \autt(L \times L')$, where $\theta \times \theta'=(\alpha
\times \alpha',\beta \times \beta',\gamma \times \gamma')$.
\end{lemm}

We will only need Lemma~\ref{LMDirectProd} in the special case when
$\theta'=(\id,\id,\id)$ is the trivial autotopism.  If the order of
$L'$ is $n'$, then the cycle structure of
$(\alpha,\beta,\gamma)\times(\id,\id,\id)$ is the cycle structure of
$(\alpha,\beta,\gamma)$ with the multiplicity of each cycle multiplied
by $n'$.

Note that it is possible to have
$\theta\times(\id,\id,\id)\in\autotops{nn'}$ while
$\theta\notin\autotops{n}$.  For example, in
Theorem~\ref{THEqualCycLen} we will find that
$(\alpha,\alpha,\alpha)\nin\autotops{n}$ if $n$ is even and $\alpha$
is an $n$-cycle, but if $n'=2$ (for example) then
$(\alpha,\alpha,\alpha)\times(\id,\id,\id)\in\autotops{nn'}$.

We begin our list of conditions for membership in $\autotops{n}$ with
the aforementioned theorem of McKay, Meynert and
Myrvold~\cite{McKayMeynertMyrvold2007}.

\begin{theo}\label{TMMM}
Let $L$ be a Latin square of order $n$ and let $(\alpha,\beta,\gamma)$
be a nontrivial autotopism of $L$. Then either
\begin{enumerate}\addtolength{\itemsep}{-0.4\baselineskip}
  \item[(a)] $\alpha$, $\beta$ and $\gamma$ have the same cycle
    structure with at least $1$ and at most $\left\lfloor \half n
    \right\rfloor$ fixed points, or
  \item[(b)] one of $\alpha$, $\beta$ or $\gamma$ has at least $1$
    fixed point and the other two permutations have the same cycle
    structure with no fixed points, or
  \item[(c)] $\alpha$, $\beta$ and $\gamma$ have no fixed points.
\end{enumerate}
\end{theo}

Any nontrivial autotopism $\theta=(\alpha,\beta,\gamma)\in\I_n$ of a
Latin square must have at least two nontrivial
components. Lemma~\ref{LMAutConj} implies that the following theorem
characterizes all nontrivial autotopisms with one trivial component.

\begin{theo}[Autotopisms with a trivial component]\label{THOneFixedPerm}
Let $\theta=(\alpha,\beta,\id)\in\I_n$. Then $\theta\in\autotops{n}$
if and only if both $\alpha$ and $\beta$ consist of $n/d$ cycles of
length $d$, for some divisor $d$ of $n$.
\end{theo}

\begin{proof}
The necessity was proved by Sade \cite{Sade1968} and rediscovered in
\cite{Ganfornina2006,Laywine1981}; a proof also appears in
\cite{Falcon2009}.  Let $L=L(i,j)$ be a Latin square with $\theta
\in \autt(L)$.  If $i$ belongs to a $c$-cycle of $\alpha$ and $j$
belongs to a $d$-cycle in $\beta$, then the entry $(i,j,L(i,j))$ maps
to $(i,\beta^c(j),L(i,j))$ by $\theta^c$.  Hence $d$ divides $c$. A
similar argument shows $c$ divides $d$, so $c=d$. Thus $\alpha$ and
$\beta$ must contain only $d$-cycles.

To prove the converse, let $L=L(i,j)$ be the Latin square on the
symbol set $[n]$ that satisfies $L(i,j) \equiv i+j\pmod n$.  Now
observe that $((12\cdots n),(12\cdots n)^{-1},\id)^{n/d} \in \autt(L)$
and consists of $n/d$ cycles of length $d$.
\end{proof}


Note that the proof of Theorem~\ref{THOneFixedPerm} implies that the
full spectrum of possible cycle structures of autotopisms with a
trivial component is displayed by Cayley tables of cyclic groups.

\begin{remark}
Let $\theta=(\alpha,\alpha,\id)$. The evaluation of $\Delta(\theta)$
was studied by Laywine \cite{Laywine1981} and Ganfornina
\cite{Ganfornina2006}. Unfortunately, \cite{Laywine1981} contained
some errors (later corrected in \cite{LM85}).  Ganfornina
\cite{Ganfornina2006} gave an explicit formula for $\Delta(\theta)$ if
$\alpha$ consists of $n/d$ cycles of length $d\le 3$.
\end{remark}

\subsection{New conditions}

We begin with the following necessary condition for membership in
$\autotops{n}$.

\begin{lemm}\label{LMLcmAB}
Let $\theta=(\alpha,\beta,\gamma) \in \I_n$ be an autotopism of a
Latin square $L$. If $i$ belongs to an $a$-cycle of $\alpha$
and $j$ belongs to a $b$-cycle of $\beta$, then $L(i,j)$ belongs to a
$c$-cycle of $\gamma$, where
$\lcm(a,b)=\lcm(b,c)=\lcm(a,c)=\lcm(a,b,c)$.
\end{lemm}


\begin{proof}
Since $\alpha^{\lcm(a,b)}$ fixes $i$ and $\beta^{\lcm(a,b)}$ fixes $j$
the entry $(i,j,L(i,j))$ must be a fixed point of $\theta^{\lcm(a,b)}$.
Hence $c$ divides $\lcm(a,b)$, and
$\lcm(a,b)=\lcm(a,b,c)$ follows. The result follows since
$\theta^{\lambda} \in \autt(L^{\lambda})$ for all $\lambda \in S_3$ by
Lemma~\ref{LMAutConj}.
\end{proof}

Lemma~\ref{LMLcmAB} precludes many isotopisms from being
autotopisms. For example, there is no autotopism with cycle structure
$(3^2\tdot2^3,3^4,2^6)$.

For our next lemma, we will need to introduce the notion of a strongly
$\lcm$-closed set.  Let $\mathbb N=\{1,2,\ldots\}$.  A nonempty subset
$\Lambda$ of $\mathbb N$ is \emph{strongly lcm-closed} if for every
$a,b\in \mathbb N$ we have $\lcm(a,b)\in\Lambda$ if and only if
$a\in\Lambda$ and $b\in\Lambda$.
Strongly lcm-closed sets are precisely the ideals in the divisibility lattice
on the set of positive integers. If $\Lambda$ is a finite strongly $\lcm$-closed set, then $\Lambda$ is
the set of divisors of $\max\Lambda$.  However, we wish to also
consider infinite strongly $\lcm$-closed sets.

For $i \geq 1$, let $p_i$ be the $i$-th prime.  For any map
$f:\mathbb N\rightarrow \mathbb N \cup\{0,\infty\}$,
the set
\begin{displaymath}
    \Lambda(f) = \Big\{\prod_{i\in I}p_i^{k_i}:\;I
\text{ is a finite subset of $\mathbb N$ and each }
k_i \in \mathbb{N} \cup \{0\} \text{ where } k_i\le f(i)\Big\}
\end{displaymath}
is strongly $\lcm$-closed. Moreover, it is not hard to see that every
strongly $\lcm$-closed set can be obtained in this way for some suitable $f$.

We will now show how strongly $\lcm$-closed sets can be used to
identify subsquares within Latin squares that admit autotopisms.  Let
$L=L(i,j)$ be a Latin square of order $n$ with
$\theta=(\alpha,\beta,\gamma) \in \autt(L)$.  Suppose $M$ is a
subsquare of $L$ formed by the rows whose indices belong to
$R\subseteq [n]$ and columns whose indices belong to $C \subseteq [n]$.
Let $S=\{L(i,j):i \in R \text{ and } j \in C\}$, so $|R|=|C|=|S|$.  We
will say $M$ is \emph{closed} under the action of $\theta$ (more
formally, under the action of the subgroup generated by $\theta$) if
$R$, $C$ and $S$ are closed under the action of $\alpha$, $\beta$ and
$\gamma$, respectively. If $M$ is closed under the action of $\theta$,
then we can form the autotopism $\theta_M$ of $M$, by restricting the
domains of $\alpha$, $\beta$ and $\gamma$ to $R$, $C$ and $S$,
respectively.

Given $(\alpha,\beta,\gamma) \in \I_n$ and a strongly $\lcm$-closed
set $\Lambda$, define
\begin{align*}
  & R_\Lambda=\{i \in [n]:\; i \text{ belongs to an } a\text{-cycle in } \alpha \text{ and } a \in \Lambda\},\\
  & C_\Lambda=\{i \in [n]:\; i \text{ belongs to a } b\text{-cycle in } \beta \text{ and } b \in \Lambda\},\\
  & S_\Lambda=\{i \in [n]:\; i \text{ belongs to a } c\text{-cycle in } \gamma \text{ and } c \in \Lambda\}.
\end{align*}
For $X \subseteq [n]$ let $\overline{X}=[n]\setminus X$.


\begin{theo}\label{THStrongLCM}
Suppose $L$ is a Latin square of order $n$.
Let $\theta=(\alpha,\beta,\gamma) \in \autt(L)$ and let $\Lambda$ be a
strongly lcm-closed set.  If at least two of $R_\Lambda$,
$\,C_\Lambda$ and $S_\Lambda$ are nonempty, then
$|R_\Lambda|=|C_\Lambda|=|S_\Lambda|$ and $L$ contains a subsquare $M$
on the rows $R_\Lambda$, columns $C_\Lambda$ and symbols $S_\Lambda$.
Moreover, $M$ admits the autotopism $\theta_M$.

In addition, if $|R_\Lambda|=|C_\Lambda|=|S_\Lambda|=\half n$ then $L$
has four subsquares, each with autotopisms induced by $\theta$. The
subsquares are on the rows, columns and symbols
$(R_\Lambda,C_\Lambda,S_\Lambda)$,
$(R_\Lambda,\overline{C_\Lambda},\overline{S_\Lambda})$,
$(\overline{R_\Lambda},C_\Lambda,\overline{S_\Lambda})$ and
$(\overline{R_\Lambda},\overline{C_\Lambda},S_\Lambda)$.
\end{theo}

\begin{proof}
Up to parastrophy, we may assume
$|R_\Lambda|\ge|C_\Lambda|\ge|S_\Lambda|$. Let $M$ be the (necessarily
nonempty) submatrix induced by rows $R_\Lambda$ and columns
$C_\Lambda$.

Pick an entry $(i,j,L(i,j))$ in $M$.  Then $i$ belongs to an $a$-cycle
of $\alpha$ for some $a\in \Lambda$ and $j$ belongs to a $b$-cycle of
$\beta$ for some $b\in \Lambda$.  Suppose $L(i,j)$ belongs to a
$c$-cycle of $\gamma$.  By Lemma~\ref{LMLcmAB},
$\lcm(a,c)=\lcm(a,b)$. Since $\Lambda$ is a strongly $\lcm$-closed
set, we deduce that $\lcm(a,c)\in \Lambda$ and
$c\in\Lambda$. Therefore, every symbol in $M$ belongs to $S_\Lambda$
and so $|S_\Lambda|\ge|R_\Lambda|$.  Hence
$|R_\Lambda|=|C_\Lambda|=|S_\Lambda|$ and $M$ is a subsquare of $L$.

To prove that $\theta_M$ is indeed an autotopism of $M$, we merely
note that $R_\Lambda$, $\,C_\Lambda$ and $S_\Lambda$ are closed under
the action of $\langle\alpha\rangle$, $\,\langle\beta\rangle$ and
$\langle\gamma\rangle$, respectively.

The remainder of the theorem follows since any Latin square containing
a subsquare of exactly half its order is composed of four disjoint
subsquares of that order.
\end{proof}

For instance, by considering the strongly $\lcm$-closed set
$\Lambda=\{1,2\}$, Theorem~\ref{THStrongLCM} implies that there is no
autotopism $(\alpha,\beta,\gamma)$ such that $\alpha$ has cycle
structure $4\tdot 2^2$ and $\beta$ has cycle structure $2^4$. A square
with such an autotopism would contain a $4 \times 8$ ``subsquare'',
which is impossible.





The next necessary condition for membership in $\autotops{n}$ checks
whether we can find enough room in a Latin square $L$ with
$\theta\in\autt(L)$ to place all $n$ copies of each symbol
so that Lemma~\ref{LMLcmAB} is satisfied.

The \emph{permanent} of an $n \times n$ square matrix $X=X(i,j)$ is
defined as
\[\per(X)=\sum_{\sigma \in S_n} \prod_{i \in [n]}X\big(i,\sigma(i)\big).\]
In particular, if $X$ is an $n\times n$
$\,(0,1)$-matrix, then $\per(X)$ counts the number of $n\times n$
permutation matrices that embed into $X$.  We direct the reader to
\cite{Minc1978} for more information on permanents.

Let $\theta=(\alpha,\beta,\gamma) \in \I_n$ and suppose $s \in [n]$
belongs to a $c$-cycle in $\gamma$.  We define $X_s=X_s(i,j)$ to be
the $(0,1)$-matrix with $X_s(i,j)=1$ if $i$ belongs to an $a$-cycle of
$\alpha$ and $j$ belongs to a $b$-cycle of $\beta$ such that
$\lcm(a,b)=\lcm(b,c)=\lcm(a,c)=\lcm(a,b,c)$, and $X_s(i,j)=0$
otherwise.  Informally, the zeroes in $X_s$ mark the positions where
Lemma~\ref{LMLcmAB} says a symbol $s$ cannot be placed in a Latin
square $L$ of order $n$ with $\theta \in \autt(L)$.

If $\theta\in\autt(L)$ for some Latin square $L$ of order $n$, then
the copies of the symbol $s$ in $L$ identify a permutation matrix
embedded in $X_s$. Hence we have just proved the following result.

\begin{lemm}\label{LMLCMPerm}
Let $\theta\in\I_n$.  If $\theta\in\autotops{n}$ then $\per(X_s)>0$
for all $s \in [n]$.
\end{lemm}

To illustrate, let $n=6+3k+2\ell$ for some integers $k\ge1$ and
$\ell\ge4$, and suppose that $\theta=(\alpha,\beta,\gamma)\in \I_{n}$
is such that $\alpha$, $\beta$ and $\gamma$ have cycle structure
$6\tdot3^k\tdot2^\ell$. Consider the $(0,1)$-matrix $X_s$ for some $s$
that belongs to a $3$-cycle in $\gamma$.  Note that $X_s(i,j)=0$ when
$i$ belongs to a $2$-cycle in $\alpha$ and $j$ belongs to either a
$2$-cycle or $3$-cycle in $\beta$.  In particular, $X_s$ has a
$(2\ell)\times(n-6)$ zero submatrix, and $2\ell+n-6>n$ so
$\per(X_s)=0$ (by the Frobenius-K\"onig Theorem \cite[p.31]{Minc1978}).
Hence, \lref{LMLCMPerm} implies that $\theta\notin\autotops{n}$.


We will establish additional conditions on the cycle structure of
autotopisms in Section~\ref{Sc:AutmSameLength}, but first we need to
develop some visual tools.

\section{Block diagrams and contours}\label{Sc:Contours}

In this section we introduce two visual tools for constructing Latin
squares with a prescribed automorphism: block diagrams and
contours. We start by looking at orbits of cells of Latin squares
under the action induced by an autotopism.

Suppose that $\theta=(\alpha,\beta,\gamma)$ is an autotopism of a
Latin square $L$, where $\alpha$ and $\beta$ are canonical.  If $i$ is
a leading symbol in a cycle of $\alpha$, $j$ is a leading symbol in a
cycle of $\beta$, and $L(i,j)=k$, the orbit of the entry $(i,j,k)\in
O(L)$ under the action of $\theta$ will look like
\begin{displaymath}
    \begin{array}{c|cccc|}
    &j&j+1&j+2&\cdots\\
    \hline
    i&k& & & \\
    i+1& &\gamma(k)& & \\
    i+2& & &\gamma^2(k)& \\
    \vdots& & & &\ddots\\
    \hline
    \end{array}\ .
\end{displaymath}
The set of cells $\{(\alpha^r(i),\beta^r(j)):r \geq 0\}$ is called a
\emph{cell orbit}.  Of course, the ``shape'' of the orbit depends on
the lengths of the cycles of $\alpha$ and $\beta$ containing $i$ and
$j$, respectively. For instance, if $i$ is in a $2$-cycle of $\alpha$
and $j$ is in a $6$-cycle of $\beta$, the orbit of $(i,j,k)$ looks
like
\begin{displaymath}
    \begin{array}{c|cccccc|}
     &j&j+1&j+2&j+3&j+4&j+5\\
    \hline
    i&k& &\gamma^2(k)& &\gamma^4(k)& \\
    i+1& &\gamma(k)& &\gamma^3(k)& &\gamma^5(k)\\
    \hline
    \end{array}\ .
\end{displaymath}
This forces $\gamma$ to behave in a certain way, as described in
Lemma~\ref{LMLcmAB}.

Note the special shape of the orbit when either $i$ is a fixed point
of $\alpha$ or $j$ is a fixed point of $\beta$.

Although it is possible to continue the discussion for general
autotopisms, we will mostly deal only with automorphisms.

\subsection{Block diagrams}\label{Ss:BlockDiagrams}

As we are going to see in Section~\ref{Ss:Contours}, constructing a
Latin square $L$ with a prescribed automorphism $\alpha$ can be
reduced to a careful placement of leading symbols of $\alpha$ into
$L$. We would therefore like to know how the leading symbols of
$\alpha$ are distributed in $L$.

For the rest of this paper, let $\alpha_1,\alpha_2,\ldots,\alpha_m$ be
the nontrivial cycles of $\alpha \in S_n$ with lengths
$d_1\ge d_2\ge\cdots \ge d_m$, respectively. For $1\le i\le m$, let
$t_i=1+\sum_{j<i}d_j$ be the non-fixed leading symbols of $\alpha$. Let
$\alpha_\infty$ be the set of all fixed points of $\alpha$, and let
$d_\infty=|\alpha_\infty|$.  Let $[m]^*=[m] \cup \{\infty\}$.  For any
$i,j \in [m]^*$, let $M_{ij}$ be the \emph{block} of $L$ formed by the
rows whose indices are in the cycle $\alpha_i$ and columns whose indices
are in the cycle $\alpha_j$.

Blocks will be our basic ``unit of construction''.  In
Lemma~\ref{LMContour}, we will give conditions that can be used to
diagnose whether or not a given collection of blocks determines a
Latin square with a specific automorphism.  Previous efforts to
construct or enumerate Latin squares with a given autotopism
(e.g.\ \cite{FalconMartinMorales2007,McKayMeynertMyrvold2007,Stones2009d})
have tended to build the squares block by block, although the terminology
and notation has varied.

We write $\alpha_k:f_k$ in a block $M_{ij}$ if every symbol in
$\alpha_k$ (equivalently, the leading symbol of $\alpha_k$) appears in
$M_{ij}$ precisely $f_k = f_k(i,j)$ times. If $f_k=0$, we usually omit
$\alpha_k:f_k$. The result is the \emph{block diagram} of $L$
according to the cycles of $\alpha$.

Although there are situations when the block diagram depends only on
$\alpha$ (that is, every Latin square $L$ with $\alpha\in\autm(L)$ has
the same block diagram), generally this is not the case. For example,
here are two Latin squares with distinct block diagrams for the
automorphism $(123)(456)$:
\begin{displaymath}
\narrowcols
\begin{array}{|ccc|ccc|}
    \hline
    1&3&2&4&6&5\\
    3&2&1&6&5&4\\
    2&1&3&5&4&6\\
    \hline
    4&6&5&1&3&2\\
    6&5&4&3&2&1\\
    5&4&6&2&1&3\\
    \hline
\end{array}
\quad\rightarrow\quad
\begin{array}{c|c|c|}
    &\alpha_1&\alpha_2\\
    \hline
    \alpha_1
        &\begin{array}{l} \alpha_1:3 \\ \alpha_2:0\end{array}
        &\begin{array}{l} \alpha_1:0 \\ \alpha_2:3\end{array} \\
    \hline
    \alpha_2
        &\begin{array}{l} \alpha_1:0 \\ \alpha_2:3\end{array}
        &\begin{array}{l} \alpha_1:3 \\ \alpha_2:0\end{array} \\
    \hline
\end{array}
\quad\text{and}\quad
\begin{array}{|ccc|ccc|}
    \hline
    4&3&2&1&6&5\\
    3&5&1&6&2&4\\
    2&1&6&5&4&3\\
    \hline
    1&6&5&4&3&2\\
    6&2&4&3&5&1\\
    5&4&3&2&1&6\\
    \hline
\end{array}
\quad\rightarrow\quad
\begin{array}{c|c|c|}
    &\alpha_1&\alpha_2\\
    \hline
    \alpha_1
        &\begin{array}{l} \alpha_1:2 \\ \alpha_2:1\end{array}
        &\begin{array}{l} \alpha_1:1 \\ \alpha_2:2\end{array} \\
    \hline
    \alpha_2
        &\begin{array}{l} \alpha_1:1 \\ \alpha_2:2\end{array}
        &\begin{array}{l} \alpha_1:2 \\ \alpha_2:1\end{array} \\
    \hline
\end{array}
\normalcols
\end{displaymath}

While constructing a block diagram, it is helpful to keep in mind that
in every block $M_{ij}$ we must have
$d_1f_1(i,j)+d_2f_2(i,j)+\cdots+d_mf_m(i,j)+d_\infty f_\infty(i,j) =
d_id_j$. In addition, for any $i \in [m]^*$ the $d_i \times n$
submatrix $\bigcup_{j \in [m]^*}M_{ij}$ contains exactly $d_i$ copies
of each symbol in $[n]$.  Hence
\begin{equation}\label{EQBlock1}
\sum_{j \in [m]^*} f_k(i,j)=d_i
\end{equation}
for any $i,k \in [m]^*$.  Similarly, $\sum_{i \in [m]^*} f_k(i,j)=d_j$
for any $j,k \in [m]^*$.

To further illustrate the concept of a block diagram, let us determine
the block diagram of any Latin square $L$ with $\alpha\in\autm(L)$,
where $\alpha$ has cycle structure $d_1\tdot d_2\tdot 1^{d_\infty}$
with $d_1>d_2>1$, which is depicted in Figure \ref{FILStruct}. The
$M_{\infty\infty}$ block contains only fixed points by Lemma
\ref{LMLcmAB} and hence it contains each fixed point $d_\infty$
times. Similarly, for $1\le i\le 2$, the blocks $M_{i\infty}$ and
$M_{\infty i}$ contain only symbols of $\alpha_i$, each $d_\infty$
times. Since $d_1>d_2$, the blocks $M_{12}$ and $M_{21}$ must contain
only symbols of $\alpha_1$, each $d_2$ times. The structure of the
remaining blocks $M_{11}$ and $M_{22}$ follows from
\eqref{EQBlock1}. So, in this case, the block diagram is determined by
$\alpha$.  In Theorem~\ref{THTwoCycles} we will give necessary and
sufficient conditions for the existence of a Latin square $L$ having
this $\alpha$ as an automorphism.

\begin{figure}[htb]
\centering
\narrowcols
\begin{displaymath}
\begin{array}{l|l|l|l|}
    & \ \alpha_1 & \ \alpha_2 & \ \alpha_\infty \\
    \hline
        \alpha_1 &
        \begin{array}{l} \alpha_1:d_1-d_2-d_\infty\\ \alpha_2:d_1 \\ \alpha_\infty:d_1 \end{array} &
        \begin{array}{l} \alpha_1:d_2 \\ \\ \phantom{x}\end{array} &
        \begin{array}{l} \alpha_1:d_\infty \\ \\ \phantom{x}\end{array} \\
    \hline
        \alpha_2 &
        \begin{array}{l} \alpha_1:d_2 \\ \phantom{x} \end{array} &
        \begin{array}{l} \alpha_2:d_2-d_\infty \\ \alpha_\infty:d_2 \end{array} &
        \begin{array}{l} \alpha_2:d_\infty \\ \phantom{x} \end{array} \\
    \hline
        \alpha_\infty &
        \begin{array}{l} \alpha_1:d_\infty \end{array} &
        \begin{array}{l} \alpha_2:d_\infty \end{array} &
        \begin{array}{l} \alpha_\infty:d_\infty \end{array} \\
    \hline
\end{array}
\end{displaymath}
\normalcols
\caption{\label{FILStruct}Block diagram of $L$ with $d_1>d_2$
the only nontrivial cycle lengths.}
\end{figure}



\subsection{Contours}\label{Ss:Contours}

Consider $\alpha = (123)(4)(5)$, and observe that
$\alpha\in\automorphs{5}$, since it is an automorphism of the Latin
square
\begin{equation}
    \begin{array}{|ccc|c|c|}
        \hline
        1&4&5&2&3\\
        5&2&4&3&1\\
        4&5&3&1&2\\
        \hline
        2&3&1&4&5\\
        \hline
        3&1&2&5&4\\
        \hline
    \end{array}\ .\label{SquareFull}
\end{equation}
Note that the placement of the horizontal (or vertical) lines in
\eqref{SquareFull} determines $\alpha$, since $\alpha$ is
canonical. Moreover, the Latin square \eqref{SquareFull} and $\alpha$
can be reconstructed from the knowledge of
\begin{displaymath}
        \begin{array}{|ccc|c|c|}
        \hline
        1&4&5&\cdot&\cdot\\
        \cdot&\cdot&\cdot&\cdot&1\\
        \cdot&\cdot&\cdot&1&\cdot\\
        \hline
        \cdot&\cdot&1&4&5\\
        \hline
        \cdot&1&\cdot&5&4\\
        \hline
    \end{array}\ .
\end{displaymath}

We call such a diagram a \emph{contour} $\C$ of $\alpha$, provided it
only contains leading symbols, each cell orbit contains precisely one
leading symbol, and the diagram determines a Latin square $L$ with
$\alpha\in\autm(L)$.

The following lemma describes what needs to be checked in a purported
contour to ensure that it is indeed a contour.

\begin{lemm}\label{LMContour}
Consider a canonical $\alpha\in S_n$. Let $T$ be the set of all
leading symbols of $\alpha$. Then a partial matrix $\C$ of order $n$,
divided into blocks according to the cycle structure of $\alpha$, is a
contour of $\alpha$ if and only if all of the following conditions are
satisfied:
\begin{enumerate}

\item[(a)] $\C$ is a partial Latin square (that is, any symbol occurs
  at most once in each row and each column) and every symbol of $\C$
  is from $T$.

\item[(b)] Let $B$ be an $a\times b$ block of $\C$.  Then $B$ contains
  precisely $\gcd(a,b)$ symbols from $T$, all in distinct cell orbits
  of $B$.

\item[(c)] If $z\in T$ is a symbol in an $a\times b$ block, then $z$
  belongs to a $c$-cycle of $\alpha$ such that
  $\lcm(a,b)=\lcm(b,c)=\lcm(a,c)=\lcm(a,b,c)$.

\item[(d)] Let $z\in T$ be in a $c$-cycle of $\alpha$.  If two distinct rows $i$ and $i'$ both contain a copy of $z$ and belong to the same
  $a$-cycle of $\alpha$, then $i\nequiv i'\pmod{\gcd(a,c)}$.

\item[(e)] Let $z\in T$ be in a $c$-cycle of $\alpha$. If two distinct columns $j$ and $j'$ both contain a copy of $z$ and belong to the same
  $b$-cycle of $\alpha$, then $j\nequiv j'\pmod{\gcd(b,c)}$.
\end{enumerate}
\end{lemm}

\begin{proof}
A detailed proof of this result can be found in
\cite[pp.~111--112]{Stones2009}.  Here we offer some guiding
observations.  Remaining details are direct consequences of our
definitions or are otherwise routine to complete.  Let $(i,j,k)$ be an
entry in an $a \times b$ block and suppose $k$ belongs to a $c$-cycle
of~$\alpha$.  There are $ab/\lcm(a,b) = \gcd(a,b)$ distinct cell
orbits in the block containing the entry $(i,j,k)$. This explains
condition (b).  The cell orbit through $(i,j,k)$ contains $\lcm(a,b)$
entries. The symbol $k$ will therefore appear $\lcm(a,b)/c$ times in
the cell orbit, spaced evenly in rows with increments of
$a/(\lcm(a,b)/c) = ac/\lcm(a,b) = ac/\lcm(a,c) = \gcd(a,c)$ rows, and
similarly spaced evenly in columns with increments of $\gcd(b,c)$
columns. Hence (d) and (e) are required to prevent repeated symbols
within rows and columns respectively.  Condition (c) is necessary by
Lemma~\ref{LMLcmAB}.
\end{proof}

As well as contours we will talk of {\em partial contours}, which we
define to be the restriction of a contour to a block or union of
blocks. A partial contour should be such that it does not result in
any violation of the conditions in Lemma~\ref{LMContour} within the
blocks on which it is defined.

We will now develop techniques that allow us to check most of the
conditions of Lemma~\ref{LMContour} visually. For instance, conditions
(d) and (e) can be fulfilled by placing identical leading symbols into
consecutive rows and consecutive columns in a given block. The
building blocks of contours introduced in
Section~\ref{Ss:BuildingBlocks} will help with conditions (a) and (b),
while the block diagrams of Section~\ref{Ss:BlockDiagrams} are
designed to cope with condition (c).

\subsection{Basic block patterns}\label{Ss:BuildingBlocks}

A contour of $\alpha$ decomposes into partial contours according to
the blocks determined by the cycles of $\alpha$.  We collect here
several constructions for partial contours. The action of $\alpha$ on
the cells of a block $M_{ij}$ partitions it into $g=\gcd(d_i,d_j)$
disjoint cell orbits. Each of these cell orbits should contain
exactly one leading symbol. The action of $\alpha$ can then be used
to fill the remaining cells in each cell orbit, thereby completing
the block.

The patterns in this section are intended to be an informal guide
only. They represent configurations that occur many times in our
specific constructions in later sections. Those sections should be
consulted for concrete examples. Our aim here is just to present
the intuition behind what we do later.

\paragraph{The odd pattern}
If $g$ is odd, then by the \emph{odd pattern} we refer to a partial
contour where the cells on the main antidiagonal of some $g \times g$
contiguous submatrix of $M_{ij}$ are filled, such as in
\begin{equation}\label{EQOddPattern}
\narrowcols
\begin{array}{|ccccc|}
    \hline
    \cdot&\cdot&\c\cdot&\cdot&k\\
    \cdot&\cdot&\cdot&\c k&\cdot\\
    \cdot&\cdot&k&\cdot&\c\cdot\\
    \c\cdot&k&\cdot&\cdot&\cdot\\
    k&\c\cdot&\cdot&\cdot&\cdot\\
    \hline
\end{array}
\quad\text{ and }\quad
\begin{array}{|ccccccc|}
    \hline
    \cdot&\cdot&\c\cdot&\cdot&\cdot&\cdot&k\\
    \cdot&\cdot&\cdot&\c\cdot&\cdot&k&\cdot\\
    \cdot&\cdot&\cdot&\cdot&\c k&\cdot&\cdot\\
    \cdot&\cdot&\cdot&k&\cdot&\c\cdot&\cdot\\
    \cdot&\cdot&k&\cdot&\cdot&\cdot&\c\cdot\\
    \c\cdot&k&\cdot&\cdot&\cdot&\cdot&\cdot\\
    k&\c\cdot&\cdot&\cdot&\cdot&\cdot&\cdot\\
    \hline
\end{array}\ .
\normalcols
\end{equation}
The shaded cells highlight a cell orbit and $k$ is a leading symbol
from an $\lcm(d_i,d_j)$-cycle.  Any pattern obtained from the odd
pattern by cyclically permuting its rows or columns can also be
thought of as an odd pattern.

\paragraph{The even pattern}
Now consider a $g \times g$ contiguous submatrix when $g$ is even.  We
can see (cf.\ Theorem~\ref{TDiagCyclic}) that a partial contour
containing a unique cell from each row, column, and cell orbit cannot
be realized.  Consequently, in most of the constructions in this
paper, the cycles of even length will be more difficult to handle than
cycles of odd length.  However, the \emph{even pattern}
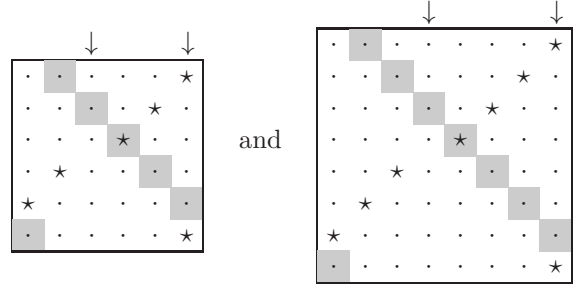
\begin{displaymath}
\narrowcols
\begin{array}{|cccccc|}
    \multicolumn{1}{r}{}& & \downarrow & & & \multicolumn{1}{r}{\downarrow} \\
    \hline
    \cdot&\c\cdot&\cdot&\cdot&\cdot&\xa\\
    \cdot&\cdot&\c\cdot&\cdot&\xa&\cdot\\
    \cdot&\cdot&\cdot&\c\xa&\cdot&\cdot\\
    \cdot&\xa&\cdot&\cdot&\c\cdot&\cdot\\
    \xa&\cdot&\cdot&\cdot&\cdot&\c\cdot\\
    \c\cdot&\cdot&\cdot&\cdot&\cdot&\xa\\
    \hline
\end{array}
\quad\text{ and }\quad
\begin{array}{|cccccccc|}
    \multicolumn{1}{r}{}& & &  \downarrow & & & & \multicolumn{1}{r}{\downarrow} \\
    \hline
    \cdot&\c\cdot&\cdot&\cdot&\cdot&\cdot&\cdot&\xa\\
    \cdot&\cdot&\c\cdot&\cdot&\cdot&\cdot&\xa&\cdot\\
    \cdot&\cdot&\cdot&\c\cdot&\cdot&\xa&\cdot&\cdot\\
    \cdot&\cdot&\cdot&\cdot&\c\xa&\cdot&\cdot&\cdot\\
    \cdot&\cdot&\xa&\cdot&\cdot&\c\cdot&\cdot&\cdot\\
    \cdot&\xa&\cdot&\cdot&\cdot&\cdot&\c\cdot&\cdot\\
    \xa&\cdot&\cdot&\cdot&\cdot&\cdot&\cdot&\c\cdot\\
    \c\cdot&\cdot&\cdot&\cdot&\cdot&\cdot&\cdot&\xa\\
    \hline
\end{array}\ ,
\normalcols
\end{displaymath}
comes close, with one cell in each row, each cell orbit, and all but
two of the columns (marked by arrows). In general, an even pattern is
formed by starting with the main antidiagonal and (cyclically)
shifting half of the occupied cells by one position.

If $k$ is a leading symbol from an $\lcm(d_i,d_j)$-cycle and $\ell$ is
a leading symbol from a $c$-cycle (where $c$ divides $\lcm(d_i,d_j)$,
as required by Lemma~\ref{LMLcmAB}), we can find a partial contour for
$M_{ij}$, such as in
\begin{equation}\label{EQEvenPattern}
\narrowcols
\begin{array}{|cccccc|}
    \hline
    \cdot&\c\cdot&\cdot&\cdot&\cdot&k\\
    \cdot&\cdot&\c\cdot&\cdot&k&\cdot\\
    \cdot&\cdot&\cdot&\c k&\cdot&\cdot\\
    \cdot&k&\cdot&\cdot&\c \cdot&\cdot\\
    \ell&\cdot&\cdot&\cdot&\cdot&\c\cdot\\
    \c\cdot&\cdot&\cdot&\cdot&\cdot&\ell\\
    \hline
\end{array}\ ,
\normalcols
\end{equation}
for example. Conditions (a)--(c) of Lemma~\ref{LMContour} are then
satisfied within the block $M_{ij}$, and we also observe that
conditions (d) and (e) of Lemma~\ref{LMContour} are satisfied for the
symbol $k$ within $M_{ij}$. Provided there are not too many copies of
$\ell$, we can observe that conditions (d) and (e) of
Lemma~\ref{LMContour} are also satisfied for the symbol $\ell$ within
$M_{ij}$, since the symbols are located in consecutive rows and
columns.  In this case, \eqref{EQEvenPattern} is a partial contour.

\paragraph{The staircase pattern}
Again with $g$ even, if $k$ and $\ell$ are leading symbols of two
$\lcm(d_i,d_j)$-cycles in $\alpha$, then we can use a partial contour that
embeds in a $g \times g$ contiguous submatrix, such as
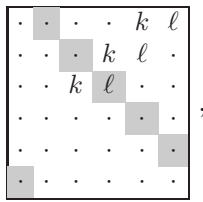
\begin{displaymath}
\narrowcols
\begin{array}{|cccccc|}
    \hline
    \cdot&\c\cdot&\cdot&\cdot&k&\ell\\
    \cdot&\cdot&\c\cdot&k&\ell&\cdot\\
    \cdot&\cdot&k&\c\ell&\cdot&\cdot\\
    \cdot&\cdot&\cdot&\cdot&\c\cdot&\cdot\\
    \cdot&\cdot&\cdot&\cdot&\cdot&\c\cdot\\
    \c\cdot&\cdot&\cdot&\cdot&\cdot&\cdot\\
    \hline
\end{array}\ ,
\normalcols
\end{displaymath}
which we call the \emph{staircase pattern}.

\paragraph{Rectangular blocks}
The above block patterns can also be used to fill rectangular
blocks. For instance, if $a$ is an odd divisor of $b$, we can place
leading symbols $k$ from a $b$-cycle into an $a\times b$ block by
filling an $a\times a$ submatrix with an odd pattern, as in
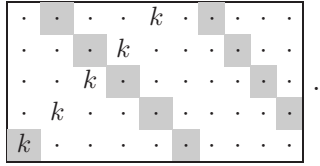
\begin{displaymath}
\narrowcols
\begin{array}{|cccccccccc|}
    \hline
    \cdot&\c\cdot&\cdot&\cdot&k&\cdot&\c\cdot&\cdot&\cdot&\cdot\\
    \cdot&\cdot&\c\cdot&k&\cdot&\cdot&\cdot&\c\cdot&\cdot&\cdot\\
    \cdot&\cdot&k&\c\cdot&\cdot&\cdot&\cdot&\cdot&\c\cdot&\cdot\\
    \cdot&k&\cdot&\cdot&\c\cdot&\cdot&\cdot&\cdot&\cdot&\c\cdot\\
    \c k&\cdot&\cdot&\cdot&\cdot&\c\cdot&\cdot&\cdot&\cdot&\cdot\\
    \hline
\end{array}\ .
\normalcols
\end{displaymath}

Additional block patterns will be given in 
Section~\ref{Sc:AdditionalBlockPatterns}.

\subsection{Notation in contours}

To save space and improve legibility, we will from now on reserve
certain symbols to denote leading symbols of cycles in contours.

As usual, we assume $\alpha$ has cycles
$\alpha_1,\alpha_2,\alpha_3,\ldots$ of lengths $d_1\ge d_2\ge d_3\ge\cdots$,
respectively.  We use
\begin{align*}
 & \tone \text{ to denote the leading symbol of } \alpha_1,\\
 & \ttwo \text{ to denote the leading symbol of } \alpha_2,\\
 & \tthree \text{ to denote the leading symbol of } \alpha_3, \text{ and}\\
 & \tinfty \text{ to denote a fixed point.}
\end{align*}
Hence the numerical equivalents of $\tone$, $\,\ttwo$, and $\tthree$
are $t_1=1$, $\,t_2=d_1+1$, and $t_3=d_1+d_2+1$, respectively.

Since we also wish to construct contours in proofs without regard to
the parity of $d_1$, we define offsets $\offset{1}{i}$
for $1\le i\le d_1$ by
\begin{displaymath}
    \offset{1}{i}=
    \begin{cases}
        1&\text{if $d_1$ is even and $\half d_1<i<d_1$},\\
        1-d_1&\text{if $d_1$ is even and $i=d_1$},\\
        0&\text{otherwise}.
    \end{cases}
\end{displaymath}
The cells in the block $M_{11}$ with coordinates
$(i,d_1+1-i-\mathcal O_{1,i})$ for $1\le i\le d_1$
then define an odd or even pattern in
accordance with the parity of $d_1$.

\section{Automorphisms with all nontrivial cycles of the same length}\label{Sc:AutmSameLength}

In this section we characterize all automorphisms of Latin squares
whose nontrivial cycles have the same length.  We begin with the
following theorem from \cite{Wanless2004}.

\begin{theo}\label{TDiagCyclic}
If $\alpha \in S_n$ has the cycle structure $d\tdot 1^{n-d}$, where
$d>1$, then $\alpha\in\automorphs{n}$ if and only if either $d=n$ is
odd or $\lceil \half n \rceil \le d<n$.
\end{theo}

We will now prove a generalization of Theorem~\ref{TDiagCyclic}, when
$\alpha$ consists of an arbitrary number of cycles of the same length.
We remark that in 1782 Euler \cite{Euler1782} proved a result
equivalent to the special case of Theorem~\ref{TDiagCyclic} with $d=n$
and no fixed points.

\begin{theo}[Automorphisms with all nontrivial cycles of the same length]\label{THEqualCycLen}
Suppose that $\alpha\in S_n$ has precisely $m$ nontrivial cycles, each
cycle having the same length $d$. If $\alpha$ has at least one fixed
point, then $\alpha\in\automorphs{n}$ if and only if $n\le 2md$. If
$\alpha$ has no fixed points, then $\alpha\in\automorphs{n}$ if and
only if $d$ is odd or $m$ is even.
\end{theo}

\begin{proof}
We begin with Case I, where we assume $\alpha$ has no fixed points.

\textit{Case I(a)}: $d$ is odd ($m$ may be even or odd).
Theorem~\ref{TDiagCyclic} and Lemma~\ref{LMDirectProd} imply that
there exists a Latin square $L$ with $\alpha \in \autm(L) \subseteq
\automorphs{n}$.

\textit{Case I(b)}: both $d$ and $m$ are even. It is sufficient to
show that $\alpha\in\automorphs{n}$ when $m=2$, since the rest of this
case then follows from Lemma~\ref{LMDirectProd}.  When $m=2$, we
identify a contour $\C$ comprising of four identical staircase
patterns, suitably shifted.  We define $M_{11}$ and $M_{12}$
respectively by
\begin{align*}
    \mathcal C(i,d/2+1-i) = t_1&\quad\text{ for } 1\leq i\leq d/2,\\
    \mathcal C(i,d/2+2-i) = t_2&\quad\text{ for } 1\leq i\leq d/2,
\\[1.5ex]
    \mathcal C(d/2+1,d+1) = t_2&,\\
    \mathcal C(d/2+i,2d+1-i) = t_1&\quad\text{ for } 1\leq i\leq d/2,\\
    \mathcal C(d/2+1+i,2d+1-i) = t_2&\quad\text{ for } 1\leq i\leq d/2-1,
\end{align*}
then let $M_{11}=M_{22}$ and $M_{12}=M_{21}$.  For example, the
contours for $d\in\{2,4\}$ are

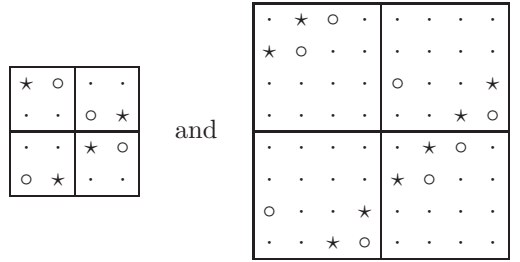
\narrowcols
\begin{displaymath}
\begin{array}{|cc|cc|}
    \hline
    \xa&\xb&\xxx{2}\\
    \xxx{2}&\xb&\xa\\
    \hline
    \xxx{2}&\xa&\xb\\
    \xb&\xa&\xxx{2}\\
    \hline
\end{array}
\quad\text{ and }\quad
\begin{array}{|cccc|cccc|}
    \hline
    \xx&\xa&\xb&\xxx{5}\\
    \xa&\xb&\xxx{6}\\
    \xxx{4}&\xb&\xxx{2}&\xa\\
    \xxx{6}&\xa&\xb\\
    \hline
    \xxx{5}&\xa&\xb&\xx\\
    \xxx{4}&\xa&\xb&\xxx{2}\\
    \xb&\xxx{2}&\xa&\xxx{4}\\
    \xxx{2}&\xa&\xb&\xxx{4}\\
    \hline
\end{array}\ .
\end{displaymath}
\normalcols

The conditions of Lemma~\ref{LMContour} can be readily verified.
Conditions (a) and (b) are immediate from the construction.  Since
both cycles of $\alpha$ have the same length, condition (c) is
satisfied.  The staircase pattern ensures conditions (d)--(e) are
satisfied.  Hence we have indeed constructed a contour $\C$.  (In
later constructions we will not explicitly describe how
Lemma~\ref{LMContour} is satisfied.)

\textit{Case I(c)}: $d$ is even and $m$ is odd,
whence \[\alpha=(12\cdots d)(d+1\cdots 2d)\cdots((m-1)d+1\cdots md).\]
Suppose, seeking a contradiction, that $L=L(i,j)$ is a Latin square of
order $n$ such that $\alpha\in\autm(L)$.

Consider the first row of the $d\times d$ block of $L$ with top left
corner $(dr+1,ds+1)$, for some $r$ and $s$ satisfying $0 \leq r \leq
m-1$ and $0 \leq s \leq m-1$. Then, calculating modulo $d$, since
$\alpha$ is an automorphism of $L$, we have
\begin{align*}
  \sum_{t=1}^{d} L(dr+t,ds+1) &=\sum_{t=0}^{d-1} L(dr+d-t,ds+1)
= \sum_{t=0}^{d-1} L\big(\alpha^{-t}(dr+d),\alpha^{-t}(ds+1+t)\big)\\
  &=\sum_{t=0}^{d-1} \alpha^t\big(L(dr+d,ds+1+t)\big)
\equiv \sum_{t=0}^{d-1} \Big(L(dr+d,ds+1+t)+t\Big).
\end{align*}
Using this congruence and the fact that every row and column sums to
$n(n+1)/2$, we obtain
\begin{align*}
    m\frac{n(n+1)}{2} &= \sum_{s=0}^{m-1}\sum_{i=1}^n L(i,ds+1)
= \sum_{r=0}^{m-1}\sum_{s=0}^{m-1} \sum_{t=1}^{d} L(dr+t,ds+1)\\
    &\equiv \sum_{r=0}^{m-1}\sum_{s=0}^{m-1} \sum_{t=0}^{d-1}
         \big(L(dr+d,ds+1+t) +t\big)
    = m^2 \sum_{t=0}^{d-1} t+\sum_{r=0}^{m-1}\sum_{j=1}^{n} L(dr+d,j)\\
    &= m^2 \frac{(d-1)d}{2}+m\frac{n(n+1)}{2},
\end{align*}
and hence $\half m^2 (d-1)d\equiv 0 \pmod d$. This contradicts our
assumption that $d$ is even and $m$ is odd.

\textit{Case II}: $\alpha$ has at least one fixed point, so $n>md$. If
$n>2md$ then $\alpha\notin\automorphs{n}$ by Theorem~\ref{TMMM}.  If
$n\le 2md$, Theorem~\ref{TDiagCyclic} guarantees the existence of a
Latin square of order $n$ that admits the automorphism
$\omega=(12\cdots (md))(md+1)\cdots(2md)$ and so
$\omega^m\in\automorphs{n}$.  Since $\omega^m$ has the same cycle
structure as $\alpha$, Lemma~\ref{LMAutConj} implies
$\alpha\in\automorphs{n}$.
\end{proof}

\begin{corol}\label{COSmithKerby}
Suppose $2^a$ is the largest power of $2$ dividing $n$, where $a \geq
1$.  Suppose $\theta=(\alpha,\beta,\gamma) \in \I_n$ is such that the
length of each cycle in $\alpha$, $\beta$ and $\gamma$ is divisible by
$2^a$.  Then $\theta \not\in \autotops{n}$.
\end{corol}

\begin{proof}
Suppose $L$ is a Latin square of order $n$ that admits the autotopism
$\theta$.  Define the strongly $\lcm$-closed set $S=\{s \in
\mathbb{N}: 2^{a+1} \text{ does not divide } s\}$.
Theorem~\ref{THStrongLCM} implies that $L$ contains a subsquare $M$
that admits an autotopism $\theta_M$ whose components have cycle
lengths that are divisible by $2^a$, but indivisible by $2^{a+1}$.
Hence the order of $\theta_M$ is $2^a x$ for some odd $x \geq 1$.

The order of $M$ is $2^a b$ for some odd $b \geq 1$ (otherwise
$2^{a+1}$ divides $n$).  Also, $M$ admits the autotopism
$(\theta_M)^x$.  But the components of $(\theta_M)^x$ each consist of
$b$ disjoint $2^a$-cycles, so Theorem~\ref{THEqualCycLen} implies that
$(\theta_M)^x \not\in \autt(2^a b)$, giving a contradiction.
\end{proof}

Kerby and Smith \cite{KerbySmith2010} independently obtained the case
of Corollary~\ref{COSmithKerby} when $\theta$ is an automorphism.  The
special case when all cycles have length two can be found in
the proof of \cite[Lemma~4]{McKayMeynertMyrvold2007}.

\section{Some useful partial contours}\label{Sc:AdditionalBlockPatterns}

There are several situations that arise repeatedly while constructing
a Latin square with a prescribed automorphism. We now give a sequence
of lemmas that handle these situations.  As usual, we assume that
$\alpha\in S_n$ is canonical and has nontrivial cycles of lengths
$d_1\ge d_2\ge \cdots \ge d_m$ and $d_\infty\ge 0$ fixed points.

First we look at adding fixed points to an existing construction.

\begin{lemm}\label{LMaddfix}
Suppose $\alpha\in\autm(L)$ for some Latin square $L$ of order $n$.  Let
$\mu$ be the minimum, over $k\in[m]$, of the number of occurrences of
the leading symbol $t_k$ in block $M_{kk}$.
Then for $0\le \nu\le\mu$ there exists a Latin square $L'$ of order
$n+\nu$ that admits an automorphism $\alpha'$ with cycles of lengths
$d_1$, $d_2,\ldots, d_m$ and ${d_\infty+\nu}$ fixed points.
\end{lemm}

\begin{proof}
We use a procedure known as \emph{prolongation} to construct
$L'$ from $L$.  We assume $\nu=1$; the remainder of the
lemma follows by induction.  We may also assume that $\alpha$ and
$\alpha'$ are canonical. We use $M_{ij}$ to denote a block of $L$
and $M'_{ij}$ to denote a block of $L'$.

For each $k \in [m]$, pick an entry $(i,j,t_k)$ in the block $M_{kk}$.
Define $L'(\alpha^r(i),n+1)=L'(n+1,\alpha^r(j))=\alpha^r(t_k)$ and
$L'(\alpha^r(i),\alpha^r(j))=n+1$ for $0 \leq r \leq d_k-1$.
Then $M'_{\infty\infty}$ can be
chosen arbitrarily from the Latin squares of order $d_\infty+1$ on the
symbols $\alpha_\infty\cup\{n+1\}$.
The remainder of $L'$ is the same as $L$.
\end{proof}

When attempting to construct a contour $\C$ for a Latin square that
admits the automorphism $\alpha$, Lemma~\ref{LMContour} implies that
we may proceed block-by-block, in any order.  We imagine that we build
$\C$ from a series of partial contours, such that at each stage we
introduce a new block $M_{ij}$.  It is sufficient to check that, at
each stage, the introduced block $M_{ij}$ does not contradict
conditions (a)--(e) of Lemma~\ref{LMContour} with respect to itself
and the other extant blocks in those rows and columns.

For the next lemma, we consider the case of when the nontrivial cycles
of $\alpha$ have distinct lengths, and give sufficient conditions for
the existence of blocks $M_{ij}$ satisfying the conditions of
Lemma~\ref{LMContour}, when either $i=\infty$ or $j=\infty$.

\begin{lemm}\label{LMinfrowcol}
Suppose that no two nontrivial cycles of $\alpha$ have the same length.
\begin{enumerate}
\item[(i)] If $\alpha\in\autm(L)$ and $i \in [m]$, then the block
  $M_{i\infty}$ contains one copy of the leading symbol $t_i$
  in each column and no other leading symbols.
\item[(ii)] While constructing $L$ such that $\alpha\in\autm(L)$, if
  the region $\bigcup_{j\in[m]}M_{ij}$ for some $i \in [m]$ has been
  successfully completed and contains exactly $d_i-d_\infty$ copies of
  $t_i$, then $M_{i\infty}$ can also be completed.
\end{enumerate}
Similar statements (transposed) hold for the blocks $M_{\infty i}$.
\end{lemm}

\begin{proof}
Assume that $\alpha\in\autm(L)$. Then each column of $M_{i\infty}$ is
a cell orbit and each must contain a leading symbol from a cycle of
length $\lcm(d_i,1)=d_i$, by Lemma~\ref{LMLcmAB}. By assumption $t_i$
is the only such leading symbol.

Now assume that $L$ is under construction and we wish to achieve
$\alpha\in\autm(L)$. If there are $d_i-d_\infty$ copies of $t_i$ in
$\bigcup_{j\in[m]}M_{ij}$ then there are $d_\infty$ rows where $t_i$
does not occur in any of these blocks. We can place $t_i$ in each of
these rows within $M_{i\infty}$, with one copy of $t_i$ per column and
per row, but otherwise arbitrarily. It is easy to check that the
conditions of Lemma~\ref{LMContour} will continue to be satisfied. The
transpose argument works for $M_{\infty i}$.
\end{proof}

The next result shows when it is possible to fill several subsquares
at once, provided they only overlap in the block $M_{\infty\infty}$.
Conditions for the existence of Latin squares with overlapping
subsquares of various sizes were given in
\cite{BrowningVojtechovskyWanless2009}.


\begin{lemm}\label{LMsubsq}
For $i\in[m]$, let $\lambda_i$ be the number of cycles in $\alpha$
that have length $d_i$. Let $I = \{i\in[m];\;$ there is no $j\in[m]$
such that $d_j$ is a proper divisor of $d_i\}$. For all $i\in I$, let
$\S_i=\bigcup_{a,b} M_{ab}$ over all $a,b \in \{c \in [m]:d_c=d_i\}
\cup \{\infty\}$.
\begin{enumerate}
\item[(i)] If $\alpha\in\autm(L)$, then $\S_i$ is a subsquare of $L$
  for every $i\in I$.  Hence the region $\bigcup_{i\in I}\S_i$ can be
  filled independently of the remainder of $L$.
\item[(ii)] If $\alpha\in\autm(L)$, then (a) $d_\infty\le\lambda_id_i$
  for every $i\in I$ and (b) if $d_i$ is even and $\lambda_i$ is odd
  for some $i\in I$ then $d_\infty>0$.
\item[(iii)] If $\alpha$ satisfies conditions (a)--(b) of (ii) above,
  then it is possible to fill the region $\bigcup_{i \in I} \S_i$.
\end{enumerate}
\end{lemm}

\begin{proof}
Suppose $L$ is a Latin square with $\alpha\in\autm(L)$.  For each
$i\in I$, taking $\Lambda$ as the set of divisors of $d_i$ in
Theorem~\ref{THStrongLCM} implies that $\S_i$ is a subsquare that
admits an automorphism with the cycle structure
$d_i^{\lambda_i}\tdot1^{d_\infty}$, thus proving (i).
Theorem~\ref{THEqualCycLen}, applied to each subsquare $\S_i$, now
implies (ii).

The assumptions on $\alpha$ imply that any two distinct subsquares
$\S_i$ and $\S_j$ intersect at $M_{\infty\infty}$ (which is empty if
$d_\infty=0$). If $d_\infty>0$ then \tref{THStrongLCM} with
$\Lambda=\{1\}$ implies that $M_{\infty\infty}$ is a subsquare.
Crucially, if $\S_i$ exists then we can replace its subsquare
$M_{\infty\infty}$ by any other subsquare on the same symbols, without
disrupting the automorphism that $\S_i$ is required to have. Hence, if
the individual $\S_i$ exist, we may assume they share the same
subsquare $M_{\infty\infty}$, which is the only place that they
overlap.  Thus $\bigcup_{i\in I}\S_i$ can be constructed if and only
if all the individual $\S_i$ can be constructed, which happens if and
only if the conditions of (ii) are satisfied, by \tref{THEqualCycLen}.
\end{proof}

We next look at a useful way to construct a partial contour.
In a block with $g$ cell orbits, a {\em transversal} is a set of $g$
cells that lie in different rows, different columns and different cell
orbits.  We will now give a simple but important condition for the
existence of a transversal.

\begin{lemm}\label{LMgapnec}
Let $S=\{(r_k,c_k)\}_{1 \leq k \leq g}$ be a transversal of a
$d_i\times d_j$ block $M_{ij}$ with $i,j \in [m]$, where
$g=\gcd(d_i,d_j)$.  Then
\begin{equation}\label{EQSumModG}
\sum_{k=1}^g c_k-\sum_{k=1}^g r_k \equiv \delta_g\pmod g
\end{equation}
where
\begin{equation*}
    \delta_g=\sum_{k\in\mathbb{Z}_g}k \equiv
        \begin{cases}
            0\pmod g&\text{if $g$ is odd},\\
            \half g\pmod g&\text{if $g$ is even}.
        \end{cases}
\end{equation*}
\end{lemm}

\begin{proof}
Label each cell of $M_{ij}$ with its column index minus its row index,
modulo $g$.  The elements of $S$ belong to different cell orbits of
$\alpha$ if and only if their labels are distinct.  Since there are
$g$ cell orbits, summing the labels yields \eqref{EQSumModG}, since
every cell orbit is represented once in $S$.
\end{proof}

Lemma~\ref{LMgapnec} is a standard argument on transversals; variants
of it have been used, for example, in \cite{EW} and \cite{WW}.

Let $N$ be the smallest submatrix containing $S$ in
Lemma~\ref{LMgapnec}. It turns out that the necessary condition in
Lemma~\ref{LMgapnec} is also sufficient in several cases that will
prove important to us later. These cases involve a situation where the
rows and columns of $N$ are contiguous within $M_{ij}$, except
possibly for one gap, which either splits the rows and columns of $N$
in half, or separates one column and one row from the remaining
columns and rows of $N$.  In the following result, $e$ is the number
of rows and columns in one of the two parts of $N$, $h_1$ is the
vertical gap (between rows) and $h_2$ is the horizontal gap (between
columns).

\begin{lemm}\label{LMgapsuff}
Suppose that $N$ is a $g\times g$ submatrix of a $d_i\times d_j$ block
$M$, where $g=\gcd(d_i,d_j)$. Suppose that for some integers $r$, $c$,
$e$, $h_1$, $h_2$ the submatrix $N$ is formed by the rows
\[\{r-e+1,r-e+2,\ldots,r\} \cup \{r+h_1+1,r+h_1+2,\ldots,r+h_1+g-e\}\]
and columns
\[\{c-e+1,c-e+2,\ldots,c\} \cup \{c+h_2+1,c+h_2+2,\ldots,c+h_2+g-e\}\]
of $M$. Suppose further that $e=g-1$ or $g=2e$. Then for $M$ to have
a transversal inside $N$ it is necessary and sufficient that
\begin{equation}\label{Eq:gapstuff}
    (h_1-h_2)e\equiv\delta_g\pmod g.
\end{equation}
\end{lemm}

\begin{proof}
For the necessity we apply Lemma~\ref{LMgapnec} and find that
\begin{align*}
\delta_g&=\sum_{i=1}^e\big((c-e+i)-(r-e+i)\big)
+\sum_{i=1}^{g-e}\big((c+h_2+i)-(r+h_1+i)\big)\\
&=g(c-r)+(g-e)(h_2-h_1)\\
&\equiv(h_1-h_2)e\pmod g.
\end{align*}
To prove sufficiency, first suppose that $g=2e$. In this case
$\delta_g=\half g=e$ and $(h_1-h_2)e\equiv\delta_g\pmod g$ implies
that $h_1-h_2$ is odd.  Let $T$ be the set of cells
\[\{(r-e+i,c+1-i):1 \leq i \leq e\}\cup\{(r+h_1+i,c+h_2+e+1-i):1 \leq i \leq e\}.\]
It is immediate that $T$ has a representative from every row and
column of $N$. Moreover, the labels on the cells in $T$ (as defined in
the proof of \lref{LMgapnec}) are $\{c-r+1+e-2i:1 \leq i \leq
e\}\cup\{c-r+1+e+h_2-h_1-2i:1 \leq i \leq e\}=\mathbb{Z}_g$ since
$h_2-h_1$ is odd.  Hence $T$ is indeed a transversal.

Next suppose that $e=g-1$ and hence
$h_2-h_1\equiv\delta_g\pmod g$. If $g$ is even then we form $T$ from the cells
\begin{align*}
\big\{(r+h_1+1,c+h_2+1)\big\}
&\cup\big\{(r-\half g-i+2,c-g+i+1):1 \leq i \leq \half g\big\}\\
&\cup\big\{(r-i+1,c-\half g+i+1):1 \leq i \leq \half g-1\big\}.
\end{align*}
The labels on these cells are
\[
\{c-r+h_2-h_1\}\cup
\{c-r-\half g+2i-1:1 \leq i \leq \half g\}\cup
\big\{c-r-\half g+2i:1 \leq i \leq \half g-1\}.
\]
Since $h_2-h_1\equiv\delta_g \equiv \half g\pmod g$, these labels cover every
possibility modulo $g$.  Similarly, it is easy to see that $T$ covers
every row and column of $N$.

It remains to show sufficiency when $e=g-1$ and $g$ is odd. In this case
we simply take $T$ to consist of the cells
\[
\big\{(r+h_1+1,c+h_2+1)\big\}\cup
\big\{(r-i+1,c-g+i+1):1 \leq i \leq e\big\}.
\]
The first of these has label $c-r+h_2-h_1\equiv c-r\pmod g$, while the others
have labels $\{c-r+2i:1 \leq i \leq e\}$, which gives us a complete set.
\end{proof}

We remark that Lemma~\ref{LMgapsuff} can also be applied when the rows
or columns are consecutive, by choosing $h_1=0$ or $h_2=0$,
respectively.

\section{Automorphisms with two nontrivial cycles}\label{Sc:Autm2Cycles}

In this section we give necessary and sufficient conditions for
membership in $\automorphs{n}$ for those $\alpha \in S_n$ that consist
of precisely two nontrivial cycles, of lengths $d_1$ and $d_2$.

\begin{theo}[Automorphisms with two nontrivial cycles]\label{THTwoCycles}
Suppose $\alpha\in S_n$ consists of a $d_1$-cycle, a $d_2$-cycle and
$d_{\infty}$ fixed points. If $d_1=d_2$ then $\alpha\in\automorphs{n}$
if and only if $0 \leq d_{\infty} \leq 2d_1$. If $d_1>d_2$ then
$\alpha\in\automorphs{n}$ if and only if all the following conditions
hold:
\begin{enumerate}
\item[(a)] $d_2$ divides $d_1$,
\item[(b)] $d_2 \geq d_\infty$,
\item[(c)] if $d_2$ is even then $d_\infty>0$.
\end{enumerate}
\end{theo}

\begin{proof}
The case $d_1=d_2$ is resolved by Theorem~\ref{THEqualCycLen}, so
assume $d_1>d_2$. Suppose $L$ is a Latin square with
$\alpha\in\autm(L)$.  The block diagram of $L$ must be as in
Figure~\ref{FILStruct}, as explained in
Section~\ref{Ss:BlockDiagrams}. The necessity of conditions (b), (c)
follows from Lemma \ref{LMsubsq}. To see that (a)
is necessary, observe that every symbol in $M_{12}$ belongs to the
$d_1$-cycle $\alpha_1$. Then $d_1=\lcm(d_1,d_1)=\lcm(d_1,d_2)$ by
Lemma~\ref{LMLcmAB}, so $d_2$ must divide $d_1$. (Note that we now have $n=d_1+d_2+d_\infty \le d_1+2d_2\le 2d_1$, so $d_1\geq \left\lceil \half n \right\rceil$, as also demanded by Lemma \ref{LMMaxSize}.)

For the rest of the proof assume that conditions (a)--(c) hold. Our
task is to find a Latin square $L$ such that $\alpha\in\autm(L)$. We
construct such a square by means of a contour $\C=\C(i,j)$ that
satisfies the conditions of Lemma~\ref{LMContour} for the least
possible $d_\infty$.  Examples with larger $d_\infty$ can then be
found using Lemma~\ref{LMaddfix}.

\textit{Case I}: $d_2$ is odd.   Here $d_\infty=0$.  First
we specify the block $M_{11}$:
\[
\C(i,t_2-i-\O_{1,i})=
\begin{cases}
t_1&\text{if }1\le i\le d_1-d_2,\\
t_2&\text{if }d_1-d_2< i\le d_1.\\
\end{cases}
\]
The block $M_{22}$ can be completed by Lemma~\ref{LMsubsq}, and the
blocks $M_{12}$ and $M_{21}$ can be completed by applying
Lemma~\ref{LMgapsuff}.


The contours for $d_1=6$, $d_2=3$, $d_\infty=0$ and $d_1=9$, $d_2=3$,
$d_\infty=0$ are illustrated in Figure~\ref{Fg:93}.

\begin{figure}[htb]
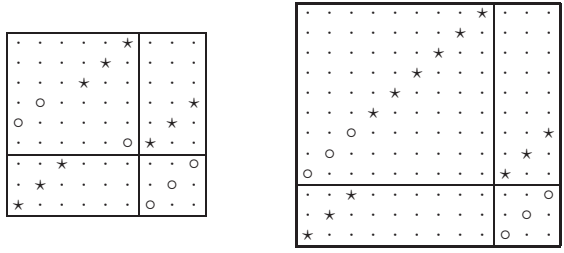

\verynarrowcols
\begin{displaymath}
\begin{array}{|cccccc|ccc|}
    \hline
    \xxx{5}&\xa&\xxx{3}\\
    \xxx{4}&\xa&\xxx{4}\\
    \xxx{3}&\xa&\xxx{5}\\
    \xx&\xb&\xxx{6}&\xa\\
    \xb&\xxx{6}&\xa&\xx\\
    \xxx{5}&\xb&\xa&\xxx{2}\\
    \hline
    \xxx{2}&\xa&\xxx{5}&\xb\\
    \xx&\xa&\xxx{5}&\xb&\xx\\
    \xa&\xxx{5}&\xb&\xxx{2}\\
    \hline
\end{array}
\qquad\qquad
\begin{array}{|ccccccccc|ccc|}
    \hline
    \xxx{8}&\xa&\xxx{3}\\
    \xxx{7}&\xa&\xxx{4}\\
    \xxx{6}&\xa&\xxx{5}\\
    \xxx{5}&\xa&\xxx{6}\\
    \xxx{4}&\xa&\xxx{7}\\
    \xxx{3}&\xa&\xxx{8}\\
    \xxx{2}&\xb&\xxx{8}&\xa\\
    \xx&\xb&\xxx{8}&\xa&\xx\\
    \xb&\xxx{8}&\xa&\xxx{2}\\
    \hline
    \xxx{2}&\xa&\xxx{8}&\xb\\
    \xx&\xa&\xxx{8}&\xb&\xx\\
    \xa&\xxx{8}&\xb&\xxx{2}\\
    \hline
\end{array}
\end{displaymath}
\verynormalcols
\caption{\label{Fg:93}Contours for $d_1=6$, $d_2=3$, $d_\infty=0$
and $d_1=9$, $d_2=3$, $d_\infty=0$.}
\end{figure}

\textit{Case II}: $d_2$ is even (and hence $d_1$ is also even).  Here
$d_\infty=1$.  We begin with $M_{11}$:
\[
\C(i,t_2-i-\O_{1,i})=
\begin{cases}
t_1&\text{if }1\le i\le\half d_1-d_2\text{ or }\half d_1<i<d_1,\\
t_2&\text{if }\half d_1-d_2< i\le \half d_1,\\
n&\text{if }i=d_1.\\
\end{cases}
\]
The blocks $M_{22}$, $M_{2\infty}$, $M_{\infty2}$ and
$M_{\infty\infty}$ can be completed by Lemma~\ref{LMsubsq}. Once we
complete $M_{12}\cup M_{21}$, we can complete $L$ by Lemma
\ref{LMinfrowcol}. The block $M_{21}$ can be completed using
Lemma~\ref{LMgapsuff} with $e=g/2$.

Finally, we fill the block $M_{12}$ with $\C(d_1,t_3-1)=t_1$, and if
$d_1/d_2$ is even, we let
\begin{align*}
  & \C(\half d_1-d_2+i, t_3-1-i)=t_1 & \text{ for } & 1 \leq i < \half d_2-1,\\
  & \C(\half d_1 - \half d_2+i,t_2+\half d_2-i)=t_1 & \text{ for } & 1 \leq i \leq \half d_2,
\end{align*}
while if $d_1/d_2$ is odd, we let
\begin{align*}
    & \C(\half d_1-d_2+i,t_2+\half d_2-i) = t_1 & \text{ for } & 1\leq i\leq \half d_2,\\
    & \C(\half d_1 - \half d_2 + i, t_3 - 1 - i) = t_1 & \text{ for } & 1\leq i<\half d_2-1.
\end{align*}

These partial contours are illustrated in Figure~\ref{Fg:25} for $d_1=16$,
$d_2=4$, $d_\infty=1$, and $d_1=18$, $d_2=6$, $d_\infty=1$. Focusing
on the ``missing'' cell in the even pattern of the block $M_{12}$
(shaded dark in Figure~\ref{Fg:25}), it is not hard to see why the
construction works. The shaded entry is in column $t_3-1$ when
$d_1/d_2$ is even, and it is in column $t_3-1-\half d_2$ when
$d_1/d_2$ is odd.
\end{proof}

\begin{figure}[htb]
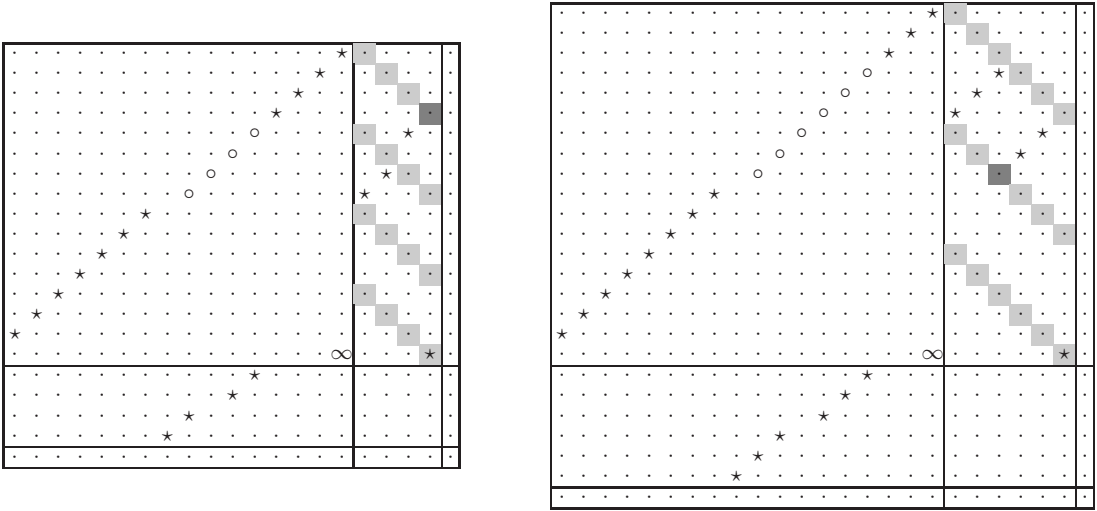

\verynarrowcols
\begin{displaymath}
\begin{array}{|cccccccccccccccc|cccc|c|}
    \hline
    \xxx{15}&\xa&\c\xx&\xxx{4}\\
    \xxx{14}&\xa&\xxx{2}&\c\xx&\xxx{3}\\
    \xxx{13}&\xa&\xxx{4}&\c\xx&\xxx{2}\\
    \xxx{12}&\xa&\xxx{6}&\d\xx&\xx\\
    \xxx{11}&\xb&\xxx{4}&\c\xx&\xx&\xa&\xxx{2}\\
    \xxx{10}&\xb&\xxx{6}&\c\xx&\xxx{3}\\
    \xxx{9}&\xb&\xxx{7}&\xa&\c\xx&\xxx{2}\\
    \xxx{8}&\xb&\xxx{7}&\xa&\xxx{2}&\c\xx&\xx\\
    \xxx{6}&\xa&\xxx{9}&\c\xx&\xxx{4}\\
    \xxx{5}&\xa&\xxx{11}&\c\xx&\xxx{3}\\
    \xxx{4}&\xa&\xxx{13}&\c\xx&\xxx{2}\\
    \xxx{3}&\xa&\xxx{15}&\c\xx&\xx\\
    \xxx{2}&\xa&\xxx{13}&\c\xx&\xxx{4}\\
    \xx&\xa&\xxx{15}&\c\xx&\xxx{3}\\
    \xa&\xxx{17}&\c\xx&\xxx{2}\\
    \xxx{15}&\xi&\xxx{3}&\c\xa&\xx\\
    \hline
    \xxx{11}&\xa&\xxx{9}\\
    \xxx{10}&\xa&\xxx{10}\\
    \xxx{8}&\xa&\xxx{12}\\
    \xxx{7}&\xa&\xxx{13}\\
    \hline
    \xxx{21}\\
    \hline
    \end{array}
\qquad\qquad
\begin{array}{|cccccccccccccccccc|cccccc|c|}
    \hline
    \xxx{17}&\xa&\c\xx&\xxx{6}\\
    \xxx{16}&\xa&\xxx{2}&\c\xx&\xxx{5}\\
    \xxx{15}&\xa&\xxx{4}&\c\xx&\xxx{4}\\
    \xxx{14}&\xb&\xxx{5}&\xa&\c\xx&\xxx{3}\\
    \xxx{13}&\xb&\xxx{5}&\xa&\xxx{2}&\c\xx&\xxx{2}\\
    \xxx{12}&\xb&\xxx{5}&\xa&\xxx{4}&\c\xx&\xx\\
    \xxx{11}&\xb&\xxx{6}&\c\xx&\xxx{3}&\xa&\xxx{2}\\
    \xxx{10}&\xb&\xxx{8}&\c\xx&\xx&\xa&\xxx{3}\\
    \xxx{9}&\xb&\xxx{10}&\d\xx&\xxx{4}\\
    \xxx{7}&\xa&\xxx{13}&\c\xx&\xxx{3}\\
    \xxx{6}&\xa&\xxx{15}&\c\xx&\xxx{2}\\
    \xxx{5}&\xa&\xxx{17}&\c\xx&\xx\\
    \xxx{4}&\xa&\xxx{13}&\c\xx&\xxx{6}\\
    \xxx{3}&\xa&\xxx{15}&\c\xx&\xxx{5}\\
    \xxx{2}&\xa&\xxx{17}&\c\xx&\xxx{4}\\
    \xx&\xa&\xxx{19}&\c\xx&\xxx{3}\\
    \xa&\xxx{21}&\c\xx&\xxx{2}\\
    \xxx{17}&\xi&\xxx{5}&\c\xa&\xx\\
    \hline
    \xxx{14}&\xa&\xxx{10}\\
    \xxx{13}&\xa&\xxx{11}\\
    \xxx{12}&\xa&\xxx{12}\\
    \xxx{10}&\xa&\xxx{14}\\
    \xxx{9}&\xa&\xxx{15}\\
    \xxx{8}&\xa&\xxx{16}\\
    \hline
    \xxx{25}\\
    \hline
\end{array}
\end{displaymath}
\verynormalcols
\caption{\label{Fg:25}Partial contours for $d_1=16$, $d_2=4$, $d_\infty=1$,
and $d_1=18$, $d_2=6$, $d_\infty=1$.}
\end{figure}

\section{Automorphisms with three nontrivial cycles}\label{Sc:Autm3Cycles}

In this section we characterize automorphisms $\alpha$ of Latin
squares with precisely three nontrivial cycles of lengths
$d_1\ge d_2\ge d_3$.

\begin{theo}[Automorphisms with three nontrivial cycles]\label{THThreeCyc}
Suppose that $\alpha\in S_n$ has precisely three nontrivial cycles of
lengths $d_1\ge d_2\ge d_3$.  Let $d_\infty$ be the number of fixed
points of $\alpha$.  Then $\alpha\in\autm(n)$ if and only if one of
the following cases holds:
\begin{enumerate}
  \item $d_1=d_2=d_3$ and (a) $d_\infty \leq 3d_1$ and (b) if $d_1$ is
    even then $d_\infty\ge 1$,
  \item $d_1>d_2=d_3$ and (a) $d_1 \geq 2d_2+d_\infty$, (b) $d_2$
    divides $d_1$, (c) $d_\infty \leq 2d_2$, and (d) if $d_2$ is even
    and $d_1/d_2$ is odd then $d_\infty>0$,
  \item $d_1=d_2>d_3$ and (a) $d_3$ divides $d_1$, (b) $d_\infty\le
    d_3$, and (c) if $d_3$ is even then $d_\infty>0$,
  \item $d_1>d_2>d_3$ and (a) $d_1 = \lcm(d_2,d_3)$, (b) $d_3\ge
    d_\infty$, and (c) if $d_1$ is even then $d_\infty>0$,
  \item $d_1>d_2>d_3$ and (a) $d_3$ divides $d_2$ which divides $d_1$,
  (b) $d_3\ge d_\infty$, and (c) if $d_3$ is even then $d_\infty>0$.
\end{enumerate}
\end{theo}

We will prove each case of Theorem~\ref{THThreeCyc} in a sequence of
propositions in the remainder of this section.  The case $d_1=d_2=d_3$
is covered by Theorem~\ref{THEqualCycLen}, so it remains to discuss
the cases when $d_1>d_2$ and/or $d_2>d_3$.

\subsection{The case $d_1>d_2=d_3$}

\begin{lemm}\label{Lm:Horse}
For any $d_1$, $d_2$ such that $d_1=2d_2$, every isotopism with cycle
structure $(d_1,d_2^2,d_1)$ belongs to $\autt(d_1)$.
\end{lemm}
\begin{proof}
We specify a contour for a Latin square $L$ of order $d_1$ that admits
the autotopism \[((1\cdots d_1),(1\cdots d_2)(d_2+1\cdots
d_1),(1\cdots d_1))\] by assigning $L(2i-1,i)=L(2i,d_2+i)=t_1$ for
$1\leq i \leq d_2$, as illustrated in Figure \ref{Fg:Horse} when
$d_1=6$.
\end{proof}

\begin{figure}[htb]
\narrowcols
\begin{displaymath}
\begin{array}{|ccc|ccc|}
    \hline
    \xa&\xxx{5}\\
    \xxx{3}&\xa&\xxx{2}\\
    \xx&\xa&\xxx{4}\\
    \xxx{4}&\xa&\xx\\
    \xxx{2}&\xa&\xxx{3}\\
    \xxx{5}&\xa\\
    \hline
\end{array}
\end{displaymath}
\normalcols
\caption{\label{Fg:Horse}An example of the construction in the proof of
Lemma~\ref{Lm:Horse}.}
\end{figure}

\begin{prop}[Automorphisms with three nontrivial cycles of lengths $d_1>d_2=d_3$]\label{P:ThreeCycles2}
Suppose that $\alpha\in S_n$ has precisely three nontrivial cycles of
lengths $d_1>d_2=d_3$. Let $d_\infty$ be the number of fixed points of
$\alpha$. Then $\alpha\in\autm(n)$ if and only if all of the following
conditions hold:
\begin{enumerate}
\item[(a)] $d_1 \geq 2d_2+d_\infty$,
\item[(b)] $d_2$ divides $d_1$,
\item[(c)] $d_\infty \leq 2d_2$,
\item[(d)] if $d_2$ is even and $d_1/d_2$ is odd then $d_\infty>0$.
\end{enumerate}
\end{prop}

\begin{proof}
Let $L$ be a Latin square such that $\alpha\in\autm(L)$. By
Lemma~\ref{LMsubsq}, $K=\bigcup_{i,j\in\{2,3,\infty\}} M_{ij}$ is a
subsquare of $L$. Since $M_{i\infty}\cup M_{\infty i}$ can be filled
later by Lemma \ref{LMinfrowcol}, it suffices to consider only the
blocks $M_{11}$, $M_{12}$, $M_{13}$, $M_{21}$ and $M_{31}$.
Figure~\ref{FIStructA} gives the part of a block diagram of $L$ that
concerns these blocks, with all entries being consequences of the fact
that $K$ is a subsquare of $L$. (Inside $K$, the block diagram of $L$
is not uniquely determined by $\alpha$, since $d_2=d_3$.)

\begin{figure}[htb]
\narrowcols
\begin{displaymath}
\begin{array}{l|l|l|l|l|}
    &\ \alpha_1&\ \alpha_2&\ \alpha_3&\ \alpha_\infty\\
    \hline
    \alpha_1&
        \begin{array}{l}\alpha_1:d_1-2d_2-d_\infty\\ \alpha_2:d_1\\ \alpha_3:d_1\\ \alpha_\infty:d_1\end{array}&
        \begin{array}{l}\alpha_1:d_2\\ \phantom{x}\\ \phantom{x} \\ \phantom{x}\end{array}&
        \begin{array}{l}\alpha_1:d_2\\ \phantom{x}\\ \phantom{x} \\ \phantom{x}\end{array}&
        \begin{array}{l}\alpha_1:d_\infty\\ \phantom{x}\\ \phantom{x} \\ \phantom{x}\end{array}\\
    \hline
    \alpha_2&\ \alpha_1:d_2&  &  & \\
    \hline
    \alpha_3&\ \alpha_1:d_2&  &  & \\
    \hline
    \alpha_\infty&\ \alpha_1:d_\infty&  &  & \\
    \hline
\end{array}
\end{displaymath}
\normalcols
\caption{\label{FIStructA}Part of the block diagram of $L$ with $d_1>d_2=d_3$.}
\end{figure}

By \tref{THEqualCycLen}, the subsquare $K$ can be filled provided (c)
holds. From the $M_{11}$ block of $L$ we deduce (a) and (b). To prove
that (d) is necessary, suppose that $d_2$ is even, $d_1/d_2$ is odd
and $d_\infty=0$. Let $k$ be the largest odd divisor of $d_1$.  Then
$\alpha^k$ consists of an odd number of cycles of the even length
$d_1/k$, contradicting \tref{THEqualCycLen}.

For the sufficiency, assume that conditions (a)--(d) are satisfied,
and let us construct a partial contour for $L\setminus K$. 

Case I: $d_\infty>0$ or $d_1$ is odd. Then we set
\begin{align*}
    &\C(i,t_2-i-\O_{1,i}) = t_2 &\text{ for } &\half d_1 - d_2 + 1 \le i \le \half d_1,\\
    &\C(i,t_2-i-\O_{1,i}) = t_3 &\text{ for } &\half d_1 + 1 \le i \le \half d_1 + d_3,
\end{align*}
and fill the remaining cells in $D=\{(i,t_2-i-\O_{1,i});\;1\le i\le d_1\}$
with the symbol $t_1$ and fixed points, making sure that a fixed point
appears in the last row when $d_1$ is even, so that the column $d_1$
does not contain two symbols $t_1$.

Note that there are at least $2d_2$ consecutive rows and columns in
$D$ not occupied by the leading symbol $t_1$. We can therefore fill
$M_{12}\cup M_{13}$ with the pattern of Lemma \ref{Lm:Horse}, and
$M_{21}\cup M_{31}$ with the transposed pattern of Lemma
\ref{Lm:Horse}. A partial contour for $L\setminus K$ in the case $d_1=6$, $d_2=d_3=2$, $d_\infty=1$ can be found in Figure \ref{Fg:Split}.

Case II: $d_\infty=0$ and $d_1$ is even. If $d_2$ is odd then it
divides $d_1/2$. If $d_2$ is even then $d_1/d_2$ is even by (d), so
$d_2$ divides $d_1/2$ again. We can modify the partial contour from
Case I as follows:

In $M_{11}$, we swap the symbols of the partial contour in rows $d_1$
and $d_1/2+d_3$ (the bottom-most occurrence of $t_3$), to prevent column
$d_1$ from containing two copies of symbol $t_1$. Since there are still
$2d_2$ consecutive columns in $D$ not occupied by $t_1$, we can fill
$M_{21}\cup M_{31}$ as above. We also leave $M_{12}$ intact, but in
$M_{13}$ we move the bottom-most symbol $t_1$ in the partial contour
down to row $d_1$, i.e., by $d_1/2-d_3$ rows. Because $d_1/2-d_3$ is a
multiple of $d_3$, \lref{LMContour} is satisfied after these changes.

A partial contour for $L\setminus K$ in the case $d_1=12$,
$d_2=d_3=3$, $d_\infty=0$ can be found in Figure \ref{Fg:Split}, with
the changes introduced in Case II highlighted.
\end{proof}

\begin{figure}[htb]
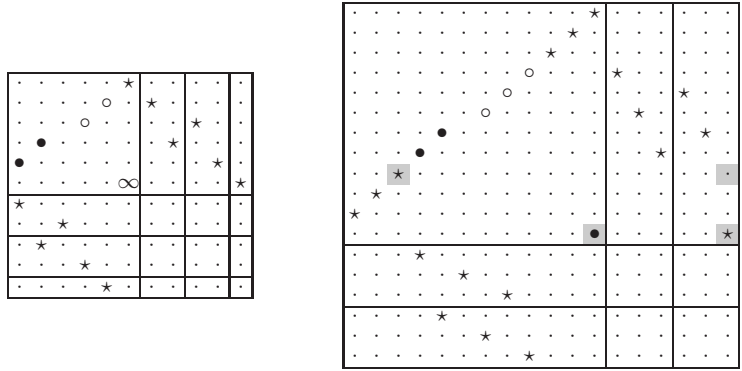

\verynarrowcols
\begin{displaymath}
\begin{array}{|cccccc|cc|cc|c|}
    \hline
    \xxx{5}&\xa&\xxx{5}\\
    \xxx{4}&\xb&\xx&\xa&\xxx{4}\\
    \xxx{3}&\xb&\xxx{4}&\xa&\xxx{2}\\
    \xx&\xc&\xxx{5}&\xa&\xxx{3}\\
    \xc&\xxx{8}&\xa&\xx\\
    \xxx{5}&\xi&\xxx{4}&\xa\\
    \hline
    \xa&\xxx{10}\\
    \xxx{2}&\xa&\xxx{8}\\
    \hline
    \xx&\xa&\xxx{9}\\
    \xxx{3}&\xa&\xxx{7}\\
    \hline
    \xxx{4}&\xa&\xxx{6}\\
    \hline
\end{array}
\qquad\qquad
\begin{array}{|cccccccccccc|ccc|ccc|}
    \hline
    \xxx{11}&\xa&\xxx{6}\\
    \xxx{10}&\xa&\xxx{7}\\
    \xxx{9}&\xa&\xxx{8}\\
    \xxx{8}&\xb&\xxx{3}&\xa&\xxx{5}\\
    \xxx{7}&\xb&\xxx{7}&\xa&\xxx{2}\\
    \xxx{6}&\xb&\xxx{6}&\xa&\xxx{4}\\
    \xxx{4}&\xc&\xxx{11}&\xa&\xx\\
    \xxx{3}&\xc&\xxx{10}&\xa&\xxx{3}\\
    \xxx{2}&\c\xa&\xxx{14}&\c\xx\\
    \xx&\xa&\xxx{16}\\
    \xa&\xxx{17}\\
    \xxx{11}&\c\xc&\xxx{5}&\c\xa\\
    \hline
    \xxx{3}&\xa&\xxx{14}\\
    \xxx{5}&\xa&\xxx{12}\\
    \xxx{7}&\xa&\xxx{10}\\
    \hline
    \xxx{4}&\xa&\xxx{13}\\
    \xxx{6}&\xa&\xxx{11}\\
    \xxx{8}&\xa&\xxx{9}\\
    \hline
\end{array}
\end{displaymath}
\verynormalcols
\caption{\label{Fg:Split}
Partial contours for $d_1=6$, $d_2=d_3=2$, $d_\infty=1$, and $d_1=12$, $d_2=d_3=3$, $d_\infty=0$.}
\end{figure}

\subsection{The case $d_1=d_2>d_3$}

\begin{prop}[Automorphisms with three nontrivial cycles of lengths $d_1=d_2>d_3$]\label{P:ThreeCycles3}
Suppose that $\alpha\in S_n$ has precisely three nontrivial cycles of
lengths $d_1=d_2>d_3$. Let $d_\infty$ be the number of fixed points of
$\alpha$. Then $\alpha\in\automorphs{n}$ if and only if all of the
following conditions hold:
\begin{enumerate}
\item[(a)] $d_3$ divides $d_1$,
\item[(b)] $d_\infty\le d_3$,
\item[(c)] if $d_3$ is even then $d_\infty>0$.
\end{enumerate}
\end{prop}

\begin{proof}
Let $L$ be a Latin square such that
$\alpha\in\autm(L)$. Lemma~\ref{LMsubsq} implies that
$K=\bigcup_{i,j\in \{3,\infty\}}M_{ij}$ is a subsquare of $L$ and that conditions
(b) and (c) must hold.  Theorem~\ref{THStrongLCM} implies that
$\alpha^{d_1}=\id$, otherwise $L$ contains a subsquare of order
$2d_1+d_\infty>\half n$, contradicting Lemma~\ref{LMMaxSize}. Hence
$d_3$ must divide $d_1$, which is condition (a).

For the sufficiency, assume that conditions (a)--(c) hold.  Again, we
need only find a partial contour for $L\setminus K$.

\textit{Case I}: $d_3$ is odd.  Theorem~\ref{THTwoCycles} implies that
$\beta \in \automorphs{n}$, where $\beta$ has the cycle structure
$(2d_1)\tdot d_3\tdot 1^{d_\infty}$. Since $\beta^2\in \autm(n)$ has
the same cycle structure as $\alpha$, we have $\alpha \in
\automorphs{n}$ by Lemma~\ref{LMAutConj}.

\textit{Case II}: $d_3$ is even.  We will construct a partial contour
satisfying the conditions of Lemma~\ref{LMContour} for the case when
$d_\infty=1$. Examples with larger $d_\infty$ can then be found using
Lemma~\ref{LMaddfix}.

We first define the partial contour for $M_{11}\cup M_{12} \cup M_{21} \cup M_{22}$.
\begin{align*}
  & \C(\half d_1+1,\half d_1)=n, & \\
  & \C(\half d_1+1+i,\half d_1-i)=t_2 & \text{ for } & 1 \leq i \leq \half d_1-1,\\
  & \C(\half d_1+i,\half d_1-i+2)=t_1 & \text{ for } & 1 \leq i \leq \half d_1,
\\[1.5ex]
  & \C(i,2d_1-i)=t_2 & \text{ for } & 1 \leq i \leq \half d_1,\\
  & \C(i,t_3-i)=t_1 & \text{ for } & 1 \leq i \leq \half d_1-d_3,\\
  & \C(\half d_1-d_3+i,t_3-\half d_1+d_3-i)=t_3 & \text{ for } & 1 \leq i \leq d_3,
\\[1.5ex]
  & \C(t_2,1)=t_1,\\
  & \C(d_1+i,t_2-i)=t_2 & \text{ for } & 1 \leq i \leq \half d_1-d_3,\\
  & \C(\tfrac32 d_1-d_3+i,\half d_1+d_3-i+1)=t_3 & \text{ for } & 1 \leq i \leq d_3,\\
  & \C(t_2+i,t_2-i)=t_1 & \text{ for } & 1 \leq i \leq \half d_1-1,
\\[1.5ex]
  & \C(\tfrac32 d_1+1,\tfrac32 d_1)=n, &\\
  & \C(\tfrac32 d_1+i,\tfrac32 d_1-i)=t_2 & \text{ for } & 1 \leq i \leq \half d_1-1,\\
  & \C(\tfrac32 d_1+1+i,\tfrac32 d_1-i)=t_1 & \text{ for } & 1 \leq i \leq \half d_1-1,\\
  & \C(2d_1,2d_1)=t_2.
\end{align*}
We can now fill $M_{31}$ using Lemma~\ref{LMgapsuff}, with $e=\half g$
and symbol $t_2$, and also $M_{32}$ with symbol $t_1$. In blocks
$M_{13}\cup M_{23}$ we use similar cells (transposed), but we use both
leading symbols $t_2$, $t_3$ in both blocks as follows:
\begin{align*}
  & \C(\half d_1-d_3+i,n-i)=t_1 & \text{ for } & 1 \leq i \leq \half d_3,\\
  & \C(\half d_1-\half d_3+i+1,n-\half d_3-i)=t_1 & \text{ for } & 1 \leq i \leq \half d_3-1,\\
  & \C(\half d_1+1,t_3)=t_2,
\\[1.5ex]
  & \C(\tfrac32 d_1-d_3+i,n-i)=t_2 & \text{ for } & 1 \leq i \leq \half d_3,\\
  & \C(\tfrac32 d_1-\half d_3+i+1,n-\half d_3-i)=t_2 & \text{ for } & 1 \leq i \leq \half d_3-1,\\
  & \C(\tfrac32 d_1+1,t_3)=t_1.
\end{align*}
The blocks $M_{1\infty} \cup M_{2\infty}\cup M_{\infty1}\cup
M_{\infty2}$ can be filled using Lemma~\ref{LMinfrowcol}. The
construction (with the partial contour of $L\setminus K$) is
illustrated in Figure \ref{Fg:2Same} for $d_1=d_2=12$, $d_3=4$ and
$d_\infty=1$.
\end{proof}

\begin{figure}[htb]
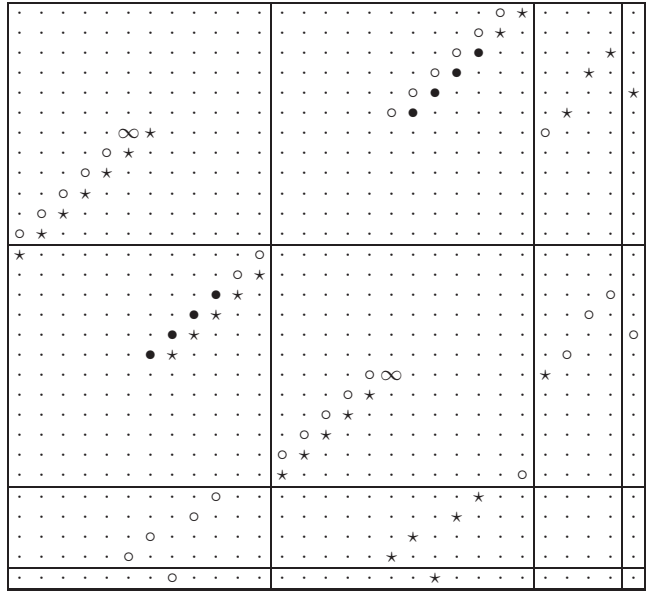

\verynarrowcols
\begin{displaymath}
\begin{array}{|cccccccccccc|cccccccccccc|cccc|c|}
    \hline
    \xxx{22}&\xb&\xa&\xxx{5}\\
    \xxx{21}&\xb&\xa&\xxx{6}\\
    \xxx{20}&\xb&\xc&\xxx{5}&\xa&\xx\\
    \xxx{19}&\xb&\xc&\xxx{5}&\xa&\xxx{2}\\
    \xxx{18}&\xb&\xc&\xxx{8}&\xa\\
    \xxx{17}&\xb&\xc&\xxx{6}&\xa&\xxx{3}\\
    \xxx{5}&\xi&\xa&\xxx{17}&\xb&\xxx{4}\\
    \xxx{4}&\xb&\xa&\xxx{23}\\
    \xxx{3}&\xb&\xa&\xxx{24}\\
    \xxx{2}&\xb&\xa&\xxx{25}\\
    \xx&\xb&\xa&\xxx{26}\\
    \xb&\xa&\xxx{27}\\
    \hline
    \xa&\xxx{10}&\xb&\xxx{17}\\
    \xxx{10}&\xb&\xa&\xxx{17}\\
    \xxx{9}&\xc&\xa&\xxx{16}&\xb&\xx\\
    \xxx{8}&\xc&\xa&\xxx{16}&\xb&\xxx{2}\\
    \xxx{7}&\xc&\xa&\xxx{19}&\xb\\
    \xxx{6}&\xc&\xa&\xxx{17}&\xb&\xxx{3}\\
    \xxx{16}&\xb&\xi&\xxx{6}&\xa&\xxx{4}\\
    \xxx{15}&\xb&\xa&\xxx{12}\\
    \xxx{14}&\xb&\xa&\xxx{13}\\
    \xxx{13}&\xb&\xa&\xxx{14}\\
    \xxx{12}&\xb&\xa&\xxx{15}\\
    \xxx{12}&\xa&\xxx{10}&\xb&\xxx{5}\\
    \hline
    \xxx{9}&\xb&\xxx{11}&\xa&\xxx{7}\\
    \xxx{8}&\xb&\xxx{11}&\xa&\xxx{8}\\
    \xxx{6}&\xb&\xxx{11}&\xa&\xxx{10}\\
    \xxx{5}&\xb&\xxx{11}&\xa&\xxx{11}\\
    \hline
    \xxx{7}&\xb&\xxx{11}&\xa&\xxx{9}\\
    \hline
\end{array}
\end{displaymath}
\verynormalcols
\caption{\label{Fg:2Same}A partial contour for
$d_1=d_2=12$, $d_3=4$ and $d_\infty=1$.}
\end{figure}

\subsection{The case $d_1>d_2>d_3$}

The case of three distinct cycle lengths splits into two, depending on
whether or not $d_3$ divides $d_2$.

\begin{prop}\label{P:ThreeCyclesA}
Suppose that $\alpha\in S_n$ has precisely three nontrivial cycles of
lengths $d_1>d_2>d_3$ where $d_3$ does not divide $d_2$.
Let $d_\infty$ be the number of fixed points of
$\alpha$. Then $\alpha\in\automorphs{n}$ if and only if
(a) $d_1 = \lcm(d_2,d_3)$, (b) $d_3\ge d_\infty$, and
(c) if $d_1$ is even then $d_\infty>0$.
\end{prop}

\begin{proof}
Let $L$ be a Latin square such that $\alpha\in\autm(L)$.
For $i \in \{1,2,3\}$, blocks $M_{i\infty}$ and $M_{\infty i}$ can be
constructed using \lref{LMinfrowcol}.  Also $M_{22}\cup
M_{2\infty}\cup M_{\infty2}\cup M_{33}\cup M_{3\infty}\cup
M_{\infty3}\cup M_{\infty\infty}$ can be constructed using
Lemma~\ref{LMsubsq}, assuming (b). Lemma~\ref{LMLcmAB} implies that
only symbols from $\alpha_1$ can appear in $M_{23}$ and $M_{32}$,
since $d_3$ does not divide $d_2$. Combining this information, and the
constraints from the definition of a Latin square shows that the block
diagram of $L$ must be as in Figure~\ref{FILStruct3Cyc}.

\begin{figure}[htb]
\narrowcols
\begin{displaymath}
\begin{array}{l|l|l|l|l|}
    &\ \alpha_1&\ \alpha_2&\ \alpha_3&\ \alpha_\infty\\
    \hline
    \alpha_1&
        \begin{array}{l} \alpha_1:d_1-d_2-d_3+2g-d_\infty\\ \alpha_2:d_1-d_3\\ \alpha_3:d_1-d_2\\ \alpha_\infty: d_1 \end{array}&
        \begin{array}{l} \alpha_1:d_2-g\\ \alpha_3:d_2\\ \phantom{x}\\ \phantom{x} \end{array}&
        \begin{array}{l} \alpha_1:d_3-g\\ \alpha_2:d_3\\ \phantom{x}\\ \phantom{x} \end{array}&
        \begin{array}{l} \alpha_1:d_\infty\\ \phantom{x}\\ \phantom{}\\ \phantom{x} \end{array}\\
    \hline
    \alpha_2&
        \begin{array}{l} \alpha_1:d_2-g\\ \alpha_3:d_2\end{array}&
        \begin{array}{l} \alpha_2:d_2-d_\infty\\ \alpha_\infty:d_2 \end{array}&
        \begin{array}{l} \alpha_1:g\\ \phantom{x} \end{array}&
        \begin{array}{l} \alpha_2:d_\infty\\ \phantom{x} \end{array}\\
    \hline
    \alpha_3&
        \begin{array}{l} \alpha_1:d_3-g\\ \alpha_2:d_3 \end{array}&
        \begin{array}{l} \alpha_1:g\\ \phantom{x} \end{array}&
        \begin{array}{l} \alpha_3:d_3-d_\infty\\ \alpha_\infty:d_3 \end{array}&
        \begin{array}{l} \alpha_3:d_\infty\\ \phantom{x} \end{array}\\
    \hline
    \alpha_\infty&
        \begin{array}{l} \alpha_1:d_\infty \end{array}&
        \begin{array}{l} \alpha_2:d_\infty \end{array}&
        \begin{array}{l} \alpha_3:d_\infty \end{array}&
        \begin{array}{l} \alpha_\infty:d_\infty \end{array}\\
    \hline
\end{array}
\end{displaymath}
\normalcols
\caption{\label{FILStruct3Cyc}Block diagram of $L$ with $d_1>d_2>d_3$
  the only nontrivial cycle lengths, $d_3$ does not divide $d_2$, and
  $g=d_2d_3/d_1 = \gcd(d_2,d_3)$.}
\end{figure}

Applying \lref{LMLcmAB} to $M_{23}$ implies
$\lcm(d_1,d_2,d_3)=\lcm(d_2,d_3)$ and so $d_1$ divides
$\lcm(d_2,d_3)$.  Applying Lemma~\ref{LMLcmAB} to $M_{11}$ implies
$d_1=\lcm(d_1,d_1)=\lcm(d_1,d_2)$, and similarly
$d_1=\lcm(d_1,d_3)$. Hence $d_2$ and $d_3$ both divide $d_1$.
Therefore $d_1=\lcm(d_2,d_3)$, which is condition (a), and
$g=\gcd(d_2,d_3) = d_2d_3/d_1$. Given (a), we see that $d_1$ is even
precisely when at least one of $d_2,d_3$ is even.
Lemma~\ref{LMsubsq}(ii) then shows the necessity of conditions (b) and
(c).

To prove sufficiency, we will give constructions of contours for the
case when $d_\infty=1$ and $d_1$ is even, and also when $d_\infty=0$
and $d_1$ is odd. Examples with larger $d_\infty$ can then be found
using Lemma~\ref{LMaddfix}.

By Figure \ref{FILStruct3Cyc}, each symbol from $\alpha_2$ occurs
$d_1-d_3$ times in the block $M_{11}$. Note that
$(d_1-d_3)/(\lcm(d_1,d_2)/d_2) = (d_1-d_3)/(d_1/d_2) = d_2-g$. We
therefore need to place $d_2-g$ leading entries $t_2$ into
$M_{11}$. Similarly, we need to place $d_3-g$ leading entries $t_3$
into $M_{11}$.

It will be of importance in some contours to place these
entries $t_2$ and $t_3$ into at most $d_1/2$ consecutive rows and
columns of $M_{11}$. We claim that this can be done, because
$\half d_1\ge (d_2-g) + (d_3-g)$. Indeed, if $d_1=abg$, $d_2=ag$ and
$d_3=bg$, the inequality is equivalent to
$\half{ab}=\half(a-2)(b-2)+a+b-2 \ge a+b-2$, which holds since $a>b\ge2$.

It is convenient to consider four cases. The cases (i)--(iii) will be
handled with the same contour (up to the usual parity offsets) but
they will require separate explanations. The case (iv) will require a
slightly different contour. The cases are:

\begin{enumerate}
\item[(i)] $d_1$, $d_2$, $d_3$ and $g$ are all odd, $d_\infty=0$,
\item[(ii)] $d_1$ is even, precisely one of $d_2$ and $d_3$ is even,
  $g$ is odd and $d_\infty=1$,
\item[(iii)] $d_1$, $d_2=ag$, $d_3=bg$ and $g$ are all even,
  $d_\infty=1$ and $a-b$ is odd,
\item[(iv)] $d_1$, $d_2=ag$, $d_3=bg$ and $g$ are all even,
  $d_\infty=1$ and $a-b$ is even.
\end{enumerate}

Cases (i)--(iii): To fill $M_{11}$, for $1\le i\le d_1$, let
\[
\C(i,t_2-i-\O_{1,i})=
\begin{cases}
t_2& \text{ for } 1 \leq i \leq d_2-g,\\
t_3& \text{ for } d_2-g< i \leq d_2+d_3-2g,\\
n& \text{ for $i=d_1$ if $d_1$ is even,}\\
t_1& \text{ otherwise}.
\end{cases}
\]
In addition to $\O_{1,i}$, define also offsets $\O_{j,i}$ for $j\in\{2,3\}$ by
\[
\O_{j,i}=
\begin{cases}
1&\text{if $d_j$ is even and $g<i\le g+\half d_j$,}\\
0&\text{otherwise.}
\end{cases}
\]
To fill $M_{12}\cup M_{21}$, for $1 \leq i \leq d_2$, let
\[
\C(d_2+1-i,d_1+i-\O_{2,i})=\C(t_3-i+\O_{2,i},d_1-d_2+i)=
\begin{cases}
t_3& \text{if }g< i \leq 2g,\\
t_1& \text{otherwise.}
\end{cases}
\]
To fill $M_{13}\cup M_{31}$, for $1 \leq i \leq d_3$, let
\[
\C(d_2-2g+i,t_4-i+\O_{3,i})=\C(t_3-1+i-\O_{3,i},d_1-d_2+2g+1-i)=
\begin{cases}
t_2& \text{if }g< i \leq 2g,\\
t_1& \text{otherwise.}
\end{cases}
\]
This partial contour is illustrated in Figure \ref{Fg:15} with $d_1=15$,
$d_2=5$, $d_3=3$ for case (i), and in Figure \ref{Fg:24} with
$d_1=24$, $d_2=8$, $d_3=6$ for case (iii).

\begin{figure}[htb]
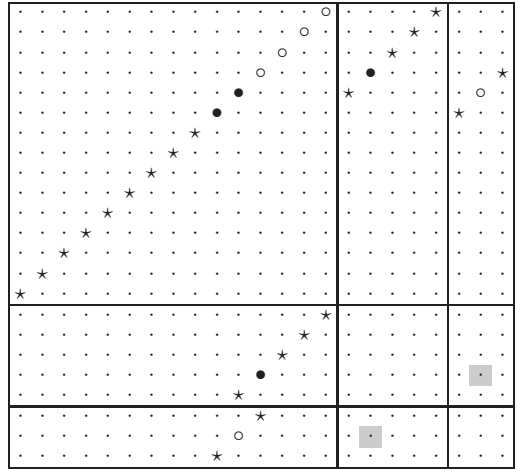

\verynarrowcols
\begin{displaymath}
\begin{array}{|ccccccccccccccc|ccccc|ccc|}
    \hline
    \xxx{14}&\xb&\xxx{4}&\xa&\xxx{3}\\
    \xxx{13}&\xb&\xxx{4}&\xa&\xxx{4}\\
    \xxx{12}&\xb&\xxx{4}&\xa&\xxx{5}\\
    \xxx{11}&\xb&\xxx{4}&\xc&\xxx{5}&\xa\\
    \xxx{10}&\xc&\xxx{4}&\xa&\xxx{5}&\xb&\xx\\
    \xxx{9}&\xc&\xxx{10}&\xa&\xxx{2}\\
    \xxx{8}&\xa&\xxx{14}\\
    \xxx{7}&\xa&\xxx{15}\\
    \xxx{6}&\xa&\xxx{16}\\
    \xxx{5}&\xa&\xxx{17}\\
    \xxx{4}&\xa&\xxx{18}\\
    \xxx{3}&\xa&\xxx{19}\\
    \xxx{2}&\xa&\xxx{20}\\
    \xx&\xa&\xxx{21}\\
    \xa&\xxx{22}\\
    \hline
    \xxx{14}&\xa&\xxx{8}\\
    \xxx{13}&\xa&\xxx{9}\\
    \xxx{12}&\xa&\xxx{10}\\
    \xxx{11}&\xc&\xxx{9}&\c\xx&\xx\\
    \xxx{10}&\xa&\xxx{12}\\
    \hline
    \xxx{11}&\xa&\xxx{11}\\
    \xxx{10}&\xb&\xxx{5}&\c\xx&\xxx{6}\\
    \xxx{9}&\xa&\xxx{13}\\
    \hline
\end{array}
\end{displaymath}
\verynormalcols
\caption{\label{Fg:15}A partial contour for
$d_1=15$, $d_2=5$, $d_3=3$, $d_\infty=0$.}
\end{figure}

We claim that blocks $M_{23}$ and $M_{32}$ can now be completed by
Lemma~\ref{LMgapsuff}. Let $N_{23}$ be the $g\times g$ submatrix of
$M_{23}$ formed by the rows of $M_{21}$ not containing the leading
symbol $t_1$, and by the columns of $M_{13}$ not containing the
leading symbol $t_1$. Define similarly the $g\times g$ submatrix
$N_{32}$ of $M_{32}$. These two submatrices are shaded gray in Figures
\ref{Fg:15} and \ref{Fg:24}.

In case (i), $N_{23}$ and $N_{32}$ consist of $g$ consecutive rows and
columns. We can represent this situation with parameters $e=g-1$,
$h_1=0$, $h_2=0$, as explained after Lemma~\ref{LMgapsuff}. Since $g$
is odd, Lemma~\ref{LMgapnec} implies that $\delta_g\equiv 0\pmod g$,
and \eqref{Eq:gapstuff} becomes $(h_1-h_2)(g-1)\equiv 0\pmod g$, which
is obviously satisfied.

In case (ii), let us first assume that $d_2$ is odd and $d_3$ is
even. The submatrix $N_{23}$ can be represented with parameters
$e=g-1$, $h_1=0$ (because $d_2$ is odd and there are no offsets) and
$h_2=d_3/2-g$. Since $g$ is odd, we again need $(h_1-h_2)(g-1)\equiv
0\pmod g$, which becomes $d_3/2\equiv 0\pmod g$, or $bg/2\equiv 0\pmod
g$, which is equivalent to $b$ being even. This is true, since $g$ is
odd and $d_3=bg$ is even. Similarly for the submatrix $N_{32}$. The
case $d_2$ even, $d_3$ odd is analogous.

In case (iii), the submatrix $N_{23}$ can be represented with
parameters $e=g-1$, $h_1=d_2/2-g$, $h_2=d_3/2-g$. Since $g$ is even,
equation \eqref{Eq:gapstuff} becomes $(h_1-h_2)e\equiv g/2\pmod g$ by
Lemma \ref{LMgapnec}, i.e., $(d_2/2-d_3/2)\equiv g/2\pmod g$. This
holds precisely when $a-b$ is odd, which is one of the assumptions of
case (iii). Similarly for $N_{32}$.

In all cases (i)--(iii), the remaining blocks can be filled in using
Lemmas~\ref{LMinfrowcol} and \ref{LMsubsq}.

\begin{figure}[htb]
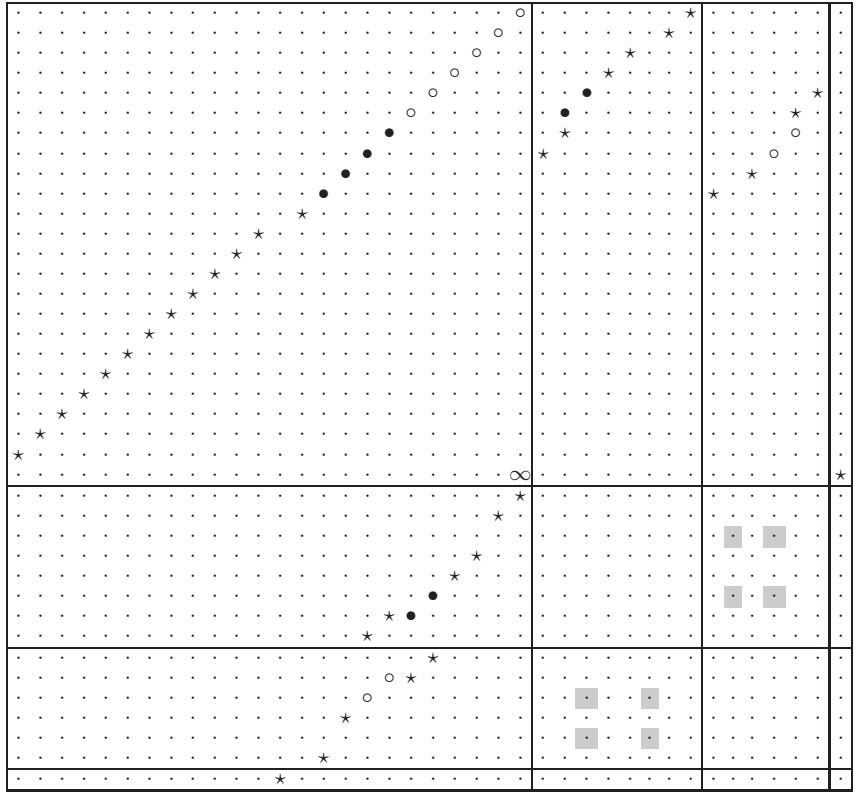

\verynarrowcols
\begin{displaymath}
\begin{array}{|cccccccccccccccccccccccc|cccccccc|cccccc|c|}
    \hline
    \xxx{23}&\xb&\xxx{7}&\xa&\xxx{7}\\
    \xxx{22}&\xb&\xxx{7}&\xa&\xxx{8}\\
    \xxx{21}&\xb&\xxx{6}&\xa&\xxx{10}\\
    \xxx{20}&\xb&\xxx{6}&\xa&\xxx{11}\\
    \xxx{19}&\xb&\xxx{6}&\xc&\xxx{10}&\xa&\xx\\
    \xxx{18}&\xb&\xxx{6}&\xc&\xxx{10}&\xa&\xxx{2}\\
    \xxx{17}&\xc&\xxx{7}&\xa&\xxx{10}&\xb&\xxx{2}\\
    \xxx{16}&\xc&\xxx{7}&\xa&\xxx{10}&\xb&\xxx{3}\\
    \xxx{15}&\xc&\xxx{18}&\xa&\xxx{4}\\
    \xxx{14}&\xc&\xxx{17}&\xa&\xxx{6}\\
    \xxx{13}&\xa&\xxx{25}\\
    \xxx{11}&\xa&\xxx{27}\\
    \xxx{10}&\xa&\xxx{28}\\
    \xxx{9}&\xa&\xxx{29}\\
    \xxx{8}&\xa&\xxx{30}\\
    \xxx{7}&\xa&\xxx{31}\\
    \xxx{6}&\xa&\xxx{32}\\
    \xxx{5}&\xa&\xxx{33}\\
    \xxx{4}&\xa&\xxx{34}\\
    \xxx{3}&\xa&\xxx{35}\\
    \xxx{2}&\xa&\xxx{36}\\
    \xx&\xa&\xxx{37}\\
    \xa&\xxx{38}\\
    \xxx{23}&\xi&\xxx{14}&\xa\\
    \hline
    \xxx{23}&\xa&\xxx{15}\\
    \xxx{22}&\xa&\xxx{16}\\
    \xxx{33}&\c\xx&\xx&\c\xx&\xxx{3}\\
    \xxx{21}&\xa&\xxx{17}\\
    \xxx{20}&\xa&\xxx{18}\\
    \xxx{19}&\xc&\xxx{13}&\c\xx&\xx&\c\xx&\xxx{3}\\
    \xxx{17}&\xa&\xc&\xxx{20}\\
    \xxx{16}&\xa&\xxx{22}\\
    \hline
    \xxx{19}&\xa&\xxx{19}\\
    \xxx{17}&\xb&\xa&\xxx{20}\\
    \xxx{16}&\xb&\xxx{9}&\c\xx&\xxx{2}&\c\xx&\xxx{9}\\
    \xxx{15}&\xa&\xxx{23}\\
    \xxx{26}&\c\xx&\xxx{2}&\c\xx&\xxx{9}\\
    \xxx{14}&\xa&\xxx{24}\\
    \hline
    \xxx{12}&\xa&\xxx{26}\\
    \hline
\end{array}
\end{displaymath}
\verynormalcols
\caption{\label{Fg:24}A partial contour for
$d_1=24$, $d_2=8$, $d_3=6$, $d_\infty=1$.}
\end{figure}

Case (iv). Since $a-b$ is even and $\gcd(a,b)=1$, both $a$ and $b$ are
odd. In particular, $d_1/d_3=a$ is odd, and $d_1/2$ is an odd multiple
of $d_3/2$.

\begin{figure}[htb]
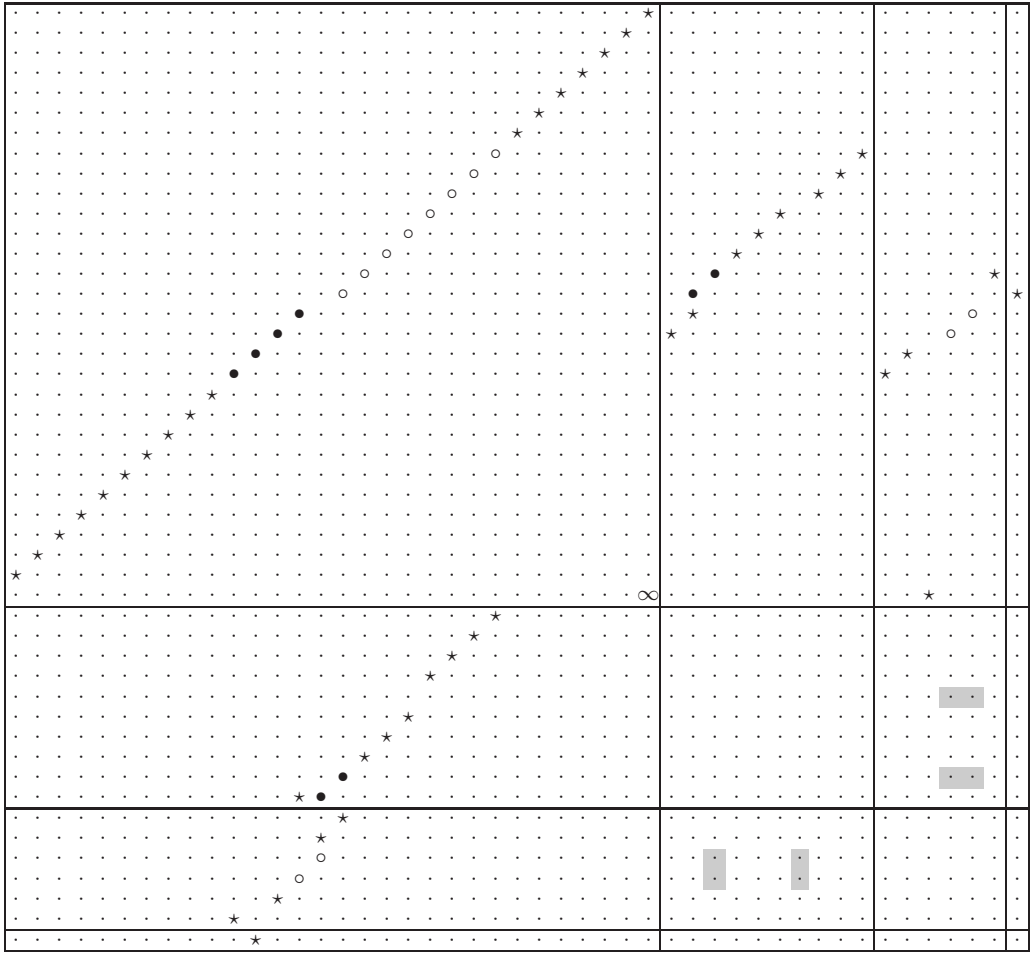

\begin{displaymath}
\verynarrowcols
\begin{array}{|cccccccccccccccccccccccccccccc|cccccccccc|cccccc|c|}
    \hline
    \xxx{29}&\xa&\xxx{17}\\
    \xxx{28}&\xa&\xxx{18}\\
    \xxx{27}&\xa&\xxx{19}\\
    \xxx{26}&\xa&\xxx{20}\\
    \xxx{25}&\xa&\xxx{21}\\
    \xxx{24}&\xa&\xxx{22}\\
    \xxx{23}&\xa&\xxx{23}\\
    \xxx{22}&\xb&\xxx{16}&\xa&\xxx{7}\\
    \xxx{21}&\xb&\xxx{16}&\xa&\xxx{8}\\
    \xxx{20}&\xb&\xxx{16}&\xa&\xxx{9}\\
    \xxx{19}&\xb&\xxx{15}&\xa&\xxx{11}\\
    \xxx{18}&\xb&\xxx{15}&\xa&\xxx{12}\\
    \xxx{17}&\xb&\xxx{15}&\xa&\xxx{13}\\
    \xxx{16}&\xb&\xxx{15}&\xc&\xxx{12}&\xa&\xx\\
    \xxx{15}&\xb&\xxx{15}&\xc&\xxx{14}&\xa\\
    \xxx{13}&\xc&\xxx{17}&\xa&\xxx{12}&\xb&\xxx{2}\\
    \xxx{12}&\xc&\xxx{17}&\xa&\xxx{12}&\xb&\xxx{3}\\
    \xxx{11}&\xc&\xxx{29}&\xa&\xxx{5}\\
    \xxx{10}&\xc&\xxx{29}&\xa&\xxx{6}\\
    \xxx{9}&\xa&\xxx{37}\\
    \xxx{8}&\xa&\xxx{38}\\
    \xxx{7}&\xa&\xxx{39}\\
    \xxx{6}&\xa&\xxx{40}\\
    \xxx{5}&\xa&\xxx{41}\\
    \xxx{4}&\xa&\xxx{42}\\
    \xxx{3}&\xa&\xxx{43}\\
    \xxx{2}&\xa&\xxx{44}\\
    \xxx{1}&\xa&\xxx{45}\\
    \xa&\xxx{46}\\
    \xxx{29}&\xi&\xxx{12}&\xa&\xxx{4}\\
    \hline
    \xxx{22}&\xa&\xxx{24}\\
    \xxx{21}&\xa&\xxx{25}\\
    \xxx{20}&\xa&\xxx{26}\\
    \xxx{19}&\xa&\xxx{27}\\
    \xxx{43}&\c\xx&\c\xx&\xxx{2}\\
    \xxx{18}&\xa&\xxx{28}\\
    \xxx{17}&\xa&\xxx{29}\\
    \xxx{16}&\xa&\xxx{30}\\
    \xxx{15}&\xc&\xxx{27}&\c\xx&\c\xx&\xxx{2}\\
    \xxx{13}&\xa&\xc&\xxx{32}\\
    \hline
    \xxx{15}&\xa&\xxx{31}\\
    \xxx{14}&\xa&\xxx{32}\\
    \xxx{14}&\xb&\xxx{17}&\c\xx&\xxx{3}&\c\xx&\xxx{10}\\
    \xxx{13}&\xb&\xxx{18}&\c\xx&\xxx{3}&\c\xx&\xxx{10}\\
    \xxx{12}&\xa&\xxx{34}\\
    \xxx{10}&\xa&\xxx{36}\\
    \hline
    \xxx{11}&\xa&\xxx{35}\\
    \hline
\end{array}
\verynormalcols
\end{displaymath}
\caption{\label{FIGBig}A partial contour for
$d_1=30$, $d_2=10$, $d_3=6$, $d_\infty=1$.}
\end{figure}

The construction (illustrated for $d_1=30$, $d_2=10$, $d_3=6$ and
$d_\infty=1$ in Figure \ref{FIGBig}), is similar to the one above, but
with several modifications, which we describe in words. First, we move
the leading symbols (but not the selected cells) in $M_{11}$ so that
the break in the diagonal occurs precisely between the leading symbols
$t_2$ and $t_3$. We fill $M_{12}$ as above, with the appropriate
vertical shift (which we will not mention any more). We fill $M_{21}$
as above, except that we shift the pattern down and left by one---this
is possible since we have an extra column to work with, the break
between symbols $t_2$ and $t_3$ in $M_{11}$. The block $M_{31}$ is
filled with the pattern from $M_{13}$ above, making sure that the
leading symbols $t_2$ in $M_{11}$ and $M_{13}$ occupy consecutive
columns. Finally, the block $M_{13}$ is filled as above, except that
the entry for $i=g$ is first moved down and to the left by $d_3/2$
(landing in row $d_1/2+d_3/2$), and then further down to row
$d_1$---since, as explained above, $d_1/2$ is an odd multiple of
$d_3/2$, the last vertical move is a multiple of $d_3$ and will
therefore not produce a clash with \lref{LMContour}.

Define the submatrices $N_{23}$ and $N_{32}$ as above. Then $N_{23}$
can be represented by the parameters $e=g-1$, $h_1=d_2/2-g$,
$h_2=0$. Equation \eqref{Eq:gapstuff} becomes $d_2/2\equiv g/2\pmod
g$, which holds because $a$ is odd. Similarly for $N_{32}$. The
remaining blocks can be filled in using Lemmas~\ref{LMinfrowcol} and
\ref{LMsubsq}.
\end{proof}

Finally, we treat the case of three distinct cycles when $d_3$ divides $d_2$.

\begin{prop}\label{P:ThreeCyclesB}
Suppose that $\alpha\in S_n$ has precisely three nontrivial cycles of
lengths $d_1>d_2>d_3$ where $d_3$ divides $d_2$.
Let $d_\infty$ be the number of fixed points of
$\alpha$. Then $\alpha\in\automorphs{n}$ if and only if
(a) $d_2$ divides $d_1$, (b) $d_3\ge d_\infty$, and
(c) if $d_3$ is even then $d_\infty>0$.
\end{prop}

\begin{proof}
We claim that the block diagram of $L$ must be as in
\fref{FILStruct3CycB}. The blocks $M_{i\infty}$ and $M_{\infty i}$,
for $1\le i\le3$, are forced by \lref{LMinfrowcol}.
Let $K=\bigcup_{i,j>1}M_{ij}$.  \tref{THStrongLCM}, with $\Lambda$
as the set of divisors of $d_2$, implies $K$ is a subsquare that 
obeys \tref{THTwoCycles}.
In particular, its block diagram
mirrors \fref{FILStruct}. It is then easy to complete the rest
of the diagram in \fref{FILStruct3CycB}.


\begin{figure}[htb]
\narrowcols
\begin{displaymath}
\begin{array}{l|l|l|l|l|}
    &\ \alpha_1&\ \alpha_2&\ \alpha_3&\ \alpha_\infty\\
    \hline
    \alpha_1&
        \begin{array}{l} \alpha_1:d_1-d_2-d_3-d_\infty\\
\alpha_2:d_1\\ \alpha_3:d_1\\ \alpha_\infty: d_1 \end{array}&
        \begin{array}{l} \alpha_1:d_2\\ \phantom{x}\\ \phantom{x}\\
\phantom{x} \end{array}&
        \begin{array}{l} \alpha_1:d_3\\ \phantom{x}\\ \phantom{x}\\
\phantom{x} \end{array}&
        \begin{array}{l} \alpha_1:d_\infty\\ \phantom{x}\\
\phantom{x}\\ \phantom{x} \end{array}\\
    \hline
    \alpha_2&
        \begin{array}{l} \alpha_1:d_2\\ \phantom{x}\\ \phantom{x}\end{array}&
        \begin{array}{l} \alpha_2:d_2-d_3-d_\infty\\ \alpha_3:d_2\\
\alpha_\infty:d_2\end{array}&
        \begin{array}{l} \alpha_2:d_3\\ \phantom{x}\\ \phantom{x} \end{array}&
        \begin{array}{l} \alpha_2:d_\infty\\ \phantom{x}\\ \phantom{x}
\end{array}\\
    \hline
    \alpha_3&
        \begin{array}{l} \alpha_1:d_3\\ \phantom{x} \end{array}&
        \begin{array}{l} \alpha_2:d_3\\ \phantom{x} \end{array}&
        \begin{array}{l} \alpha_3:d_3-d_\infty\\ \alpha_\infty:d_3 \end{array}&
        \begin{array}{l} \alpha_3:d_\infty\\ \phantom{x} \end{array}\\
    \hline
    \alpha_\infty&
        \begin{array}{l} \alpha_1:d_\infty \end{array}&
        \begin{array}{l} \alpha_2:d_\infty \end{array}&
        \begin{array}{l} \alpha_3:d_\infty \end{array}&
        \begin{array}{l} \alpha_\infty:d_\infty \end{array}\\
    \hline
\end{array}
\end{displaymath}
\normalcols
\caption{\label{FILStruct3CycB}Block diagram of $L$ with $d_1>d_2>d_3$
the only nontrivial cycle lengths, where $d_3$ divides $d_2$.}
\end{figure}

A consequence of $\alpha_1$ appearing in $M_{21}$ in the block diagram
of $L$, is that $d_2$ divides $d_1$, establishing (a). Since $K$ is a
subsquare, (b) and (c) follow from \tref{THTwoCycles}.

Conversely, if (a)--(c) hold then the subsquare $K$ can be constructed
by \tref{THTwoCycles}. It thus suffices to provide a partial contour for
$L\setminus K$. As usual, employing \lref{LMaddfix}, we may assume that
$d_\infty=1$ if $d_3$ is even and $d_\infty=0$ if $d_3$ is odd. 

We define $D$ and the partial contour for the block $M_{11}$
exactly as we did in Proposition~\ref{P:ThreeCycles2}.


Case I: $d_3$ is even (and $d_\infty=1$). Then $d_1$, $d_2$, $d_3$ are 
all even. There are $d_2+d_3+1$ consecutive columns in $D$ that do not
contain $t_1$. Utilizing these columns, we place an even pattern
filled with symbols $t_1$ into $M_{21}$ (occupying $d_2+1$ columns),
and ``wrap around it'' an even pattern filled with symbols $t_1$ in
$M_{31}$. Specifically, we define 
\begin{equation}\label{e:wrap}
\begin{array}{llll}
&\C(d_1+i,\half d_1+d_2-\half d_3+1-i)=t_1&&\text{for }1\le i\le\half d_2,\\[0.33ex]
&\C(d_1+i,\half d_1+d_2-\half d_3-i)=t_1&&\text{for }\half d_2<i\le d_2,\\[1.5ex]
&\C(d_1+d_2+i,\half d_1+d_2+1-i)=t_1&&\text{for }1\le i\le\half d_3,\\[0.33ex]
&\C(d_1+d_2+i,\half d_1-i)=t_1&\phantom{wider}&\text{for }\half d_3<i\le d_3.
\end{array}
\end{equation}
See \fref{FIGB} for examples. This wrap-around construction works here
because the gap in the even pattern in $M_{31}$ has size 
$d_2+1\equiv1\pmod{d_3}$.

Since there are only $d_2+d_3$ consecutive rows without symbols $t_1$
in $D$, we will use a modified wrap-around construction in $M_{12}\cup
M_{13}$. Namely, we transpose the partial contour in \eref{e:wrap} and
slide it vertically so that only the top symbol $t_1$ in $M_{13}$
collides with $M_{11}$. If $d_3$ divides $d_2/2$, we move this
colliding symbol down to row $d_1$, i.e., by $d_2+d_1/2$ rows, a
multiple of $d_3$. If $d_3$ does not divide $d_2/2$, we have
$d_2=(2k+1)d_3$ for some $k$, and we move the colliding symbol down to
the row corresponding to the gap in the even pattern in $M_{12}$,
i.e., by $d_2/2+d_3/2 = (k+1)d_3$ rows, again a multiple of
$d_3$. Figure \ref{FIGB} illustrates both possibilities, with the
moved colliding symbol highlighted.

\begin{figure}[htb]
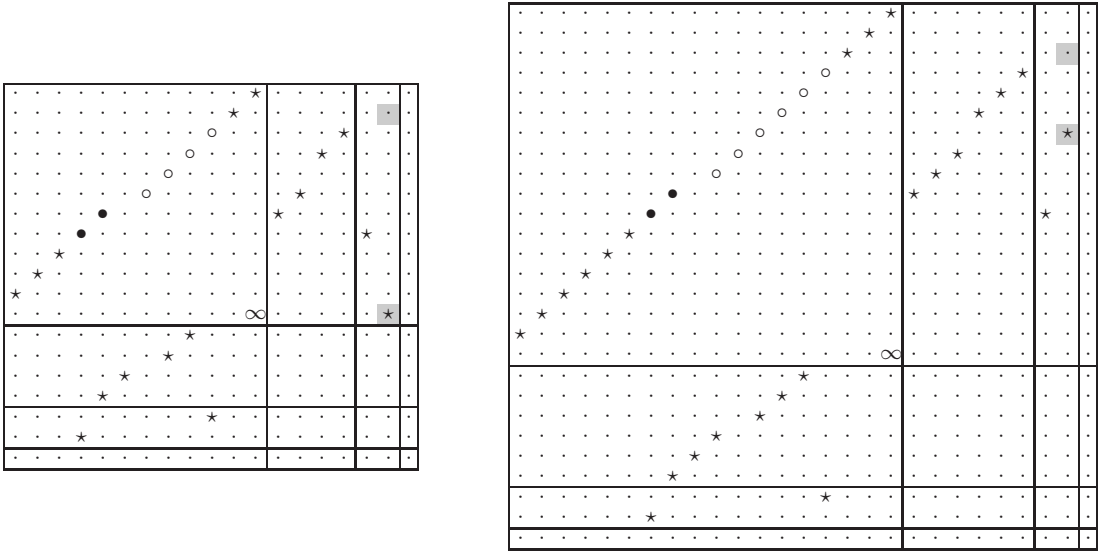

\begin{displaymath}
\verynarrowcols
\begin{array}{|cccccccccccc|cccc|cc|c|}
    \hline
    \xxx{11}&\xa&\xxx{7}\\
    \xxx{10}&\xa&\xxx{6}&\c\xx&\xx\\
    \xxx{9}&\xb&\xxx{5}&\xa&\xxx{3}\\
    \xxx{8}&\xb&\xxx{5}&\xa&\xxx{4}\\
    \xxx{7}&\xb&\xxx{11}\\
    \xxx{6}&\xb&\xxx{6}&\xa&\xxx{5}\\
    \xxx{4}&\xc&\xxx{7}&\xa&\xxx{6}\\
    \xxx{3}&\xc&\xxx{12}&\xa&\xxx{2}\\
    \xxx{2}&\xa&\xxx{16}\\
    \xx&\xa&\xxx{17}\\
    \xa&\xxx{18}\\
    \xxx{11}&\xi&\xxx{5}&\c\xa&\xx\\
    \hline
    \xxx{8}&\xa&\xxx{10}\\
    \xxx{7}&\xa&\xxx{11}\\
    \xxx{5}&\xa&\xxx{13}\\
    \xxx{4}&\xa&\xxx{14}\\
    \hline
    \xxx{9}&\xa&\xxx{9}\\
    \xxx{3}&\xa&\xxx{15}\\
    \hline
    \xxx{19}\\
    \hline
\end{array}
\qquad\qquad
\begin{array}{|cccccccccccccccccc|cccccc|cc|c|}
    \hline
    \xxx{17}&\xa&\xxx{9}\\
    \xxx{16}&\xa&\xxx{10}\\
    \xxx{15}&\xa&\xxx{9}&\c\xx&\xx\\
    \xxx{14}&\xb&\xxx{8}&\xa&\xxx{3}\\
    \xxx{13}&\xb&\xxx{8}&\xa&\xxx{4}\\
    \xxx{12}&\xb&\xxx{8}&\xa&\xxx{5}\\
    \xxx{11}&\xb&\xxx{13}&\c\xa&\xx\\
    \xxx{10}&\xb&\xxx{9}&\xa&\xxx{6}\\
    \xxx{9}&\xb&\xxx{9}&\xa&\xxx{7}\\
    \xxx{7}&\xc&\xxx{10}&\xa&\xxx{8}\\
    \xxx{6}&\xc&\xxx{17}&\xa&\xxx{2}\\
    \xxx{5}&\xa&\xxx{21}\\
    \xxx{4}&\xa&\xxx{22}\\
    \xxx{3}&\xa&\xxx{23}\\
    \xxx{2}&\xa&\xxx{24}\\
    \xx&\xa&\xxx{25}\\
    \xa&\xxx{26}\\
    \xxx{17}&\xi&\xxx{9}\\
    \hline
    \xxx{13}&\xa&\xxx{13}\\
    \xxx{12}&\xa&\xxx{14}\\
    \xxx{11}&\xa&\xxx{15}\\
    \xxx{9}&\xa&\xxx{17}\\
    \xxx{8}&\xa&\xxx{18}\\
    \xxx{7}&\xa&\xxx{19}\\
    \hline
    \xxx{14}&\xa&\xxx{12}\\
    \xxx{6}&\xa&\xxx{20}\\
    \hline
    \xxx{27}\\
    \hline
\end{array}
\verynormalcols
\end{displaymath}
\caption{\label{FIGB}Partial contours for
$d_1=12$, $d_2=4$, $d_3=2$, $d_\infty=1$, 
and $d_1=18$, $d_2=6$, $d_3=2$, $d_\infty=1$.}
\end{figure}

Case II: $d_3$ is odd (and $d_\infty=0$). If $d_1$ is odd, it is
straightforward to align odd patterns in $M_{21}\cup M_{31}$ and
$M_{12}\cup M_{13}$ with, respectively, the columns and rows of $D$
that do not contain $t_1$. We can therefore assume that $d_1$ is even.

For now, assume $d_2$ is even; an example of a
partial contour in this case is given in \fref{FIGB2}. Then $d_3$
divides $d_2/2$. We place an even pattern into the $d_2+1$ right-most
available columns of $M_{21}$, and we position an odd pattern
immediately to the left of it in $M_{31}$. Since there are now only
$d_2+d_3$ columns without $t_1$ in $D$, the left-most symbol in the
partial contour of $M_{31}$ collides with $M_{11}$, and we move it to
the column corresponding to the break in the even pattern of $M_{21}$,
i.e., by $d_3+d_2/2$ columns, a multiple of $d_3$. Now we place an
even pattern into the $d_2+1$ top-most available rows of $M_{12}$, and
we position an odd pattern into $M_{13}$ so that the top symbol $t_1$
in $M_{13}$ collides with the bottom symbol $t_1$ in $M_{12}$. This
colliding element can be moved into the row corresponding to the gap
in the even pattern in $M_{12}$, i.e., by $d_2/2$ rows, a multiple of
$d_3$. However, the bottom symbol $t_1$ in $M_{13}$ still collides
with $M_{11}$, and we move it to row $d_1$, i.e., by $d_1/2-d_3$ rows,
a multiple of $d_3$.

Finally, suppose that $d_2$ is odd. Since we can place odd patterns into
$M_{12}$, $M_{13}$, $M_{21}$, $M_{31}$, the blocks $M_{21}\cup M_{31}$
can be filled easily, with $d_2+d_3$ columns of $D$ at our
disposal. Place the odd pattern in $M_{12}$ as high as possible, and
the odd pattern in $M_{13}$ immediately below it, so that only the
bottom symbol $t_1$ in $M_{13}$ collides with $M_{11}$. This colliding
element can again be moved to row $d_1$, i.e., by $d_1/2-d_3$ rows, a
multiple of $d_3$.
\end{proof}

\begin{figure}[htb]
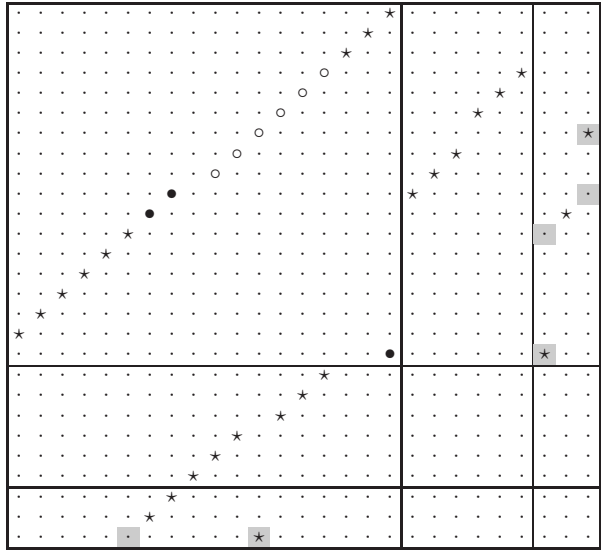

\begin{displaymath}
\verynarrowcols
\begin{array}{|cccccccccccccccccc|cccccc|ccc|}
    \hline
    \xxx{17}&\xa&\xxx{9}\\
    \xxx{16}&\xa&\xxx{10}\\
    \xxx{15}&\xa&\xxx{11}\\
    \xxx{14}&\xb&\xxx{8}&\xa&\xxx{3}\\
    \xxx{13}&\xb&\xxx{8}&\xa&\xxx{4}\\
    \xxx{12}&\xb&\xxx{8}&\xa&\xxx{5}\\
    \xxx{11}&\xb&\xxx{14}&\c\xa\\
    \xxx{10}&\xb&\xxx{9}&\xa&\xxx{6}\\
    \xxx{9}&\xb&\xxx{9}&\xa&\xxx{7}\\
    \xxx{7}&\xc&\xxx{10}&\xa&\xxx{7}&\c\xx\\
    \xxx{6}&\xc&\xxx{18}&\xa&\xx\\
    \xxx{5}&\xa&\xxx{18}&\c\xx&\xxx{2}\\
    \xxx{4}&\xa&\xxx{22}\\
    \xxx{3}&\xa&\xxx{23}\\
    \xxx{2}&\xa&\xxx{24}\\
    \xx&\xa&\xxx{25}\\
    \xa&\xxx{26}\\
    \xxx{17}&\xc&\xxx{6}&\c\xa&\xxx{2}\\
    \hline
    \xxx{14}&\xa&\xxx{12}\\
    \xxx{13}&\xa&\xxx{13}\\
    \xxx{12}&\xa&\xxx{14}\\
    \xxx{10}&\xa&\xxx{16}\\
    \xxx{9}&\xa&\xxx{17}\\
    \xxx{8}&\xa&\xxx{18}\\
    \hline
    \xxx{7}&\xa&\xxx{19}\\
    \xxx{6}&\xa&\xxx{20}\\
    \xxx{5}&\c\xx&\xxx{5}&\c\xa&\xxx{15}\\
    \hline
\end{array}
\verynormalcols
\end{displaymath}
\caption{\label{FIGB2}Partial contour for $d_1=18$, $d_2=6$, $d_3=3$, $d_\infty=0$.}
\end{figure}

\section{Autotopisms of small Latin squares}\label{Sc:Computational}

Falc\'{o}n \cite{Falcon2009} identified $\autotops{n}$ for $n\le
11$. The elements of $\autotops{n}$ for $n \leq 17$ are given in
Appendix~\ref{APCycStr}.  A representative Latin square that admits
each claimed autotopism is given at \cite{WWWW},
along with the \verb!GAP! code \cite{GAP} used in this project.

In \cite{Falcon2009}, Falc\'{o}n listed six isotopisms $\theta$, that
are not equivalent in the sense of Lemma~\ref{LMAutConj}, for which he
proved computationally that $\theta\notin\autotops{n}$ but no
theoretical reason was known.  Five of these cases are resolved
theoretically by Corollary~\ref{COSmithKerby}. The remaining case has
cycle structure $(4\tdot2,4\tdot2,4\tdot1^2)$. It is simple to check
by hand that such an autotopism is not possible, although for reasons
that seem peculiar to this example. Applying Theorem~\ref{THStrongLCM}
we see that there is also no autotopism in $\autotops{14}$ with cycle
structure $(8\tdot4\tdot2,8\tdot4\tdot2,8\tdot4\tdot1^2)$. We also
observe that the example after \lref{LMLCMPerm} shows that there is no
automorphism in $\automorphs{17}$ with cycle structure
$6\tdot3\tdot2^4$.

With the exception of the special cases just discussed, every
isotopism $\theta \in \I_n$ for $n \leq 17$ either belongs to
$\autotops{n}$ or can be shown to have $\theta\notin\autotops{n}$
using Lemma~\ref{LMLcmAB}, Theorem~\ref{THStrongLCM} or
Corollary~\ref{COSmithKerby}.


\section{Concluding remarks}

We conclude this paper with some future research ideas. While it is
known that the probability that a random Latin square admits a
nontrivial autotopism is asymptotically zero \cite{McKayWanless2005},
we propose the following conjecture.

\begin{conj}
For $n>0$ let $\mathbb{P}(n)$ be the probability that a randomly
chosen $\alpha\in S_n$ is a component of some isotopism
$(\alpha,\beta,\gamma)\in\autotops{n}$. Then $\lim_{n\to\infty}
\mathbb{P}(n)=0$.
\end{conj}

Motivated by the results of Falc\'on \cite{Falcon2009}, we have
verified computationally that the following question has an
affirmative answer for all primes $p\le 23$.

\begin{prob}\label{PRPrimeNCycStr}
Let $\theta=(\alpha$, $\beta$, $\gamma) \in \autotops{p}$ for some
prime $p$.  Must it be true that either $\theta$ is equivalent to
$(\delta$, $\delta$, $\id)$ where $\delta$ is a $p$-cycle, or that
$\alpha$, $\beta$ and $\gamma$ all have the same cycle structure?
\end{prob}

Horo\v{s}evski\u{\i} proved \cite[Theorem 2]{Hor} that if $G$ is a
group of order $n>1$ and $\varphi$ is an automorphism of $G$ then the
order of $\varphi$ cannot exceed $n-1$. Motivated by \cite{Hor} and by
our computational results, we ask:

\begin{prob}
Suppose $\theta$ is an autotopism of a Latin square $L$ of order $n$.
Is the order of $\theta$ at most $n$?
\end{prob}

\section*{Acknowledgements} The authors are grateful to Chris Mears,
who provided some assistance in finding Latin squares with a specified
autotopism (see \cite{Mears2009}), and to Michael Kinyon, who helped with the review of the
loop-theoretical literature concerned with autotopisms. The authors are
also grateful to the referees for their diligence and their useful
feedback.

  \let\oldthebibliography=\thebibliography
  \let\endoldthebibliography=\endthebibliography
  \renewenvironment{thebibliography}[1]{%
    \begin{oldthebibliography}{#1}%
      \setlength{\parskip}{0.4ex plus 0.1ex minus 0.1ex}%
      \setlength{\itemsep}{0.4ex plus 0.1ex minus 0.1ex}%
  }%
  {%
    \end{oldthebibliography}%
  }

\newpage

\appendix

\section{Autotopism cycle structures for orders up to 17}\label{APCycStr}

Appealing to Lemma~\ref{LMAutConj}, we only list cycle structures
$(a,b,c)$ of autotopisms $(\alpha,\beta,\gamma)$. For a given order
$n$, the first column gives the cycle structure of $\alpha$. In a
given row, the second column gives all possible cycle structures of
$\beta$ and $\gamma$, separated by commas. If $\beta$ and $\gamma$
have the same cycle structure, we only list the cycle structure of
$\beta$, else we give the cycle structures of $\beta$ and $\gamma$ as
an ordered pair in parentheses.

\[
\begin{array}{ccccc}
\verynarrowcols\begin{array}{|r|r|}
\hline
\tspacer n=1 & \\
\alpha & \beta\text{ and }\gamma\bspacer\\
\hline\tspacer
1 & 1 \\
\hline
\hline
\tspacer n=2 & \\
\alpha & \beta\text{ and }\gamma\bspacer\\
\hline\tspacer
1^2 & 1^2,2 \\
\hline
\hline
\tspacer n=3 & \\
\alpha & \beta\text{ and }\gamma\bspacer\\
\hline\tspacer
1^3 & 1^3,3 \\
2 \tdot 1 & 2 \tdot 1 \\
3 & 3 \\
\hline
\hline
\tspacer n=4 & \\
\alpha & \beta\text{ and }\gamma\bspacer\\
\hline\tspacer
1^4 & 1^4,2^2,4 \\
2 \tdot 1^2 & 2 \tdot 1^2,2^2,4 \\
2^2 & 2^2,4 \\
3 \tdot 1 & 3 \tdot 1 \\
\hline
\hline
\tspacer n=5 & \\
\alpha & \beta\text{ and }\gamma\bspacer\\
\hline\tspacer
1^5 & 1^5,5 \\
2^2 \tdot 1 & 2^2 \tdot 1 \\
3 \tdot 1^2 & 3 \tdot 1^2 \\
4 \tdot 1 & 4 \tdot 1 \\
5 & 5 \\
\hline
\hline
\tspacer n=6 & \\
\alpha & \beta\text{ and }\gamma\bspacer\\
\hline\tspacer
1^6 & 1^6,2^3,3^2,6 \\
2 \tdot 1^4 & 2^3,6 \\
2^2 \tdot 1^2 & 2^2 \tdot 1^2,2^3,6 \\
2^3 & (3^2,6) \\
3 \tdot 1^3 & 3 \tdot 1^3,3^2,6 \\
3 \tdot 2 \tdot 1 & 6 \\
3^2 & 3^2,6 \\
4 \tdot 1^2 & 4 \tdot 1^2 \\
5 \tdot 1 & 5 \tdot 1 \\
\hline
\end{array}\verynormalcols

&&

\verynarrowcols\begin{array}{|r|r|}
\hline
\tspacer n=7 & \\
\alpha & \beta\text{ and }\gamma\bspacer\\
\hline\tspacer
1^7 & 1^7,7 \\
2^2 \tdot 1^3 & 2^2 \tdot 1^3 \\
2^3 \tdot 1 & 2^3 \tdot 1 \\
3^2 \tdot 1 & 3^2 \tdot 1 \\
4 \tdot 1^3 & 4 \tdot 1^3 \\
4 \tdot 2 \tdot 1 & 4 \tdot 2 \tdot 1 \\
5 \tdot 1^2 & 5 \tdot 1^2 \\
6 \tdot 1 & 6 \tdot 1 \\
7 & 7 \\
\hline
\hline
\tspacer n=8 & \\
\alpha & \beta\text{ and }\gamma\bspacer\\
\hline\tspacer
1^8 & 1^8,2^4,4^2,8 \\
2 \tdot 1^6 & 2^4,4^2,8 \\
2^2 \tdot 1^4 & 2^2 \tdot 1^4,2^4,4^2,8 \\
2^3 \tdot 1^2 & 2^3 \tdot 1^2,2^4,4^2,8 \\
2^4 & 2^4,4^2,8 \\
3^2 \tdot 1^2 & 3^2 \tdot 1^2,6 \tdot 2 \\
4 \tdot 1^4 & 4 \tdot 1^4,4 \tdot 2^2,4^2,8 \\
4 \tdot 2 \tdot 1^2 & 4 \tdot 2 \tdot 1^2,4 \tdot 2^2,4^2,8 \\
4 \tdot 2^2 & 4 \tdot 2^2,4^2,8 \\
4^2 & 4^2,8 \\
5 \tdot 1^3 & 5 \tdot 1^3 \\
6 \tdot 1^2 & 6 \tdot 1^2,6 \tdot 2 \\
7 \tdot 1 & 7 \tdot 1 \\
\hline
\hline
\tspacer n=9 & \\
\alpha & \beta\text{ and }\gamma\bspacer\\
\hline\tspacer
1^9 & 1^9,3^3,9 \\
2^3 \tdot 1^3 & 2^3 \tdot 1^3,6 \tdot 3 \\
2^4 \tdot 1 & 2^4 \tdot 1 \\
3 \tdot 1^6 & 3^3,9 \\
3 \tdot 2^3 & 6 \tdot 3 \\
3^2 \tdot 1^3 & 3^2 \tdot 1^3,3^3,9 \\
3^3 & 3^3,9 \\
4^2 \tdot 1 & 4^2 \tdot 1 \\
5 \tdot 1^4 & 5 \tdot 1^4 \\
6 \tdot 1^3 & 6 \tdot 1^3,6 \tdot 3 \\
6 \tdot 2 \tdot 1 & 6 \tdot 2 \tdot 1 \\
6 \tdot 3 & 6 \tdot 3 \\
7 \tdot 1^2 & 7 \tdot 1^2 \\
8 \tdot 1 & 8 \tdot 1 \\
9 & 9 \\
\hline
\end{array}\verynormalcols

&&

\verynarrowcols\begin{array}{|r|r|}
\hline
\tspacer n=10 & \\
\alpha & \beta\text{ and }\gamma\bspacer\\
\hline\tspacer
1^{10} & 1^{10},2^5,5^2,10 \\
2 \tdot 1^8 & 2^5,10 \\
2^2 \tdot 1^6 & 2^5,10 \\
2^3 \tdot 1^4 & 2^3 \tdot 1^4,2^5,10 \\
2^4 \tdot 1^2 & 2^4 \tdot 1^2,2^5,10 \\
2^5 & (5^2,10) \\
3^2 \tdot 1^4 & 3^2 \tdot 1^4,6 \tdot 2^2 \\
3^2 \tdot 2 \tdot 1^2 & 6 \tdot 2^2 \\
3^2 \tdot 2^2 & 6 \tdot 2^2 \\
3^3 \tdot 1 & 3^3 \tdot 1 \\
4^2 \tdot 1^2 & 4^2 \tdot 1^2,4^2 \tdot 2 \\
5 \tdot 1^5 & 5 \tdot 1^5,5^2,10 \\
5 \tdot 2 \tdot 1^3 & 10 \\
5 \tdot 2^2 \tdot 1 & 10 \\
5^2 & 5^2,10 \\
6 \tdot 1^4 & 6 \tdot 1^4,6 \tdot 2^2 \\
6 \tdot 2 \tdot 1^2 & 6 \tdot 2 \tdot 1^2,6 \tdot 2^2 \\
6 \tdot 3 \tdot 1 & 6 \tdot 3 \tdot 1 \\
7 \tdot 1^3 & 7 \tdot 1^3 \\
8 \tdot 1^2 & 8 \tdot 1^2,8 \tdot 2 \\
9 \tdot 1 & 9 \tdot 1 \\
\hline
\hline
\tspacer n=11 & \\
\alpha & \beta\text{ and }\gamma\bspacer\\
\hline\tspacer
1^{11} & 1^{11},11 \\
2^3 \tdot 1^5 & 2^3 \tdot 1^5 \\
2^4 \tdot 1^3 & 2^4 \tdot 1^3 \\
2^5 \tdot 1 & 2^5 \tdot 1 \\
3^2 \tdot 1^5 & 3^2 \tdot 1^5 \\
3^3 \tdot 1^2 & 3^3 \tdot 1^2 \\
4^2 \tdot 1^3 & 4^2 \tdot 1^3 \\
4^2 \tdot 2 \tdot 1 & 4^2 \tdot 2 \tdot 1 \\
5^2 \tdot 1 & 5^2 \tdot 1 \\
6 \tdot 1^5 & 6 \tdot 1^5 \\
6 \tdot 2^2 \tdot 1 & 6 \tdot 2^2 \tdot 1 \\
6 \tdot 3 \tdot 1^2 & 6 \tdot 3 \tdot 1^2 \\
7 \tdot 1^4 & 7 \tdot 1^4 \\
8 \tdot 1^3 & 8 \tdot 1^3 \\
8 \tdot 2 \tdot 1 & 8 \tdot 2 \tdot 1 \\
9 \tdot 1^2 & 9 \tdot 1^2 \\
10 \tdot 1 & 10 \tdot 1 \\
11 & 11 \\
\hline
\end{array}\verynormalcols

\end{array}
\]

\[
\begin{array}{ccccc}

\verynarrowcols\begin{array}{|r|r|}
\hline
\tspacer n=12 & \\
\alpha & \beta\text{ and }\gamma\bspacer\\
\hline\tspacer
1^{12} & 1^{12},2^6,3^4,4^3,6^2,12 \\
2 \tdot 1^{10} & 2^6,4^3,6^2,12 \\
2^2 \tdot 1^8 & 2^6,4^3,6^2,12 \\
2^3 \tdot 1^6 & 2^3 \tdot 1^6,2^6,4^3,6 \tdot 3^2,6^2,12 \\
2^4 \tdot 1^4 & 2^4 \tdot 1^4,2^6,4^3,6^2,12 \\
2^5 \tdot 1^2 & 2^5 \tdot 1^2,2^6,4^3,6^2,12 \\
2^6 & 2^6,(3^4,6^2),4^3,(6 \tdot 3^2,6^2),6^2,12 \\
3 \tdot 1^9 & 3^4,6^2,12 \\
3 \tdot 2 \tdot 1^7 & 6^2,12 \\
3 \tdot 2^2 \tdot 1^5 & 6^2,12 \\
3 \tdot 2^3 \tdot 1^3 & 6 \tdot 3^2,6^2,12 \\
3 \tdot 2^4 \tdot 1 & 6^2,12 \\
3^2 \tdot 1^6 & 3^2 \tdot 1^6,3^4,6 \tdot 2^3,6^2,12 \\
3^2 \tdot 2 \tdot 1^4 & 6 \tdot 2^3,6^2,12 \\
3^2 \tdot 2^2 \tdot 1^2 & 6 \tdot 2^3,6^2,12 \\
3^2 \tdot 2^3 & (3^2 \tdot 2^3,6 \tdot 1^6),6 \tdot 3^2,6^2,12 \\
3^3 \tdot 1^3 & 3^3 \tdot 1^3,3^4,6^2,12 \\
3^3 \tdot 2 \tdot 1 & 6^2,12 \\
3^4 & 3^4,(4^3,12),(6 \tdot 2^3,6^2),6^2,12 \\
4 \tdot 1^8 & 4^3,12 \\
4 \tdot 2 \tdot 1^6 & 4^3,12 \\
4 \tdot 2^2 \tdot 1^4 & 4^3,12 \\
4 \tdot 2^3 \tdot 1^2 & 4^3,12 \\
4 \tdot 2^4 & 4^3,12 \\
4 \tdot 3 \tdot 1^5 & 12 \\
4 \tdot 3 \tdot 2 \tdot 1^3 & 12 \\
4 \tdot 3 \tdot 2^2 \tdot 1 & 12 \\
4 \tdot 3^2 \tdot 1^2 & 12 \\
4 \tdot 3^2 \tdot 2 & 12 \\
4^2 \tdot 1^4 & 4^2 \tdot 1^4,4^2 \tdot 2^2,4^3,12 \\
4^2 \tdot 2 \tdot 1^2 & 4^2 \tdot 2 \tdot 1^2,4^2 \tdot 2^2,4^3,12 \\
4^2 \tdot 2^2 & 4^2 \tdot 2^2,4^3,12 \\
4^2 \tdot 3 \tdot 1 & 12 \\
4^3 & (6 \tdot 3^2,12),(6^2,12) \\
5^2 \tdot 1^2 & 5^2 \tdot 1^2,10 \tdot 2 \\
6 \tdot 1^6 & 6 \tdot 1^6,6 \tdot 3^2,6^2,12 \\
6 \tdot 2 \tdot 1^4 & 6^2,12 \\
6 \tdot 2^2 \tdot 1^2 & 6 \tdot 2^2 \tdot 1^2,6^2,12 \\
6 \tdot 2^3 & (6 \tdot 3^2,6^2),6^2,12 \\
6 \tdot 3 \tdot 1^3 & 6 \tdot 3 \tdot 1^3,6 \tdot 3^2,6^2,12 \\
6 \tdot 3 \tdot 2 \tdot 1 & 6 \tdot 3 \tdot 2 \tdot 1,6^2,12 \\
6 \tdot 3^2 & 6 \tdot 3^2,6^2,12 \\
6 \tdot 4 \tdot 1^2 & 12 \\
6 \tdot 4 \tdot 2 & 12 \\
6^2 & 6^2,12 \\
7 \tdot 1^5 & 7 \tdot 1^5 \\
8 \tdot 1^4 & 8 \tdot 1^4,8 \tdot 2^2,8 \tdot 4 \\
8 \tdot 2 \tdot 1^2 & 8 \tdot 2 \tdot 1^2,8 \tdot 2^2,8 \tdot 4 \\
8 \tdot 2^2 & 8 \tdot 2^2,8 \tdot 4 \\
9 \tdot 1^3 & 9 \tdot 1^3,9 \tdot 3 \\
9 \tdot 3 & 9 \tdot 3 \\
10 \tdot 1^2 & 10 \tdot 1^2,10 \tdot 2 \\
11 \tdot 1 & 11 \tdot 1 \\
\hline
\end{array}\verynormalcols

&&

\verynarrowcols\begin{array}{|r|r|}
\hline
\tspacer n=13 & \\
\alpha & \beta\text{ and }\gamma\bspacer\\
\hline\tspacer
1^{13} & 1^{13},13 \\
2^4 \tdot 1^5 & 2^4 \tdot 1^5 \\
2^5 \tdot 1^3 & 2^5 \tdot 1^3 \\
2^6 \tdot 1 & 2^6 \tdot 1 \\
3^3 \tdot 1^4 & 3^3 \tdot 1^4 \\
3^4 \tdot 1 & 3^4 \tdot 1 \\
4^2 \tdot 1^5 & 4^2 \tdot 1^5 \\
4^2 \tdot 2^2 \tdot 1 & 4^2 \tdot 2^2 \tdot 1 \\
4^3 \tdot 1 & 4^3 \tdot 1 \\
5^2 \tdot 1^3 & 5^2 \tdot 1^3 \\
6 \tdot 3 \tdot 2 \tdot 1^2 & 6 \tdot 3 \tdot 2 \tdot 1^2 \\
6 \tdot 3 \tdot 2^2 & 6 \tdot 3 \tdot 2^2 \\
6^2 \tdot 1 & 6^2 \tdot 1 \\
7 \tdot 1^6 & 7 \tdot 1^6 \\
8 \tdot 1^5 & 8 \tdot 1^5 \\
8 \tdot 2^2 \tdot 1 & 8 \tdot 2^2 \tdot 1 \\
8 \tdot 4 \tdot 1 & 8 \tdot 4 \tdot 1 \\
9 \tdot 1^4 & 9 \tdot 1^4 \\
9 \tdot 3 \tdot 1 & 9 \tdot 3 \tdot 1 \\
10 \tdot 1^3 & 10 \tdot 1^3 \\
10 \tdot 2 \tdot 1 & 10 \tdot 2 \tdot 1 \\
11 \tdot 1^2 & 11 \tdot 1^2 \\
12 \tdot 1 & 12 \tdot 1 \\
13 & 13 \\
\hline
\end{array}\verynormalcols

&&

\verynarrowcols\begin{array}{|r|r|}
\hline
\tspacer n=14 & \\
\alpha & \beta\text{ and }\gamma\bspacer\\
\hline\tspacer
1^{14} & 1^{14},2^7,7^2,14 \\
2 \tdot 1^{12} & 2^7,14 \\
2^2 \tdot 1^{10} & 2^7,14 \\
2^3 \tdot 1^8 & 2^7,14 \\
2^4 \tdot 1^6 & 2^4 \tdot 1^6,2^7,14 \\
2^5 \tdot 1^4 & 2^5 \tdot 1^4,2^7,14 \\
2^6 \tdot 1^2 & 2^6 \tdot 1^2,2^7,14 \\
2^7 & (7^2,14) \\
3^3 \tdot 1^5 & 3^3 \tdot 1^5 \\
3^4 \tdot 1^2 & 3^4 \tdot 1^2,6^2 \tdot 2 \\
4^2 \tdot 1^6 & 4^2 \tdot 1^6,4^2 \tdot 2^3 \\
4^2 \tdot 2 \tdot 1^4 & 4^2 \tdot 2^3 \\
4^2 \tdot 2^2 \tdot 1^2 & 4^2 \tdot 2^2 \tdot 1^2,4^2 \tdot 2^3 \\
4^3 \tdot 1^2 & 4^3 \tdot 1^2,4^3 \tdot 2 \\
5^2 \tdot 1^4 & 5^2 \tdot 1^4,10 \tdot 2^2 \\
5^2 \tdot 2 \tdot 1^2 & 10 \tdot 2^2 \\
5^2 \tdot 2^2 & 10 \tdot 2^2 \\
6 \tdot 3 \tdot 2^2 \tdot 1 & 6 \tdot 3 \tdot 2^2 \tdot 1 \\
6 \tdot 3^2 \tdot 1^2 & 6^2 \tdot 2 \\
6^2 \tdot 1^2 & 6^2 \tdot 1^2,6^2 \tdot 2 \\
7 \tdot 1^7 & 7 \tdot 1^7,7^2,14 \\
7 \tdot 2 \tdot 1^5 & 14 \\
7 \tdot 2^2 \tdot 1^3 & 14 \\
7 \tdot 2^3 \tdot 1 & 14 \\
7^2 & 7^2,14 \\
8 \tdot 1^6 & 8 \tdot 1^6,8 \tdot 2^3 \\
8 \tdot 2 \tdot 1^4 & 8 \tdot 2^3 \\
8 \tdot 2^2 \tdot 1^2 & 8 \tdot 2^2 \tdot 1^2,8 \tdot 2^3 \\
8 \tdot 4 \tdot 1^2 & 8 \tdot 4 \tdot 1^2 \\
9 \tdot 1^5 & 9 \tdot 1^5 \\
9 \tdot 3 \tdot 1^2 & 9 \tdot 3 \tdot 1^2 \\
10 \tdot 1^4 & 10 \tdot 1^4,10 \tdot 2^2 \\
10 \tdot 2 \tdot 1^2 & 10 \tdot 2 \tdot 1^2,10 \tdot 2^2 \\
11 \tdot 1^3 & 11 \tdot 1^3 \\
12 \tdot 1^2 & 12 \tdot 1^2,12 \tdot 2 \\
13 \tdot 1 & 13 \tdot 1 \\
\hline
\end{array}\verynormalcols

\end{array}
\]

\[
\begin{array}{ccccc}

\verynarrowcols
\begin{array}{|r|r|}
\hline
\tspacer n=15 & \\
\alpha & \beta\text{ and }\gamma\bspacer\\
\hline\tspacer
1^{15} & 1^{15},3^5,5^3,15 \\
2^4 \tdot 1^7 & 2^4 \tdot 1^7 \\
2^5 \tdot 1^5 & 2^5 \tdot 1^5,10 \tdot 5 \\
2^6 \tdot 1^3 & 2^6 \tdot 1^3,6^2 \tdot 3 \\
2^7 \tdot 1 & 2^7 \tdot 1 \\
3 \tdot 1^{12} & 3^5,15 \\
3 \tdot 2^6 & 6^2 \tdot 3 \\
3^2 \tdot 1^9 & 3^5,15 \\
3^3 \tdot 1^6 & 3^3 \tdot 1^6,3^5,15 \\
3^4 \tdot 1^3 & 3^4 \tdot 1^3,3^5,15 \\
3^5 & 3^5,(5^3,15),15 \\
4^2 \tdot 1^7 & 4^2 \tdot 1^7 \\
4^2 \tdot 2^2 \tdot 1^3 & 4^2 \tdot 2^2 \tdot 1^3 \\
4^2 \tdot 2^3 \tdot 1 & 4^2 \tdot 2^3 \tdot 1 \\
4^3 \tdot 1^3 & 4^3 \tdot 1^3,12 \tdot 3 \\
4^3 \tdot 2 \tdot 1 & 4^3 \tdot 2 \tdot 1 \\
4^3 \tdot 3 & 12 \tdot 3 \\
5 \tdot 1^{10} & 5^3,15 \\
5 \tdot 2^5 & 10 \tdot 5 \\
5 \tdot 3 \tdot 1^7 & 15 \\
5 \tdot 3^2 \tdot 1^4 & 15 \\
5 \tdot 3^3 \tdot 1 & 15 \\
5^2 \tdot 1^5 & 5^2 \tdot 1^5,5^3,15 \\
5^2 \tdot 3 \tdot 1^2 & 15 \\
5^3 & 5^3,15 \\
6 \tdot 2^3 \tdot 1^3 & 6^2 \tdot 3 \\
6 \tdot 3 \tdot 2^2 \tdot 1^2 & 6 \tdot 3 \tdot 2^2 \tdot 1^2 \\
6 \tdot 3 \tdot 2^3 & 6^2 \tdot 3 \\
6 \tdot 3^2 \tdot 2 \tdot 1 & 6 \tdot 3^2 \tdot 2 \tdot 1 \\
6^2 \tdot 1^3 & 6^2 \tdot 1^3,6^2 \tdot 3 \\
6^2 \tdot 2 \tdot 1 & 6^2 \tdot 2 \tdot 1 \\
6^2 \tdot 3 & 6^2 \tdot 3 \\
7^2 \tdot 1 & 7^2 \tdot 1 \\
8 \tdot 1^7 & 8 \tdot 1^7 \\
8 \tdot 2^2 \tdot 1^3 & 8 \tdot 2^2 \tdot 1^3 \\
8 \tdot 2^3 \tdot 1 & 8 \tdot 2^3 \tdot 1 \\
8 \tdot 4 \tdot 1^3 & 8 \tdot 4 \tdot 1^3 \\
8 \tdot 4 \tdot 2 \tdot 1 & 8 \tdot 4 \tdot 2 \tdot 1 \\
9 \tdot 1^6 & 9 \tdot 1^6,9 \tdot 3^2 \\
9 \tdot 3 \tdot 1^3 & 9 \tdot 3 \tdot 1^3,9 \tdot 3^2 \\
9 \tdot 3^2 & 9 \tdot 3^2 \\
10 \tdot 1^5 & 10 \tdot 1^5,10 \tdot 5 \\
10 \tdot 2^2 \tdot 1 & 10 \tdot 2^2 \tdot 1 \\
10 \tdot 5 & 10 \tdot 5 \\
11 \tdot 1^4 & 11 \tdot 1^4 \\
12 \tdot 1^3 & 12 \tdot 1^3,12 \tdot 3 \\
12 \tdot 2 \tdot 1 & 12 \tdot 2 \tdot 1 \\
12 \tdot 3 & 12 \tdot 3 \\
13 \tdot 1^2 & 13 \tdot 1^2 \\
14 \tdot 1 & 14 \tdot 1 \\
15 & 15 \\
\hline
\end{array}\verynormalcols

&&

\verynarrowcols\begin{array}{|r|r|}
\hline
\tspacer n=16 & \\
\alpha & \beta\text{ and }\gamma\bspacer\\
\hline\tspacer
1^{16} & 1^{16},2^8,4^4,8^2,16 \\
2 \tdot 1^{14} & 2^8,4^4,8^2,16 \\
2^2 \tdot 1^{12} & 2^8,4^4,8^2,16 \\
2^3 \tdot 1^{10} & 2^8,4^4,8^2,16 \\
2^4 \tdot 1^8 & 2^4 \tdot 1^8,2^8,4^4,8^2,16 \\
2^5 \tdot 1^6 & 2^5 \tdot 1^6,2^8,4^4,8^2,16 \\
2^6 \tdot 1^4 & 2^6 \tdot 1^4,2^8,4^4,8^2,16 \\
2^7 \tdot 1^2 & 2^7 \tdot 1^2,2^8,4^4,8^2,16 \\
2^8 & 2^8,4^4,8^2,16 \\
3^3 \tdot 1^7 & 3^3 \tdot 1^7 \\
3^4 \tdot 1^4 & 3^4 \tdot 1^4,6^2 \tdot 2^2,12 \tdot 4 \\
3^4 \tdot 2 \tdot 1^2 & 6^2 \tdot 2^2,12 \tdot 4 \\
3^4 \tdot 2^2 & 6^2 \tdot 2^2,12 \tdot 4 \\
3^5 \tdot 1 & 3^5 \tdot 1 \\
4 \tdot 1^{12} & 4^4,8^2,16 \\
4 \tdot 2 \tdot 1^{10} & 4^4,8^2,16 \\
4 \tdot 2^2 \tdot 1^8 & 4^4,8^2,16 \\
4 \tdot 2^3 \tdot 1^6 & 4^4,8^2,16 \\
4 \tdot 2^4 \tdot 1^4 & 4^4,8^2,16 \\
4 \tdot 2^5 \tdot 1^2 & 4^4,8^2,16 \\
4 \tdot 2^6 & 4^4,8^2,16 \\
4^2 \tdot 1^8 & 4^2 \tdot 1^8,4^2 \tdot 2^4,4^4,8^2,16 \\
4^2 \tdot 2 \tdot 1^6 & 4^2 \tdot 2^4,4^4,8^2,16 \\
4^2 \tdot 2^2 \tdot 1^4 & 4^2 \tdot 2^2 \tdot 1^4,4^2 \tdot 2^4,4^4,8^2,16 \\
4^2 \tdot 2^3 \tdot 1^2 & 4^2 \tdot 2^3 \tdot 1^2,4^2 \tdot 2^4,4^4,8^2,16 \\
4^2 \tdot 2^4 & 4^2 \tdot 2^4,4^4,8^2,16 \\
4^3 \tdot 1^4 & 4^3 \tdot 1^4,4^3 \tdot 2^2,4^4,8^2,16 \\
4^3 \tdot 2 \tdot 1^2 & 4^3 \tdot 2 \tdot 1^2,4^3 \tdot 2^2,4^4,8^2,16 \\
4^3 \tdot 2^2 & 4^3 \tdot 2^2,4^4,8^2,16 \\
4^4 & 4^4,8^2,16 \\
5^2 \tdot 1^6 & 5^2 \tdot 1^6,10 \tdot 2^3 \\
5^2 \tdot 2 \tdot 1^4 & 10 \tdot 2^3 \\
5^2 \tdot 2^2 \tdot 1^2 & 10 \tdot 2^3 \\
5^3 \tdot 1 & 5^3 \tdot 1 \\
6 \tdot 3 \tdot 2^2 \tdot 1^3 & 6 \tdot 3 \tdot 2^2 \tdot 1^3 \\
6 \tdot 3 \tdot 2^3 \tdot 1 & 6 \tdot 3 \tdot 2^3 \tdot 1 \\
6 \tdot 3^2 \tdot 1^4 & 6^2 \tdot 2^2,12 \tdot 4 \\
6 \tdot 3^2 \tdot 2 \tdot 1^2 & 6 \tdot 3^2 \tdot 2 \tdot 1^2,6^2 \tdot 2^2,12 \tdot 4 \\
6 \tdot 3^2 \tdot 2^2 & 6 \tdot 3^2 \tdot 2^2,6^2 \tdot 2^2,12 \tdot 4 \\
6^2 \tdot 1^4 & 6^2 \tdot 1^4,6^2 \tdot 2^2,12 \tdot 4 \\
6^2 \tdot 2 \tdot 1^2 & 6^2 \tdot 2 \tdot 1^2,6^2 \tdot 2^2,12 \tdot 4 \\
6^2 \tdot 2^2 & 6^2 \tdot 2^2,12 \tdot 4 \\
6^2 \tdot 3 \tdot 1 & 6^2 \tdot 3 \tdot 1 \\
7^2 \tdot 1^2 & 7^2 \tdot 1^2,14 \tdot 2 \\
8 \tdot 1^8 & 8 \tdot 1^8,8 \tdot 2^4,8 \tdot 4^2,8^2,16 \\
8 \tdot 2 \tdot 1^6 & 8 \tdot 2^4,8 \tdot 4^2,8^2,16 \\
8 \tdot 2^2 \tdot 1^4 & 8 \tdot 2^2 \tdot 1^4,8 \tdot 2^4,8 \tdot 4^2,8^2,16 \\
8 \tdot 2^3 \tdot 1^2 & 8 \tdot 2^3 \tdot 1^2,8 \tdot 2^4,8 \tdot 4^2,8^2,16 \\
8 \tdot 2^4 & 8 \tdot 2^4,8 \tdot 4^2,8^2,16 \\
8 \tdot 4 \tdot 1^4 & 8 \tdot 4 \tdot 1^4,8 \tdot 4 \tdot 2^2,8 \tdot 4^2,8^2,16 \\
8 \tdot 4 \tdot 2 \tdot 1^2 & 8 \tdot 4 \tdot 2 \tdot 1^2,8 \tdot 4 \tdot 2^2,8 \tdot 4^2,8^2,16 \\
8 \tdot 4 \tdot 2^2 & 8 \tdot 4 \tdot 2^2,8 \tdot 4^2,8^2,16 \\
8 \tdot 4^2 & 8 \tdot 4^2,8^2,16 \\
8^2 & 8^2,16 \\
9 \tdot 1^7 & 9 \tdot 1^7 \\
9 \tdot 3^2 \tdot 1 & 9 \tdot 3^2 \tdot 1 \\
10 \tdot 1^6 & 10 \tdot 1^6,10 \tdot 2^3 \\
10 \tdot 2 \tdot 1^4 & 10 \tdot 2^3 \\
10 \tdot 2^2 \tdot 1^2 & 10 \tdot 2^2 \tdot 1^2,10 \tdot 2^3 \\
10 \tdot 5 \tdot 1 & 10 \tdot 5 \tdot 1 \\
11 \tdot 1^5 & 11 \tdot 1^5 \\
12 \tdot 1^4 & 12 \tdot 1^4,12 \tdot 2^2,12 \tdot 4 \\
12 \tdot 2 \tdot 1^2 & 12 \tdot 2 \tdot 1^2,12 \tdot 2^2,12 \tdot 4 \\
12 \tdot 2^2 & 12 \tdot 2^2,12 \tdot 4 \\
12 \tdot 3 \tdot 1 & 12 \tdot 3 \tdot 1 \\
13 \tdot 1^3 & 13 \tdot 1^3 \\
14 \tdot 1^2 & 14 \tdot 1^2,14 \tdot 2 \\
15 \tdot 1 & 15 \tdot 1 \\
\hline
\end{array}
\verynormalcols

&&

\verynarrowcols\begin{array}{|r|r|}
\hline
\tspacer n=17 & \\
\alpha & \beta\text{ and }\gamma\bspacer\\
\hline\tspacer
1^{17} & 1^{17},17 \\
2^5 \tdot 1^7 & 2^5 \tdot 1^7 \\
2^6 \tdot 1^5 & 2^6 \tdot 1^5 \\
2^7 \tdot 1^3 & 2^7 \tdot 1^3 \\
2^8 \tdot 1 & 2^8 \tdot 1 \\
3^3 \tdot 1^8 & 3^3 \tdot 1^8 \\
3^4 \tdot 1^5 & 3^4 \tdot 1^5 \\
3^5 \tdot 1^2 & 3^5 \tdot 1^2 \\
4^3 \tdot 1^5 & 4^3 \tdot 1^5 \\
4^3 \tdot 2^2 \tdot 1 & 4^3 \tdot 2^2 \tdot 1 \\
4^4 \tdot 1 & 4^4 \tdot 1 \\
5^2 \tdot 1^7 & 5^2 \tdot 1^7 \\
5^3 \tdot 1^2 & 5^3 \tdot 1^2 \\
6 \tdot 3 \tdot 2^3 \tdot 1^2 & 6 \tdot 3 \tdot 2^3 \tdot 1^2 \\
6 \tdot 3^2 \tdot 2^2 \tdot 1 & 6 \tdot 3^2 \tdot 2^2 \tdot 1 \\
6^2 \tdot 1^5 & 6^2 \tdot 1^5 \\
6^2 \tdot 2^2 \tdot 1 & 6^2 \tdot 2^2 \tdot 1 \\
6^2 \tdot 3 \tdot 1^2 & 6^2 \tdot 3 \tdot 1^2 \\
7^2 \tdot 1^3 & 7^2 \tdot 1^3 \\
8^2 \tdot 1 & 8^2 \tdot 1 \\
9 \tdot 1^8 & 9 \tdot 1^8 \\
9 \tdot 3^2 \tdot 1^2 & 9 \tdot 3^2 \tdot 1^2 \\
10 \tdot 1^7 & 10 \tdot 1^7 \\
10 \tdot 2^2 \tdot 1^3 & 10 \tdot 2^2 \tdot 1^3 \\
10 \tdot 2^3 \tdot 1 & 10 \tdot 2^3 \tdot 1 \\
10 \tdot 5 \tdot 1^2 & 10 \tdot 5 \tdot 1^2 \\
11 \tdot 1^6 & 11 \tdot 1^6 \\
12 \tdot 1^5 & 12 \tdot 1^5 \\
12 \tdot 2^2 \tdot 1 & 12 \tdot 2^2 \tdot 1 \\
12 \tdot 3 \tdot 1^2 & 12 \tdot 3 \tdot 1^2 \\
12 \tdot 4 \tdot 1 & 12 \tdot 4 \tdot 1 \\
13 \tdot 1^4 & 13 \tdot 1^4 \\
14 \tdot 1^3 & 14 \tdot 1^3 \\
14 \tdot 2 \tdot 1 & 14 \tdot 2 \tdot 1 \\
15 \tdot 1^2 & 15 \tdot 1^2 \\
16 \tdot 1 & 16 \tdot 1 \\
17 & 17 \\
\hline
\end{array}\verynormalcols
\end{array}
\]


\begin{thebibliography}{10}

\bibitem{Bruck} R.\,H.~Bruck,
A Survey of Binary Systems, Springer, Berlin, 1971.

\bibitem{BP} R.\,H.~Bruck and L.\,J.~Paige,
Loops whose inner mappings are automorphisms,
\textit{Ann.~of Math.} (2) \textbf{63} (1956), 308--323.

\bibitem{BrowningStonesWanless}
J.~Browning, D.\,S.~Stones and I.\,M.~Wanless,
Bounds on the number of autotopisms and subsquares of a Latin square.  Submitted.

\bibitem{BrowningVojtechovskyWanless2009}
J.~Browning, P.~Vojt\v{e}chovsk\'{y} and I.\,M.~Wanless,
Overlapping Latin subsquares and full products,
\textit{Comment. Math. Univ. Carolin.} \textbf{51} (2010), 175--184.

\bibitem{BryantBuchananWanless2009}
D.~Bryant, M.~Buchanan and I.\,M.~Wanless,
The spectrum for quasigroups with cyclic automorphisms
and additional symmetries,
{\em Discrete Math.}, {\bf304} (2009), 821--833.

\bibitem{Cameron1999}
{P.\,J.~Cameron}, {\em Permutation Groups},
Cambridge University Press, 1999.

\bibitem{CavenaghStones2010a}
{N.\,J.~Cavenagh and D.\,S.~Stones},
Near-automorphisms of Latin squares.
{\it J. Combin. Des.}, {\bf19} (2011), 365–-377.

\bibitem{CoRo} C.\,J.~Colbourn and A.~Rosa,
{\it Triple systems\/},
Clarendon, Oxford, 1999.

\bibitem{CDK} P.~Cs\"org\H{o}, A.~Dr\'apal and M.~Kinyon,
\emph{Buchsteiner loops},
Internat. J. Algebra Comput. \textbf{19} (2009), 1049--1088.

\bibitem{Doro} S.~Doro,
Simple Moufang loops,
{\em Math. Proc. Cambridge Philos. Soc.} \textbf{83} (1978), 377--392.

\bibitem{DJ} A.~Dr\'apal and P.~Jedli\v{c}ka,
On loop identities that can be obtained by a nuclear identification,
{\em European J.~Combin}, \textbf{31} (2010), 1907--1923.

\bibitem{Dr} A.\,A.~Drisko,
Loops with transitive automorphisms,
{\em J. Algebra} \textbf{184} (1996), 213--229.

\bibitem{Drisko1997b}
A.\,A.~Drisko,
Loops of order $p^n+1$ with transitive automorphism groups,
{\em Adv. Math.}, 128 (1997), 36--39.

\bibitem{drisko2} A.\,A.~Drisko,
Proof of the Alon-Tarsi conjecture for $n=2^rp$,
{\it Electron. J. Combin.} {\bf5} (1998) \#R28, 5 pp.

\bibitem{EW}
J.~Egan and I.\,M.~Wanless,
Latin squares with no small odd $k$-plexes,
{\it J.\ Combin.\ Designs\/} {\bf16} (2008), 477--492.

\bibitem{Euler1782}
L.~Euler,
Recherches sur une nouvelle esp\`{e}ce de quarr\'{e}s magiques,
{\em Verh. Zeeuwsch. Gennot. Weten. Vliss.}, {\bf9} (1782), 85--239.
Enestr\"{o}m E530, Opera Omnia OI7, 291--392.

\bibitem{Falcon2009}
{R.\,M.~Falc\'{o}n},
Cycle structures of autotopisms of the Latin squares of order up to $11$,
{\em Ars Combin.},
to appear.

\bibitem{FalconMartinMorales2007}
{R.\,M.~Falc\'{o}n and J.~Mart\'{i}n-Morales}, Gr\"{o}bner bases and
  the number of Latin squares related to autotopisms of order $\leq 7$,
{\em J. Symbolic Comput.}, {\bf42} (2007), 1142--1154.

\bibitem{Ganfornina2006b}
R.\,M.\,F.~Ganfornina, {\em Decomposition of principal autotopisms into
  triples of a Latin square}, in Book of abstracts of the Tenth Meeting on
  Computer Algebra and Applications, Seville, Spain, 7-9 Sep 2006, 95--98.
(Author also known as R.\,M.~Falc\'{o}n.)

\bibitem{Ganfornina2006}
R.\,M.\,F.~Ganfornina, {\em Latin squares
  associated to principal autotopisms of long cycles. Application in
  cryptography}, in Proceedings of Transgressive Computing, 
Granada, Spain, 24--26 April 2006, 213--230.
\url{http://www.orcca.on.ca/conferences/tc2006/TC2006-Proceedings.pdf}

\bibitem{GAP}
{\em {GAP} -- Groups, algorithms, programming -- A
  system for computational discrete algebra}.
\url{http://www.gap-system.org/}.

\bibitem{GR} E.\,G.~Goodaire and D.\,A.~Robinson,
A class of loops which are isomorphic to all loop isotopes,
{\em Canad. J. Math.} \textbf{34} (1982), 662--672.

\bibitem{Hor} M.\,V.~Horo\v{s}evski\u{\i},
Automorphisms of finite groups,
{\em Math. Sb. (N.S.)} \textbf{93} (\textbf{135}) (1974). 576--587.

\bibitem{HulpkeKaskiOstergard}
A.~Hulpke, P.~Kaski and P.\,R.\,J.~\"Osterg{\aa}rd,
The number of Latin squares of order $11$,
{\em Math. Comp.}, 80 (2011) 1197--1219.

\bibitem{JKV} P.~Jedli\v{c}ka, M.\,K.~Kinyon and P.~Vojt\v{e}chovsk\'{y},
The structure of commutative automorphic loops,
{\em Trans. Amer. Math. Soc.} \textbf{363} (2011), 365--384.

\bibitem{JKV2} P.~Jedli\v{c}ka, M.\,K.~Kinyon and P.~Vojt\v{e}chovsk\'{y},
Constructions of commutative automorphic loops,
{\em Comm. Algebra} \textbf{38} (2010), 3243--3267.

\bibitem{KerbySmith2009}
B.~Kerby and J.\,D.\,H.~Smith,
Quasigroup automorphisms and symmetric group characters,
{\em Comment. Math. Univ. Carol.}, {\bf51} (2010), 279--286.

\bibitem{KerbySmith2010}
B.\,L.~Kerby and J.\,D.\,H.~Smith,
Quasigroup automorphisms and the Norton-Stein complex,
{\em Proc. Amer. Math. Soc.} \textbf{138} (2010), 3079--3088.

\bibitem{KK} M.\,K.~Kinyon and K.~Kunen,
The structure of extra loops,
{\em Quasigroups Related Systems} \textbf{12} (2004), 39--60.

\bibitem{Laywine1981}
C.~Laywine,
An expression for the number of equivalence classes of
Latin squares under row and column permutations,
{\em J. Combin. Theory Ser. A}, {\bf30} (1981), 317--320.

\bibitem{LM85}
C.~Laywine and G.\,L.~Mullen,
Latin cubes and hypercubes of prime order,
{\em Fibonacci Quart.} {\bf23} (1985), 139--145.

\bibitem{Mears2009}
C.~Mears,
{\em Automatic Symmetry Detection and Dynamic Symmetry Breaking for Constraint Programming}, PhD thesis, Monash University, 2009.\\
\url{http://www.csse.monash.edu.au/~cmears/files/thesis.pdf}.

\bibitem{Moufang}
R.~Moufang,
Zur Struktur von Alternativk\"orpern,
{\em Math. Ann.} \textbf{110} (1935), 416--430.

\bibitem{ninf}
B.\,M.~Maenhaut, I.\,M.~Wanless and B.\,S.~Webb,
Subsquare-free Latin squares of odd order,
{\it European J. Combin.\/}, {\bf28} (2007) 322--336.

\bibitem{nauty}
{B.~D.~McKay}, {\em nauty -- Graph isomorphic software},\\
\url{http://cs.anu.edu.au/~bdm/nauty/}.

\bibitem{McKayMeynertMyrvold2007}
{B.\,D.~McKay, A.~Meynert and W.~Myrvold},
Small Latin squares, quasigroups and loops,
{\em J. Combin. Des.}, {\bf15} (2007), 98--119.

\bibitem{McKayWanless2005}
{B.\,D.~McKay and I.\,M.~Wanless},
On the number of Latin squares,
{\em Ann. Comb.}, {\bf9} (2005), 335--344.

\bibitem{Minc1978}
H.~Minc, {\em Permanents}, Addison-Wesley, 1978.

\bibitem{LOOPS}
G.\,P.~Nagy and P.~Vojt\v{e}chovsk\'{y},
{\em {LOOPS} -- Computing with quasigroups and loops in {GAP}}.
\url{http://www.math.du.edu/loops/}.

\bibitem{NagyVojtechovsky2007}
G.\,P.~Nagy and P.~Vojt\v{e}chovsk\'{y},
Computing with small  quasigroups and loops,
{\em Quasigroups Related Systems}, {\bf15} (2007), 77--94.

\bibitem{Sade1968}
A.\,A.~Sade,
Autotopies des quasigroupes et des syst\`{e}mes associatifs,
{\em Arch. Math. (Brno)}, {\bf4} (1968), 1--23.

\bibitem{Stones2009d}
D.\,S.~Stones,
The parity of the number of quasigroups.
{\it Discrete Math.}, {\bf310} (2010), 3033--3039.

\bibitem{Stones2009}
D.\,S.~Stones,
{\em On the Number of Latin Rectangles}, PhD thesis, Monash University, 2010.
\url{http://arrow.monash.edu.au/hdl/1959.1/167114}.

\bibitem{Sto10}
D.\,S.~Stones,
The many formulae for the number of Latin rectangles,
{\em Electron. J. Combin.}, {\bf17} (2010) A1.

\bibitem{StonesWanless2009c}
D.\,S.~Stones and I.\,M.~Wanless,
Compound orthomorphisms of the cyclic group,
{\em Finite Fields Appl.}, {\bf16} (2010), 277--289.

\bibitem{SW}
D.\,S.~Stones and I.\,M.~Wanless,
Divisors of the number of Latin rectangles,
{\em J. Combin. Th. Ser. A}, {\bf117} (2010), 204--215.

\bibitem{StonesWanless2009b}
D.\,S.~Stones and I.\,M.~Wanless,
A congruence connecting Latin rectangles and partial orthomorphisms,
{\em Ann. Comb.}, to appear.

\bibitem{StonesWanless2009d}
D.\,S.~Stones and I.\,M.~Wanless,
How not to prove the Alon-Tarsi Conjecture,
\newblock Nagoya Math. J., to appear.

\bibitem{Wanless2004}
{I.\,M.~Wanless}, Diagonally cyclic Latin squares,
{\em European J. Combin.}, {\bf25} (2004), 393--413.

\bibitem{cyclatom}
I.\,M.~Wanless,
Atomic Latin squares based on cyclotomic orthomorphisms,
{\it Electron.\ J.\ Combin.\/}, {\bf12} (2005), R22.

\bibitem{WWWW}
I.\,M.~Wanless, author's homepage,\\
\url{http://users.monash.edu.au/~iwanless/data/autotopisms}.

\bibitem{WI}
I.\,M.~Wanless and E.\,C.~Ihrig, Symmetries that
Latin Squares Inherit from $1$-Factor\-izations,
{\it J.\ Combin.\ Designs\/}, {\bf13} (2005), 157--172.

\bibitem{WW}
I.\,M.~Wanless and B.\,S.~Webb,
The existence of Latin squares without orthogonal mates,
{\it Des.\ Codes Cryptogr.}, {\bf40} (2006), 131--135.

\end{thebibliography}
\end{document}